\documentclass[a4paper,reqno,11pt]{amsart}

\usepackage{amsmath,amsfonts,amsthm,amssymb,euler,newpxtext,latexsym,mathrsfs,mathtools,enumerate}

\usepackage[dvipsnames]{xcolor}

\usepackage{tikz-cd,tikz-3dplot}

\usepackage{eqparbox}
\newcommand{\eqmathbox}[2][N]{\eqmakebox[#1]{$\displaystyle#2$}}

\usepackage[colorlinks=true,linkcolor=blue,citecolor=red,urlcolor=teal,bookmarksdepth=subsection,final]{hyperref}

\DeclareMathOperator{\conv}{\mathrm{conv}}

\DeclareMathOperator{\diam}{\mathrm{diam}}
\DeclareMathOperator{\ev}{\mathrm{ev}}
\DeclareMathOperator{\id}{\mathrm{id}}

\DeclareMathOperator{\aff}{\mathrm{Aff}}
\DeclareMathOperator{\rank}{\mathrm{rank}}
\DeclareMathOperator{\prob}{\mathcal{M}}
\DeclareMathOperator{\lip}{\mathrm{Lip}_1}
\DeclareMathOperator{\lipo}{\mathrm{Lip}}
\DeclareMathOperator{\emb}{\mathrm{Emb}}
\DeclareMathOperator{\Span}{\mathrm{span}}

\DeclareMathOperator{\sgn}{sgn}
\DeclareMathOperator{\Var}{Var}
\DeclareMathOperator{\Angle}{\widetilde\measuredangle}

\newcommand{\nn}{\mathbb{N}}
\newcommand{\zz}{\mathbb{Z}}

\newcommand{\rr}{\mathbb{R}}
\newcommand{\cc}{\mathbb{C}}
\newcommand{\pp}{\mathbb{P}}
\newcommand{\ee}{\mathbb{E}}
\newcommand{\cs}{\mathrm{C}^*}
\newcommand{\js}{\mathcal{Z}}
\newcommand{\ws}{\text{weak}^*}
\newcommand{\eps}{\varepsilon}
\newcommand{\lp}{d_{LP}}
\newcommand{\mm}{\mathbb{M}}
\newcommand{\sphere}{\mathbb{S}}
\newcommand{\ball}{\mathbb{B}}
\newcommand{\lyre}{\mathbf{Lyre}}
\newcommand{\alex}{\mathbf{Alex}}
\newcommand{\Euc}{\mathrm{Euc}}
\newcommand{\fp}{\mathfrak{p}}
\newcommand{\fq}{\mathfrak{q}}
\newcommand{\ssubset}{\subset\joinrel\subset}

\makeatletter
\@namedef{subjclassname@2020}{%
  \textup{2020} Mathematics Subject Classification}
\makeatother

\newenvironment{items}{\begin{list} {$\boldsymbol{\cdot}$} {\setlength{\leftmargin}{0.5cm}}}{\end{list}}

\newtheorem {theorem}{Theorem}[section]
\newtheorem {lemma}[theorem]{Lemma}
\newtheorem {proposition}[theorem]{Proposition}
\newtheorem {corollary}[theorem]{Corollary}

\newtheorem {thm}{Theorem}

\theoremstyle {definition}

\newtheorem {example}[theorem]{Example}
\newtheorem {definition}[theorem]{Definition}
\newtheorem {remark}[theorem]{Remark}
\newtheorem {question}[theorem]{Question}
\newtheorem* {notation}{Notation}

\numberwithin{equation}{section}

\title{Tracially lyriform $\mathrm{C}^*$-algebras}

\author[B.~Jacelon]{Bhishan Jacelon}
\address[B.~Jacelon]{
Institute of Mathematics of the Czech Academy of Sciences\\ \v{Z}itn\'{a} 25\\115 67 Prague 1\\Czech Republic}
\email{\href{mailto:jacelon@math.cas.cz}{jacelon@math.cas.cz}}


\begin{document}

\begin{abstract} 
Quantum metric Choquet simplices are special kinds of compact quantum metric spaces designed for distance measurement in and around the category of stably finite Elliott-classifiable $\mathrm{C}^*$-algebras. The primary objective of this article is to introduce versions of these structures for which the associated tracial metrics need not be induced by Lipschitz seminorms and may induce strictly stronger topologies than the weak$^*$-topology. The resulting category of `tracially lyriform $\mathrm{C}^*$-algebras' behaves well with respect to sequential inductive limits and accommodates the full family of classical $\mathfrak{p}$-Wasserstein metrics on probability spaces, including $\mathfrak{p}=\infty$. Examples of projectionless, classifiable tracial Wasserstein spaces are built as noncommutative spaces of observables of certain compact length spaces, including: fractals like the Sierpi\'nski gasket, the Sierpi\'nski carpet and the Menger sponge; finite-dimensional Alexandrov spaces with two-sided curvature bounds; and metric spaces like simplicial spheres and balls that are Lipschitz equivalent to Riemannian. These simplicial structures are used as building blocks that furnish arbitrary simple inductive limits of prime dimension drop algebras with tracial lyriform structure. Appealing to optimal transport theory, we study the geometry and statistics of the spaces of embeddings of these building blocks and their limits into suitable classifiable $\mathrm{C}^*$-algebras like the Jiang--Su algebra $\mathcal{Z}$ or the universal UHF algebra $\mathcal{Q}$. 
\end{abstract}

\maketitle

\section*{Introduction} \label{section:intro}

This article is part of an ongoing investigation \cite{Jacelon:2014aa,Jacelon:2021wa,Jacelon:2021vc,Jacelon:2022wr,Jacelon:2024aa} into the metric geometry of trace spaces. The main goals of the present work are to:
\begin{enumerate}[(I)]
\item \label{goal:a} develop a unified theory of tracial metrics that need not be induced by Lipschitz seminorms and whose topologies may be finer than the $\ws$-topology;
\item \label{goal:c} explore the role of metric curvature in tracial geometry;
\item \label{goal:b} use optimal transport theory to describe the geometry and statistics of embedding spaces associated with classifiable $\cs$-algebras.
\end{enumerate}

The title of the article is our response to Goal~\ref{goal:a}. Consisting of a $\cs$-algebra $A$ whose trace space $T(A)$ is equipped with a metric $\rho$ that is reasonably compatible with the structure of $T(A)$ as a $\ws$-compact, convex set, \emph{tracially lyriform $\cs$-algebras} are generalisations of the quantum metric Choquet simplices introduced in \cite{Jacelon:2024aa}. As highlighted above, the difference here is that we no longer require $\rho$ to be $\ws$-continuous or to be induced by a seminorm. For example, continuity of $\rho$ is weakened to lower semicontinuity, and convexity to effective quasiconvexity. See Definition~\ref{def:lyriform}. This more flexible setup grants us access not just to the full family of classical Wasserstein metrics (defined in Section~\ref{section:wass} and discussed below), but also to a category $\lyre$ whose morphisms are tracially nonexpansive $^*$-monomorphisms and which admits limits of suitably bounded inductive sequences (see Section~\ref{section:limits}). 

As in \cite{Jacelon:2024aa}, our most easily accessible examples are associated with Bauer simplices, or equivalently, spaces of Borel probability measures $\prob(X)$ over compact metric spaces $(X,\rho)$. In the context of \cite{Jacelon:2024aa}, there is a canonical way of extending the metric on $X\cong\partial_e(\prob(X))$ to $\prob(X)$, namely via the \emph{Monge--Kantorovich} or \emph{$1$-Wasserstein} formula
\[
W_1(\mu,\nu) = \sup\left\{\left|\int_Xf\,d\mu - \int_Xf\,d\nu\right| \mid f\colon X\to\rr \:\text{is $1$-Lipschitz}\right\}.
\]
While noncommutative geometers may be most familiar with $W_1$ (via work of Connes \cite{Connes:1989aa} and Rieffel \cite{Rieffel:1999aa}), the theory that we develop in the present article is designed to accommodate not just $W_1$ but also $W_\fp=W_{\fp,\rho}$, for any $\fp\in[1,\infty]$. These \emph{$\fp$-Wasserstein} metrics \eqref{eqn:wp} play important roles in geometric measure theory, and desirable geometric features in classical and noncommutative probability spaces are often tied to other values of $\fp$ than $\fp=1$. Famously, Biane and Voiculescu \cite{Biane:2001aa}, Sturm \cite{Sturm:2006aa,Sturm:2006ab}, Lott and Villani \cite{Lott:2009aa} and Carlen and Maas \cite{Carlen:2014aa} focus on $W_2$ in their geometric analyses of Wasserstein spaces. On the other hand, the most relevant distance for the geometry of embedding spaces of stably finite classifiable $\cs$-algebras is $W_\infty$ \eqref{eqn:winf1}. Unlike the other Wasserstein metrics and the \emph{L\'{e}vy--Prokhorov metric} $d_{LP}$ defined in \eqref{eqn:lp1}, all of which are dominated by $W_\infty$, the distinguishing feature of the $\infty$-Wasserstein distance is that its topology is much finer than the $\ws$-topology (see Remark~\ref{rem:compact}). In fact, we will come to see the $\infty$-Wasserstein space $(\prob(X),W_{\infty,\rho})$ as a special instance of a $\lyre$-limit (see Example~\ref{ex:wlim}).

One point that should be emphasised is that by `embedding space' what we really mean is the set of \emph{approximate unitary equivalence classes} of unital $^*$-monomorphisms $A\to B$, which is what are classified in  \cite{Carrion:wz}. The domains $A$ considered in this article are either subhomogeneous or `classifiable', which we use as an abbreviation for `unital, simple, separable, nuclear, $\js$-stable and satisfying the universal coefficient theorem (UCT)'. As usual, $\js$ denotes the Jiang--Su algebra, constructed in \cite{Jiang:1999hb}. Throughout, we only consider \emph{unital} $\cs$-algebras. The codomains $B$ are also required to be classifiable, and in addition be either real rank zero or monotracial, assuming in the latter case that $K_1(A)=0$ as well (see Theorem~\ref{thm:orbits} and its corollaries). Developing a geometric picture beyond these restricted codomains would necessarily involve the Hausdorffised unitary algebraic $K_1$ component of the classifying invariant (see \cite[Section~2.2]{Carrion:wz}). Progress in that direction would be very welcome.

Another point that should be clarified is that the Wasserstein distances that are the focus of this article are of the classical variety. The \emph{free} Wasserstein metrics and related distances appearing, for instance, in \cite{Biane:2001aa,Ioana:2024aa,Anshu:2025aa}, are important sources of fine topologies on trace spaces and are certainly desirable as examples of tracial lyriform structure. In fact, it is shown in \cite{Anshu:2025aa} that Wang's universal quantum permutation group $C(S_n^+)$ is indeed tracially lyriform when equipped with the \emph{free Hamming distances} introduced in that article (see Theorem~\ref{thm:free}).

In order to develop a collection of examples that we call \emph{tracial Wasserstein spaces}, and in service of Goal~\ref{goal:c}, we take a new look at \cite[Theorem 4.4]{Jacelon:2022wr}. This theorem explains how to construct projectionless, classifiable $\cs$-algebraic quantum metric Bauer simplices $\js_X$ observing compact, connected Riemannian manifolds $(X,\rho)$. In more familiar terms, $\js_X$ is a simple, separable, unital $\cs$-algebra that has finite nuclear dimension and satisfies the UCT, whose extreme tracial boundary $\partial_e(T(\js_X))$ is homeomorphic to $X$, and in which tracially $\rho$-Lipschitz elements are dense. The metric $\rho$ is the \emph{intrinsic} metric associated with the Riemannian structure, meaning that the distance between two points is defined to be the length of a shortest path between them. In Section~\ref{section:qmcs}, we clarify that our use of the Riemannian structure is indeed purely metric, and we adapt the theorem to include certain compact \emph{Alexandrov spaces}. These are intrinsic length spaces whose metric curvature is bounded below (see Definitions~~\ref{def:length},~~\ref{def:geodesic}~and~\ref{def:alexandrov}), and are noteworthy for providing Gromov--Hausdorff limits of suitably bounded sequences of Riemannian manifolds (see Example~\ref{example:riemannianlimits}). Under the additional imposition of \emph{upper} curvature bounds, we get the following generalisation of \cite[Theorem 4.4]{Jacelon:2022wr} (see Theorem~\ref{thm:alexandrov} and Corollary~\ref{cor:qwass}).
 
\begin{thm} \label{thm:a}
Let $(X,\rho)$ be a compact, finite-dimensional Alexandrov space with lower and upper curvature bounds, and let $\fp\in[1,\infty]$. Suppose that either the boundary of $X$ is empty or that $X$ is contractible and the upper curvature bound is nonpositive. Then, there exists a classifiable, projectionless tracial $\fp$-Wasserstein space $(\js_X,W_{\fp})$ observing $(X,\rho)$.
\end{thm}

Actually, we can observe more than Alexandrov spaces. Our more general version of \cite[Theorem 4.4]{Jacelon:2022wr} also covers fractals like the Sierpi\'nski gasket, the Sierpi\'nski carpet and the Menger sponge (see Theorem~\ref{thm:models} and Example~\ref{ex:fractals}). Furthermore, our results in the Bauer setting are not exclusive to $\fp$-Wasserstein metrics. In Section~\ref{section:wass}, we identify abstract metric properties shared by $\{W_\fp\}_{p\in[1,\infty)}$ and $d_{LP}$, yielding the notion of a \emph{quasi-Wasserstein space} (see Definition~\ref{def:quasiwass}). By design, these structures are conducive to tracial Gromov--Hausdorff analysis in $\lyre$. In particular, the \emph{quantum intertwining gap} $\gamma_q$ (introduced in \cite{Jacelon:2024aa} and recalled in Definition~\ref{def:quig}) is a complete, separable quasimetric on the set of tracial isomorphism classes of quasi-Wasserstein spaces (see Theorem~\ref{thm:quigwass}). Restricting our attention to quasi-Wasserstein spaces that have prescribed $K$-theory and are classifiable and \emph{$K$-connected} (meaning that there is a unique state on the ordered $K_0$-group), classification \cite{Carrion:wz} implies that $\gamma_q$ is a quasimetric on $\lyre$-isomorphism classes (see Corollary~\ref{cor:quigwass}). In other words, the quantum intertwining gap between such objects is zero precisely when there is an isometric isomorphism of trace spaces that is induced by a $^*$-isomorphism between the observing $\cs$-algebras. Further assuming, as in Corollary~\ref{cor:contractible}, that these classifiable quasi-Wasserstein spaces have point-like ordered $K$-theory \eqref{eqn:pointlike} and are observing suitably bounded, contractible Alexandrov spaces of nonpositive curvature, $\gamma_q$ is in fact complete at the $\cs$-algebraic level.
 
\begin{thm} \label{thm:b}
For $n\in\nn$, $D,v>0$ and $k\le\kappa\le0$, let $\alex_0^n(D,k,\kappa,v)$ denote the collection of $n$-dimensional compact, contractible Alexandrov spaces $(X,\rho)$ of diameter at most $D$, curvature at least $k$ and at most $\kappa$, and such that the $n$-dimensional Hausdorff measure of $X$ is at least $v$. Then, for every $\fp\in[1,\infty)$, $\gamma_q$ is a complete $2$-quasimetric on the set of $\lyre$-isomorphism classes of classifiable tracial $\fp$-Wasserstein spaces $(A_X,W_{\fp,\rho})$ with point-like ordered $K$-theory observing objects $(X,\rho)$ in $\alex_0^n(D,k,\kappa,v)$.
\end{thm}

Like Theorem~\ref{thm:a}, whose proof relies on an important theorem of Nikolaev \cite[Theorem 10.10.13]{Burago:2001aa}, the proof of Theorem~\ref{thm:b} uses a deep result in the theory of metric geometry, namely, Perelman's stability theorem \cite{Kapovitch:2007aa}. More specifically, Nikolaev tells us that Alexandrov spaces with two-sided curvature bounds can be equipped with Riemannian structure (possibly with a reduced degree of smoothness and with a possibly nonsmooth boundary) and Perelman tells us that a Gromov--Hausdorff-convergent sequence of (isometry classes of) $n$-dimensional Alexandrov spaces with a uniform lower bound on curvature must eventually be topologically constant, provided that the limit is also $n$-dimensional. The role of the Hausdorff measure hypothesis appearing in Theorem~\ref{thm:b} is to ensure that there is indeed no dimensional collapse in the limit. 

Sticking with point-like ordered $K$-theory, but moving beyond tracially lyriform $\cs$-algebras of `Bauer type' (in the sense of Definition~\ref{def:compact}), we show in Section~\ref{section:limits} how to view arbitrary simple inductive limits of prime dimension drop algebras as $\lyre$-limits of sequences of tracial Wasserstein spaces observing  spheres or balls.

\begin{thm} \label{thm:c}
Let $A$ be a simple $\cs$-algebra that is isomorphic to an inductive limit of prime dimension drop algebras. Then, there are inductive systems of classifiable tracial Wasserstein spaces $((\js_{\sphere^{2k-1}},W_{1,\rho_k}),\varphi_k)_{k\in\nn}$ and $((\js_{\ball^{2k}},W_{1,\omega_k}),\psi_k)_{k\in\nn}$ in $\lyre$ such that $(\varinjlim (\js_{\sphere^{2k-1}},\varphi_k),W_{1,\rho})$ and $(\varinjlim (\js_{\ball^{2k}},\psi_k),W_{1,\omega})$ provide isomorphic tracially lyriform structures on $A$, that is, there is a $^*$-isomorphism $\varinjlim (\js_{\sphere^{2k-1}},\varphi_k) \to \varinjlim (\js_{\ball^{2k}},\psi_k) \cong A$ that induces an isometric isomorphism of trace spaces with respect to the limit metrics $W_{1,\omega}$ and $W_{1,\rho}$.
\end{thm}

The proof of Theorem~\ref{thm:c} is based on the Lazar--Lindenstrauss theorem \cite[Theorem 5.2]{Lazar:1971kx}, which shows that every metrisable Choquet simplex (in particular, the trace space $T(A)$) is affinely homeomorphic to an inverse limit of finite-dimensional simplices $\Delta_k$.  In Theorem~\ref{thm:arachnid}, we construct an intertwining between such a sequence and sequences of spheres $\sphere^{2k-1}$ or balls $\ball^{2k}$, with the metrics $\rho_k$ and $\omega_k$ then defined to be the intrinsic metrics obtained by pulling back the $\ell_1$-metrics on the simplices $\Delta_k$. For this reason, we refer to these metric balls or spheres as \emph{simplicial}. Together with the $K$-theory computation Corollary~\ref{cor:ballsandspheres}, the connecting maps between the classifiable $\cs$-algebras $\js_{\sphere^{2k-1}}$ or $\js_{\ball^{2k}}$ are then obtained by classification \cite{Carrion:wz}, which also provides the isomorphisms $\varinjlim (\js_{\sphere^{2k-1}},\varphi_k) \cong A \cong \varinjlim (\js_{\ball^{2k}},\psi_k)$.  

We use simplicial spheres and balls because they are covered by Theorem~\ref{thm:models} and they support Goal~\ref{goal:b}. More specifically, these length spaces exhibit optimal continuous transport in the sense developed in \cite{Jacelon:2021wa,Jacelon:2021vc} and recalled in Section~\ref{subsection:transport}. Definition~\ref{def:transport} proposes a noncommutative extension of the theory to the setting of tracial Wasserstein spaces, and Theorem~\ref{thm:transport} establishes the optimal transport property for generalised dimension drop algebras over simplicial spheres and balls, which are the building blocks used in Theorem~\ref{thm:a} to construct $(\js_{\sphere^{2k-1}},\rho_k)$ and $(\js_{\ball^{2k}},\omega_k)$. These building blocks are equipped with canonical \emph{nuclei} \eqref{eqn:drmat}, which are compact sets of tracially Lipschitz elements relative to which we can measure distances between unitary orbits of embeddings (see Definition~\ref{def:unitarydistance}). These nuclei fit together to provide $\sigma$-compact \emph{cores} $\mathcal{L}_r(\js_X)$ for $X=\sphere^{2k-1}$ or $\ball^{2k}$ (see \eqref{eqn:unioncore}), and corresponding cores for $A$ under the construction of Theorem~\ref{thm:c}. The conclusion presented in Corollary~\ref{cor:spiderorbits} is that, for certain codomains, the unitary distance relative to these cores can be computed in terms of the pro-$W_\infty$ metrics $\overrightarrow{W}_{\infty,\omega}$ and $\overrightarrow{W}_{\infty,\rho}$ defined in \eqref{eqn:omegapro} and \eqref{eqn:rhopro}.

\begin{thm} \label{thm:d}
Let $A$ be a tracially lyriform $\cs$-algebra constructed as in Theorem~\ref{thm:c}, and let $\mathfrak{ball}(A)$ and $\mathfrak{sphere}(A)$ be the associated cores. Let $B$ be a simple, separable, unital, nuclear, $\js$-stable $\cs$-algebra with a unique trace, and let $\varphi,\psi\colon A\to B$ be unital $^*$-homomorphisms. Then,
\[
d_U(\varphi,\psi)|_{\mathfrak{ball}(A)} = \overrightarrow{W}_{\infty,\omega}(\varphi,\psi).
\]
If in addition $B$ has real rank zero, then
\[
d_U(\varphi,\psi)|_{\mathfrak{sphere}(A)} = \overrightarrow{W}_{\infty,\rho}(\varphi,\psi).
\]
\end{thm}

At the conclusion of the paper in Section~\ref{subsection:random}, we approach the intersection of Goals~\ref{goal:c}~and~\ref{goal:b} and begin investigating the probabilistic structure of the embedding spaces $\emb(\js_{\ball^k},B)/\sim_{au}$ and $\emb(\js_{\sphere^{k-1}},B)/\sim_{au}$. Here, we assume that $B$ is a monotracial classifiable $\cs$-algebra that is also real rank zero and $K_1$-trivial in the case of spheres $\sphere^{k-1}$. Examples to keep in mind are the Jiang--Su algebra $\js$ or the universal uniformly hyperfinite (UHF) algebra $\mathcal{Q}$, as appropriate. By classification \cite[Theorem B]{Carrion:wz} and Corollary~\ref{cor:dimdroporbits}, these metric spaces are isometrically isomorphic to the pro-$W_\infty$ spaces $(\prob(\ball^k),\overrightarrow{W}_\infty)$ and $(\prob(\sphere^{k-1}),\overrightarrow{W}_\infty)$, respectively. To make sense of the probability structure, we generate random (approximate unitary equivalence classes of) extremal embeddings $\varphi_x$ by randomly selecting an extremal trace $x$ according to a fixed faithful measure $\mu$ (see Section~\ref{subsubsection:extremal}), and random embeddings $\varphi_{\mu_n}$ associated with an empirical measure $\mu_n=\frac{1}{n}\sum_{i=1}^n\delta_{x_i}$ (see Section~\ref{subsubsection:empirical}). For highly curved spaces like high-dimensional spheres (with $\mu$ the usual volume measure), or more broadly, for suitable measured length spaces $(X,\rho,\mu)$ whose $\infty$-Ricci curvature in the sense of \cite{Lott:2009aa,Sturm:2006aa,Sturm:2006ab} is bounded below, we can estimate the tracial variance of core elements $f\in\mathcal{L}_r(\js_X)$ (see \eqref{eqn:variance}). For the empirical embeddings $\varphi_{\mu_n}$, we get a large deviations estimate \eqref{eqn:dusanov} via a deep result in probability theory called Sanov's theorem.

\begin{thm} \label{thm:e}
Let $\sphere^{k-1}$ be an odd Euclidean unit sphere, for some natural number $k\ge3$, and let $\mu$ denote the Haar measure on $\sphere^{k-1}$. Let $B$ be a simple, separable, unital, nuclear, $\js$-stable $\cs$-algebra that has a unique trace , real rank zero and trivial $K_1$. Then, the tracial variance of $\mu$-random extremal embeddings $\varphi_x \colon \js_{\sphere^{k-1}} \to B$ relative to the core $\mathcal{L}(\js_{\sphere^{k-1}})$ is bounded by
\[
\sup_{f\in\mathcal{L}(\js_{\sphere^{k-1}})}\Var\tau_B\circ\varphi_x(f) \le \frac{1}{k-2}.
\]
For empirical embeddings $\varphi_{\mu_n}\colon \js_{\sphere^{k-1}} \to B$, the core $\mathcal{L}_r(\js_{\sphere^{k-1}})$ is with high probability tracially concentrated near the mean. More precisely, for $\eps>0$ and large $n\in\nn$,
\[
\pp_\mu\left(\sup_{f\in\mathcal{L}(\js_{\sphere^{k-1}})}|\tau_B\circ\varphi_{\mu_n}(f)  - \tau_B\circ\varphi_\mu(f)| \le \eps\right) \ge 1 - \exp\left(-\frac{n\eps^2}{\pi^2}\right).
\]
\end{thm}

By the spherical isoperimetric inequality, we also get an estimate of large tracial deviation \eqref{eqn:mean} for random ($\sim_{au}$-equivalence classes of) extremal embeddings $\js_{\sphere^{2k-1}} \to B$. Another famous consequence of spherical isoperimetry is Dvoretzky's theorem \cite{Dvoretzky:1961aa}, which shows that the unit sphere $(\sphere_E,\rho_{E})$ of a real normed space $E$ admits high-dimensional approximately Euclidean sections provided that the dimension of $E$ is sufficiently large. We apply this in Theorem~\ref{thm:dvoretzky} to show that the intrinsic tracial Wasserstein spaces $(\js_{\sphere_E},W_{\fp,\rho_E})$ and $(\js_{\ball_E},W_{\fp,\rho_E})$ admit approximately tracially isometric embeddings into tracial $\fp$-Wasserstein spaces observing high-dimensional \emph{Euclidean} spheres or balls. These sorts of measure concentration phenomena and their implications for the asymptotic geometry of embedding spaces were essential motivations for the present work.

This article is also inspired by spiders. As Barth writes in \cite[Chapter VIII]{Barth:2002aa}, ``The world in which spiders live is a world full of vibrations, and their vibration sense is correspondingly well developed.'' Arachnids detect these vibrations via slit sense organs embedded in their exoskeletons. In contrast to others in the class Arachnida, like scorpions, whip spiders and whip scorpions (see \cite[p.41]{Barth:2002aa}), spiders are unique in the tendency of slit organs near their leg joints to be clustered into closely spaced parallel arrangements called \emph{lyriform organs}. One such arrangement is displayed in Figure~\ref{fig:orbweaver}, which shows a female cross orb-weaver (\emph{Araneus diadematus}) and a selection of her anatomy (cf.\ \cite[Fig.~1]{Miller:2022aa}). The picture explains the terminology: `lyriform' means `lyre shaped'. The different lyriform organs distributed across the spider's legs play exteroceptive or proprioceptive roles, that is, they give information about the spider's environment and her positioning within it. That is also the purpose of lyriform structure in the $\cs$-world. An inductive limit of Wasserstein spheres has the function of a tracial exoskeleton, with the lyriform structure providing geometric information about the way that the underlying $\cs$-algebra is embedded in its surroundings.

\begin{figure}[!htbp]
\centering
\includegraphics[width=0.9\textwidth]{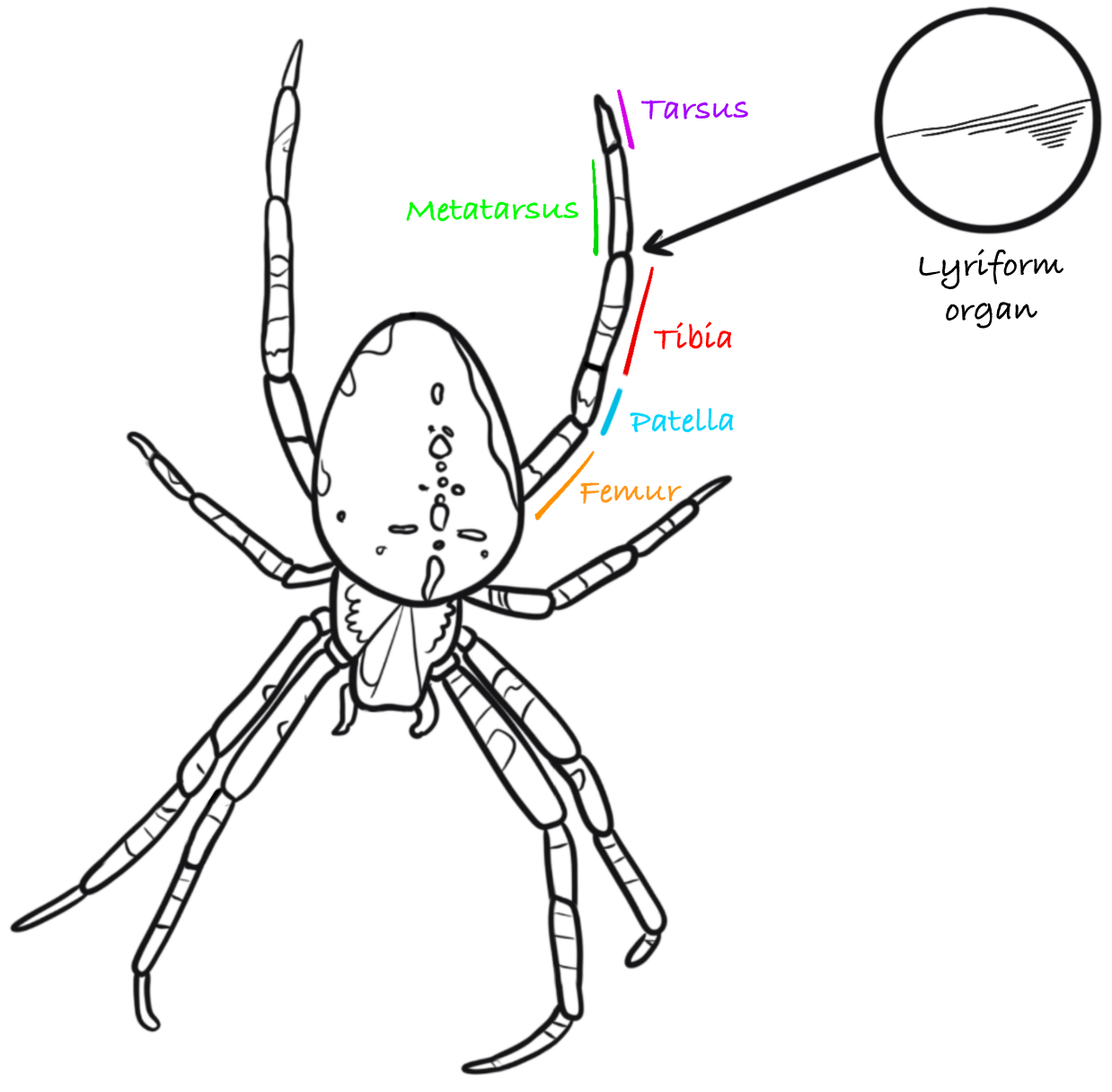}
\caption{Lyriform organ on the lateral side of the distal end of the tibia in \emph{Araneus diadematus}.} \label{fig:orbweaver}
\end{figure}

\subsection*{Acknowledgements} This work is dedicated to the orb-weaver who lived outside my bedroom window in the summer of 2020, and to my dear friend Galina Starchevus, who has captured this spider's likeness in Figure~\ref{fig:orbweaver}.

\subsection*{Organisation}

In Section~\ref{section:qmcs}, we recall the definitions of quantum metric Choquet and Bauer simplices and identify the metric property used to build projectionless classifiable models $\js_X$ observing compact path-connected metric spaces $(X,\rho)$ (see Theorem~\ref{thm:models}). Our list of examples is expanded to include certain fractals (Example~\ref{ex:fractals}) and certain Alexandrov spaces (Theorem~\ref{thm:alexandrov}). We show in Proposition~\ref{prop:ktheory} that the models $\js_{\sphere^{2k-1}}$ and $\js_{\ball^{k}}$ have suitable $K$-theory to be used as building blocks in later inductive limit constructions. Section~\ref{section:lyriform} introduces the notion of a tracially lyriform $\cs$-algebra (Definition~\ref{def:lyriform}), explains the relationship with quantum metric Choquet simplices (Theorem~\ref{thm:compact}) and describes examples arising from quantum groups (Theorem~\ref{thm:free}) and topological dynamics (Theorem~\ref{thm:dynamics}). Section~\ref{section:wass} focuses on our primary examples, what we refer to as tracial $\fp$-Wasserstein spaces (Definition~\ref{def:qwass}) and quasi-Wasserstein spaces (Definition~\ref{def:quasiwass}), and we analyse the collections of isometry classes of these spaces via the quantum intertwining gap (see Definition~\ref{def:quig}, Theorem~\ref{thm:quigwass}, Corollary~\ref{cor:quigwass} and Corollary~\ref{cor:contractible}). Section~\ref{section:limits} is devoted to the category $\lyre$ of tracially lyriform $\cs$-algebras (Definition~\ref{def:category}) and inductive limits within this category (Proposition~\ref{prop:colimits}). In Section~\ref{section:arachnid}, we construct simple inductive limits of prime dimension drop algebras as $\lyre$-limits of tracial Wasserstein spaces over simplicial spheres and balls (Theorem~\ref{thm:arachnid}). Finally, in Section~\ref{section:geostats} we use optimal transport theory (Definition~\ref{def:transport} and Theorem~\ref{thm:transport}) to compute distances between unitary orbits of embeddings (Theorem~\ref{thm:orbits}, Corollary~\ref{cor:dimdroporbits} and Corollary~\ref{cor:spiderorbits}) and to analyse the tracial central tendency of random embeddings (Section~\ref{subsection:random}).

\begin{notation} \label{notation}
We use the following notation and terminology throughout the article.

\begin{items}
\item The set of nonzero positive integers is denoted by $\nn$.
\item The notation $X\ssubset Y$ means that $X$ is a finite subset of $Y$.
\item If $A$ is a $\cs$-algebra, then $A_{sa}$ denotes the set of self-adjoint elements of $A$ and $A_+$ denotes the set of positive elements of $A$.
\item A \emph{trace} on a unital $\cs$-algebra $A$ means a \emph{tracial state}, that is, a positive linear functional $\tau\colon A\to\cc$ with $\|\tau\|=1=\tau(1)$ that satisfies the trace identity $\tau(uau^*)=\tau(a)$ for every $a\in A$ and unitary $u\in A$. We write $T(A)$ for the set of traces on $A$ and refer to it as the \emph{trace space of $A$}. The trace space is a Choquet simplex when equipped with the $\ws$-topology (that is, the topology of pointwise convergence on elements of $A$).
\item A $\cs$-algebra $A$ is said to be \emph{classifiable} if it is unital, simple, separable, nuclear, $\js$-stable and satisfies the UCT.
\item If $A$ and $B$ are unital $\cs$-algebras, then $\emb(A,B)$ denotes the set of unital embeddings (that is, injective $^*$-homomorphisms) $A\to B$. Two $^*$-homomorphisms $\varphi,\psi\colon A\to B$ are said to be \emph{approximately unitarily equivalent}, written $\varphi\sim_{au}\psi$, if for every $\varepsilon>0$ and every finite set $F\subseteq A$ there is a unitary $u\in B$ such that $\|\varphi(a)-u\psi(a)u^*\|<\varepsilon$ for every $a\in F$. If $A$ is separable, this is equivalent to the existence of a sequence $(u_n)_{n\in\nn}$ of unitaries in $B$ such that $\varphi(a)=\lim_{n\to\infty}u_n\psi(a)u_n^*$ for every $a\in A$. We frequently  rely on both the existence \cite[Corollary C and Theorem 9.9]{Carrion:wz} and uniqueness \cite[Theorem B]{Carrion:wz} aspects of classification to lift compatible morphisms between ordered $K$-theory $(K_0(\cdot),K_0(\cdot)_+,[1_{(\cdot)}],K_1(\cdot))$ and traces $T(\cdot)$ to unital embeddings, uniquely up to approximate unitary equivalence in situations where the total invariant $\underline{K}T_u$ developed in \cite{Carrion:wz} reduces to ordered $K$-theory and traces.
\item A stably finite $\cs$-algebra $A$ is said to be \emph{$K$-connected} if the ordered $K_0$-group of the minimal unitisation of $A$ admits a unique state.
\item For a compact Hausdorff space $X$, we identify $T(C(X))$ with the space $\prob(X)$ of Borel probability measures on $X$. Given $\mu\in\prob(X)$ and a continuous map $h\colon X\to X$, $h_*\mu$ denotes the pushforward measure $\mu\circ h^{-1}$.
\item Let $(X,\rho)$ be a metric space (implicitly nonempty throughout the article). For $K>0$, $\mathrm{Lip}_K(X,\rho)$ denotes the $K$-Lipschitz functions $X\to\rr$, that is,
\[
\qquad \quad \mathrm{Lip}_K(X,\rho) = \{f\colon X\to\rr \mid |f(x)-f(y)|\le K\cdot\rho(x,y) \text{ for every } x,y \in X\}.
\]
The set of Lipschitz functions $X\to\rr$ is $\lipo(X,\rho) = \bigcup_{K=1}^\infty \mathrm{Lip}_K(X,\rho)$. If the metric $\rho$ is understood from context, we may write $\mathrm{Lip}_K(X)$ or $\lipo(X)$. The \emph{diameter} of $(X,\rho)$ is $\diam(X,\rho) = \sup_{x,y\in X}\rho(x,y)$ and the \emph{radius} is $r_{(X,\rho)} = \diam(X,\rho)/2$. For $Y\subseteq X$, $Y_\eps$ denotes its \emph{$\eps$-neighbourhood}, that is, $Y_\eps = \bigcup\limits_{x\in Y} B_\eps(x)$, where $B_\eps(x)$ is the open ball of radius $\eps$ centred at $x\in X$.
\item If $K$ is a compact, convex subset of a real topological vector space, then $\aff(K)$ denotes the space of of continuous affine maps $K\to\rr$ (where `affine' means that finite convex combinations are preserved). When $K=T(A)$ and $a\in A$, $\widehat{a}$ denotes the element $\tau\mapsto\tau(a)$ of $\aff(T(A))$.
\end{items}
\end{notation}

\section{Quantum metric Choquet simplices} \label{section:qmcs}

The Cuntz--Pedersen quotient of a unital $\cs$-algebra $A$ is the Banach space $A^q=A_{sa}/A_0$,
where
\begin{equation} \label{eqn:a0}
A_0 = \{a\in A_{sa} \mid \tau(a)=0 \text{ for every } \tau\in T(A)\}.
\end{equation}
It is an ordered vector space whose positive cone $A^q_+$ is the image $q(A_+)$ of the positive cone of $A$ under the quotient map $q\colon A_{sa} \to A_+$. The definition of a quantum metric Choquet simplex is framed for those $A$ for which $A^q$ is \emph{tracially ordered}, meaning that $T(A)\ne\emptyset$ and
\[
A^q_+ = \{a+A_0 \mid a\in A_{sa},\: \tau(a)\ge0 \text{ for every } \tau\in T(A)\}.
\]
This property holds for algebraically simple $\cs$-algebras $A$, as well as certain continuous-trace $\cs$-algebras including those of the form $M_n(C(X))$ and subalgebras like generalised dimension drop algebras (see \cite[Proposition 2.7]{Jacelon:2024aa}, Definition~\ref{def:dimdrop} and Corollary~\ref{cor:wasslyre}). It ensures that $A_q$ is a complete order unit space (in fact, isomorphic to $\aff(T(A))$), and that the order-unit norm coincides with the quotient norm (see \cite[Proposition 2.8]{Jacelon:2024aa}).

\begin{definition} \label{def:qmcs}
Let $A$ be a separable, unital $\cs$-algebra and $L\colon A\to[0,\infty]$ a self-adjoint, lower semicontinuous, densely finite seminorm. We say that $(A,L)$ is a \emph{$\cs$-algebraic quantum metric Choquet simplex} if
\begin{enumerate}[(1)]
\item $A^q$ is tracially ordered;
\item $\ker L = \Span_\cc(\{1\}\cup A_0) = \{a \in A \mid \widehat{a} \in \cc1\}$;
\item the metric
\begin{equation} \label{eqn:rho}
\rho_L(\sigma,\tau) = \sup\{|\sigma(a)-\tau(a)| \mid a\in A,\:L(a)\le 1\}
\end{equation}
induces the $\ws$-topology on the trace space $T(A)$.
\end{enumerate}
We may refer to $(A,L)$ as \emph{classifiable} or \emph{projectionless} (and so on) if these adjectives apply to $A$.
\end{definition}

Note that, by definition, $(A,L)$ is a $\cs$-algebraic quantum metric Choquet simplex if and only if $(A^q,L)$ is a compact quantum metric space (with the convention that the seminorm $L$ is allowed to take the value $\infty$).

\begin{definition} \label{def:qmbs}
Let $(A,L)$ be a $\cs$-algebraic quantum metric Choquet simplex. Then, we call $(A,L)$ a \emph{$\cs$-algebraic  quantum metric Bauer simplex} if in addition $\partial_e(T(A))$ is compact and is equipped with a $\ws$-compatible metric $\rho$ for which
\begin{equation} \label{eqn:L}
L(a) = L_\rho\left(\widehat{a}|_{\partial_e(T(A))}\right) := \sup\left\{\frac{|\sigma(a)-\tau(a)|}{\rho(\sigma,\tau)} \mid \sigma\ne \tau \in \partial_e(T(A))\right\}.
\end{equation}
In this case, we may say that $(A,L_\rho)$ \emph{is associated with} or \emph{observes} the metric space $(\partial_e(T(A)),\rho)$.
\end{definition}

The latter terminology reflects the fact that the typical starting point in the construction of a $\cs$-algebraic quantum metric Bauer simplex is a compact metric space $(X,\rho)$, with $(A,L_\rho)$ serving as a noncommutative space of observables.

A major source of examples of $\cs$-algebraic quantum metric Bauer simplices $(A,L)$ is provided by \cite[Theorem 4.4]{Jacelon:2022wr}, which in particular asserts that every compact, connected Riemannian manifold equipped with its intrinsic metric $\rho$ is realisable as $(\partial_e(T(A)),\rho_L)$  for some $A$ that is classifiable and projectionless. The proof of this result presented in \cite[Theorem 4.4]{Jacelon:2022wr} is a bit murkier than it should be, since the intention was to fit the considered examples (Riemannian manifolds) into a broader class of geometric structures called \emph{length spaces}, or more specifically \emph{Alexandrov spaces}, which are complete length spaces whose curvature is bounded below 
(see Definition~\ref{def:length} and Definition~\ref{def:alexandrov}).

Let us take this opportunity to unmuddy the waters. First, let us clarify the property of a compact, connected metric space $(X,\rho)$ that is used in the proof of \cite[Theorem 4.4]{Jacelon:2022wr} to produce a projectionless model $(\js_X,L_\rho)$.

\begin{definition} \label{def:dimdrop}
Let $X$ be a connected compact Hausdorff space. A \emph{generalised dimension drop algebra over $X$} is a $\cs$-algebra of the form
\[
X_{p,q} = \{f\in C(X,M_p\otimes M_q) \mid f(x_0)\in M_p\otimes1_q,f(x_1)\in 1_p\otimes M_q\}
\]
for some integers $p,q\in\nn$ and distinct points $x_0,x_1\in X$. If $p$ and $q$ are coprime, then $X_{p,q}$ is said to be \emph{prime}.
\end{definition}

Note that $X_{p,q}$ is separable, unital, nuclear and satisfies the UCT, and is also projectionless if it is prime. These spaces are used as building blocks in the construction of $\js_X$ as an inductive limit in which Lipschitz elements in finite stages are mapped to tracially Lipschitz elements in the limit.

\begin{theorem}[\cite{Jacelon:2022wr}] \label{thm:models}
Let $(X,\rho)$ be a connected compact metric space with the following property: there exist distinct points $x_0,x_1\in X$ such that, for every $y\in X$ and every $\eps>0$, there is a bi-Lipschitz path from $x_0$ to $x_1$ that passes within $\eps$ of $y$. Then, there exists a projectionless, classifiable $\cs$-algebraic quantum metric Bauer simplex $(\js_X,L_\rho)$ observing $(X,\rho)$. 
\end{theorem}

Now let us see how to broaden our collection of examples.

\begin{definition} \label{def:length}
A metric space $(X,\rho)$ is called a \emph{length space} if $\rho$ is \emph{intrinsic}, that is,
\begin{equation} \label{eqn:rect}
\rho(x,y) = \inf\{\ell(\gamma) \mid \gamma \text{ is a rectifiable path from $x$ to $y$}\}.
\end{equation}
Here, a path $\gamma$ from $x$ to $y$ (that is, a continuous function $\gamma$ from some interval $[a,b]$ to $X$ with $\gamma(a)=x$ and $\gamma(y)=b$) is \emph{rectifiable} if the length
\begin{equation} \label{eqn:length}
\ell(\gamma) = \sup\left\{ \sum_{i=1}^{N} \rho(\gamma(t_{i-1}),\gamma(t_i)) \mid N\in\nn,\: a = t_0 \le t_1 \le \dots \le t_N = b\right\}
\end{equation}
of $\gamma$ is finite. In other words, $(X,\rho)$ is a length space if and only if, given $x,y\in X$ and $\eps>0$, there is a path from $x$ to $y$ whose length is at most $\rho(x,y)+\eps$.
\end{definition}

\begin{definition} \label{def:geodesic}
A \emph{shortest path} is a continuous function $\gamma \colon [a,b] \to X$ whose length is minimal among all paths joining the same endpoints $\gamma(a)$ and $\gamma(b)$. When \emph{parameterised by arc length}, a shortest path $\gamma$ in a length space is defined on $[0,\ell(\gamma)]$ and satisfies
\[
\rho(\gamma(s),\gamma(t)) = \ell(\gamma|_{[s,t]}) = |s-t|
\]
for every $s$ and $t$ in this interval.  A length space is \emph{strictly intrinsic} if any two points can be joined by a shortest path.
\end{definition}

\begin{remark} \label{rem:intrinsic}
Every complete, locally compact length space $(X,\rho)$ is strictly intrinsic (see \cite[Theorem 2.5.23]{Burago:2001aa}). The Hopf--Rinow theorem for locally compact length spaces (see \cite[Theorem 2.5.28]{Burago:2001aa}, which is a generalisation of the corresponding result \cite[Section 5.3]{Carmo:1976um} for Riemannian manifolds) in fact characterises completeness in terms of the extendability of \emph{geodesics} (curves that \emph{locally} are shortest paths).
\end{remark}

\begin{example} \label{ex:normed}
If $(E,\|\cdot\|_E)$ is a real normed space and $\rho_E$ is the associated metric $\rho_E(x,y)=\|x-y\|_E$, then $(E,\rho_E)$ is a strictly intrinsic length space. Indeed, $t\mapsto(1-t)x+ty$ is a shortest path from $x\in E$ to $y\in E$. This is in fact the \emph{unique} shortest path if (and only if) the unit ball of $E$ is strictly convex, meaning that $\|(1-t)v+tw\|<1$ for every $t\in(0,1)$ and every pair of distinct norm-one vectors $v,w\in E$ (see \cite[Proposition I.1.6]{Bridson:1999aa}). So, uniqueness holds in $\ell_p(\rr)$ for every $p\in(1,\infty)$ but nor for $p\in\{1,\infty\}$. These observations apply equally well to compact, convex subsets of $(E,\rho_E)$, which also satisfy the hypotheses of Theorem~\ref{thm:models} and therefore can be observed via classifiable, projectionless $\cs$-algebraic quantum metric Bauer simplices in the sense of Definition~\ref{def:qmbs}.
\end{example}

\begin{example} \label{ex:intrinsic}
Every metric space $(X,\rho)$ admits an associated intrinsic metric $\widehat{\rho}$ defined via \eqref{eqn:rect}. Note that the length of a rectifiable curve is the same whether one inputs $\rho$ or $\widehat{\rho}$ on the right-hand side of \eqref{eqn:length}, and consequently, $\widehat{\widehat{\rho}}=\widehat{\rho}$ (see \cite[Proposition 2.3.12]{Burago:2001aa}). For example, in Section~\ref{section:arachnid} we identify the sphere $\sphere^{k-1}$ with the $(k-1)$-skeleton of a face of the $(k+1)$-dimensional $\ell_1$-ball, and equip it with the associated intrinsic metric. This metric is Lipschitz equivalent to the one induced by the Euclidean norm $\|\cdot\|_2$ (that is, to the usual round metric on the sphere), so in particular is complete and therefore strictly intrinsic (cf.\ Remark~\ref{rem:intrinsic}). Moreover, since round spheres satisfy the hypotheses of Theorem~\ref{thm:models}, so do these `simplicial' spheres, which can therefore also be observed via classifiable, projectionless $\cs$-algebraic quantum metric Bauer simplices.
\end{example}

There are various equivalent ways of defining curvature bounds for length spaces (see \cite[Chapter 4]{Burago:2001aa}), which in the Riemannian setting correspond to bounds on the sectional curvature (see \cite[Chapter 6]{Burago:2001aa}). In Definition~\ref{def:alexandrov}, we recall the `monotonicity condition' \cite[Definition 4.3.1]{Burago:2001aa}. What these definitions typically involve is comparison between triangles formed from shortest paths in the space and corresponding reference triangles in standard `model spaces'. These latter are two-dimensional Riemannian manifolds $(\mm_k,\rho_k)$ of constant curvature $k$, namely the Euclidean sphere of radius $\frac{1}{\sqrt k}$ (if $k>0$), the Euclidean plane (if $k=0$) and the hyperbolic plane with its metric scaled by $\frac{1}{\sqrt{-k}}$ (if $k<0$). Note that the diameter of $\mm_k$ is
\[
R_k = \begin{cases}
\frac{\pi}{\sqrt k} & \text{if $k>0$}\\
\infty & \text{if $k\le 0$}
\end{cases}
\]
and that whenever $x,y,z$ are distinct elements of a length space that are sufficiently close together (meaning that $\max\{\rho(x,y),\rho(y,z),\rho(z,x)\}<R_k$ and $\rho(x,y)+\rho(y,z)+\rho(z,x)<2R_k$), then there is a (unique, up to isometry) triangle in $\mm_k$ whose side lengths are $\rho(x,y)$, $\rho(y,z)$ and $\rho(z,x)$. 

\begin{definition} \label{def:alexandrov}
Let $(X,\rho)$ be a connected, strictly intrinsic length space and let $k\in\rr$. For distinct points $x,y,z\in X$, the \emph{comparison angle $\Angle_k(x,y,z)$ relative to $\mm_k$} is the angle $\measuredangle_k(\bar x, \bar y, \bar z)$, where $\bar x, \bar y, \bar z \in \mm_k$ form a triangle whose side lengths are $\rho_k(\bar x,\bar y)=\rho(x,y)$, $\rho_k(\bar y,\bar z)=\rho(y,z)$ and $\rho_k(\bar z,\bar x)=\rho(z,x)$. (By convention, the variable in the middle indicates the vertex of the angle being measured.) We say that $(X,\rho)$ has \emph{curvature at least $k$} (respectively, \emph{curvature at most $k$}) if every point in $X$ has a neighbourhood $U$ such that, for any two shortest paths $\alpha$ and $\beta$ contained in $U$ and starting from the same point $\alpha(0)=\beta(0)=p\in X$, the function
\begin{equation} \label{eqn:angle}
\theta_{\alpha,\beta}(s,t) = \Angle_k(\alpha(s),p,\beta(t))
\end{equation}
is nonincreasing (respectively, nondecreasing) in $s$ when $t$ is kept fixed and vice versa. If $(X,\rho)$ has curvature at least $k$ for some $k\in\rr$, then it is called an \emph{Alexandrov space}.
\end{definition}

\begin{example} \label{example:riemannianlimits}
By Gromov's theorem (see \cite[Theorem 10.7.2]{Burago:2001aa}), for every $k\in\rr$, $n\in\nn$ and $D>0$, if $(X_n,\rho_n)$ is a Gromov--Hausdorff-convergent sequence of compact, connected Riemannian manifolds of sectional curvature at least $k$, dimension at most $n$ and diameter at most $D$, then the limit space is an Alexandrov space of curvature at least $k$, dimension at most $n$ and diameter at most $D$.
\end{example}

\begin{example} \label{ex:fractals}
Fractals like the Sierpi\'nski carpet $\mathcal{F}^2\subseteq\rr^2$, the Menger sponge $\mathcal{F}^3\subseteq\rr^3$ and their higher-dimensional analogues $\mathcal{F}^n\subseteq\rr^n$ (studied, for example, in \cite{Berkove:2020aa}), as well as the Sierpi\'nski gasket $\mathcal{G}\subseteq\rr^2$, are compact, connected, strictly intrinsic length spaces when equipped with the intrinsic metrics induced by the ambient Euclidean distance.  A slick way of seeing that they are \emph{not} Alexandrov spaces is to observe that in each case, the Hausdorff dimension $\dim_H$ is an element of $\rr\setminus\zz$, which by \cite[Theorem 10.8.2]{Burago:2001aa} is not possible for Alexandrov spaces. Indeed, the Moran--Hutchinson theorem for iterated function systems (see \cite[Theorem 5.3(1)]{Hutchinson:1981aa}) shows that
\[
\dim_H\mathcal{G}=\frac{\log3}{\log2} \quad\text{and}\quad \dim_H\mathcal{F}^n=\frac{\log((n+2)2^{n-1})}{\log3}.
\]
That said, we can still verify the hypotheses of Theorem~\ref{thm:models}. In each case, the intrinsic metric is Lipschitz equivalent to the Euclidean one (see \cite[Lemma 2.12]{Barlow:1998aa} for  $\mathcal{G}$ and \cite[Theorem 4.5]{Berkove:2020aa} for $\mathcal{F}^n$). Each space is also approximable by its increasing sequence of finite-stage $1$-skeletons $X_k$, $k\in\nn$. See Figure~\ref{fig:sierpinski}. We fix diametrically opposite points $x_0,x_1$ in the zeroth stage (that is, points in the cube $[0,1]^n$ in the case of $\mathcal{F}^n$ or the preliminary equilateral triangle in the case of $\mathcal{G}$). Given an arbitrary point $y$ and an $\eps>0$, we find an integer $k$ and $y'\in X_k$ whose distance to $y$ is $<\eps$. By concatenating coordinatewise increasing paths from $x_0$ to $y'$ and $y'$ to $x_1$ in the finite graph $X_k$, we can find a simple, piecewise linear bi-Lipschitz path from $x_0$ to $x_1$ that passes through $y'$.  We deduce the existence of classifiable, projectionless $\cs$-algebraic quantum metric Bauer simplices observing $\mathcal{G}$ and each $\mathcal{F}^n$.
\begin{figure}[!htbp]
\centering
\includegraphics[width=0.45\textwidth]{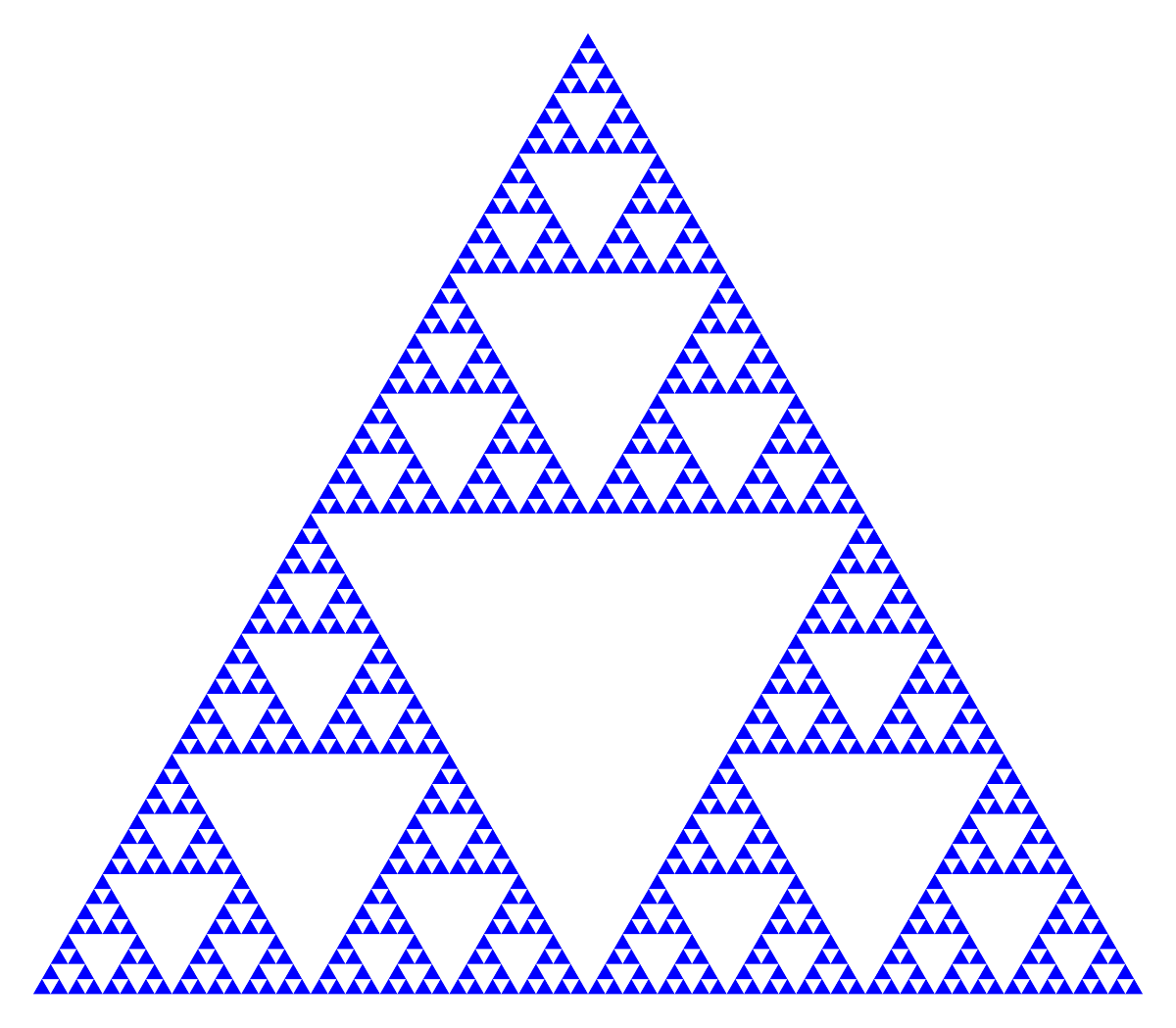} \includegraphics[width=0.4\textwidth]{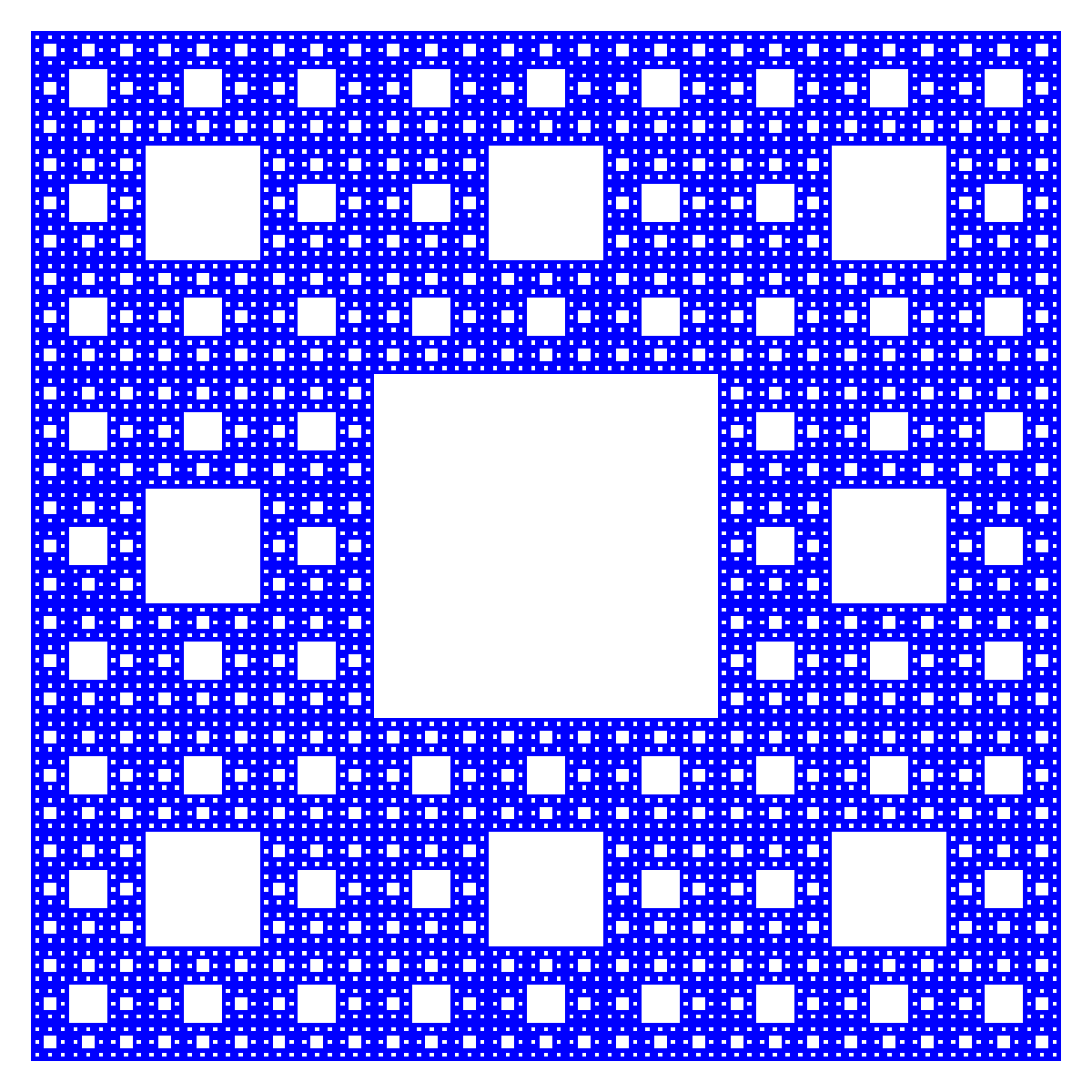}
\caption{The Sierpi\'nski gasket $\mathcal{G}$ (left) and the Sierpi\'nski carpet $\mathcal{F}^2$ (right).} \label{fig:sierpinski}
\end{figure}
\end{example}

Compact Alexandrov spaces of finite dimension have almost the right structure to apply Theorem~\ref{thm:models}. Here, `dimension' typically means the Hausdorff dimension, but as mentioned in \cite[Chapter 10]{Burago:2001aa}, all known notions of dimension, including the Lebesgue covering dimension, are equivalent for Alexandrov spaces. Such a space $X$ is somewhat close to being a manifold, and indeed contains an open dense set $U$ whose points are all contained in neighbourhoods bi-Lipschitz equivalent to open subsets of Euclidean space (see \cite[Theorem 10.8.3]{Burago:2001aa}). In Definition~\ref{def:angles}, we recall the definitions of spaces of directions $\Sigma_p(X)$ and associated tangent cones $K_p(X)$, which are generalisations of tangent spaces of manifolds. For points $p\in U$, $\Sigma_p(X)$ is isometrically isomorphic to the Euclidean sphere $\sphere^{n-1}$ and $K_p(X)$ is isometrically isomorphic to the Euclidean space $\rr^{n}$, where $\dim X=n>1$. (If $n=1$, then $X$ is an interval or circle and $\Sigma_p(X)$ may consist of one or two points.) What remains true for an arbitrary $p\in X$ is that $K_p(X)$ is an $n$-dimensional Alexandrov space of nonnegative curvature, and $\Sigma_p(X)$ is a compact $(n-1)$-dimensional Alexandrov space of curvature $\ge1$ (see \cite[Proposition 10.9.1, Corollary 10.9.5 and Corollary 10.9.6]{Burago:2001aa}). Ergo, a common technique in the theory of Alexandrov spaces is induction on dimension; see Definition~\ref{def:angles}, which provides an inductive definition of the boundary of $X$. In Theorem~\ref{thm:alexandrov}, we will assume that this boundary is empty, and also that the curvature of $X$ is bounded both below \emph{and above}. A deep result of Nikolaev \cite[Theorem 10.10.13]{Burago:2001aa} will then allow us to adapt the proof of \cite[Theorem 4.4]{Jacelon:2022wr}.

\begin{definition} \label{def:angles}
Let $(X,\rho)$ be an Alexandrov space. If $\alpha$ and $\beta$ are shortest paths in $X$ starting from the same point $\alpha(0)=\beta(0)=p\in X$ and $\theta_{\alpha,\beta}$ is the function defined in \eqref{eqn:angle}, then the limit
\[
\measuredangle(\alpha,\beta) := \lim_{s,t\to0}\theta_{\alpha,\beta}(s,t)
\]
exists (cf.\ \cite[Proposition 4.3.2]{Burago:2001aa}), is independent of the lower bound $k$ (see the discussion after \cite[Definition 4.6.2]{Burago:2001aa}) and is called the \emph{angle between $\alpha$ and $\beta$}. If the angle is zero, then $\alpha$ and $\beta$ are said to have the \emph{same direction at $p$}. The \emph{space of directions at $p$}, denoted $\Sigma_p(X)$, is the completion with respect to $\measuredangle$ of the collection of equivalence classes of shortest paths starting at $p$ (where two paths are equivalent if they have the same direction at $p$). The \emph{tangent cone at $p$}, denoted $K_p(X)$, is the Euclidean cone over $\Sigma_p(X)$. It consists of an origin $o$ and pairs $(\xi,r)$ with $\xi\in\Sigma_p(X)$ and $r\in(0,\infty)$. The distance between $o$ and $(\xi,r)$ is $r$, and the distance $\rho_{K_p}(a_1,a_2)$ between $a_1=(\xi_1,r_1)$ and $a_2=(\xi_2,r_2)$ is defined so that $\Angle_0(a_1,o,a_2)=\measuredangle(\xi_1,\xi_2)$, or equivalently by the cosine law,
\[
\rho_{K_p}((\xi_1,r_1),(\xi_2,r_2)) = (r_1^2+r_2^2-2r_1r_2\cos\measuredangle(\xi_1,\xi_2))^\frac{1}{2}.
\]
Suppose that $\dim X=n\in\nn$. The \emph{boundary} of $X$ is defined inductively: the boundary of a one-dimensional space (which is necessarily a segment, ray, line or circle) is its topological boundary, and for $n>1$, the boundary of $X$ is the set of points $p\in X$ for which the boundary of $\Sigma_p(X)$ is nonempty.
\end{definition}

\begin{theorem} \label{thm:alexandrov}
Let $(X,\rho)$ be a compact Alexandrov space whose topological dimension is finite and whose curvature is bounded above as well as below. Suppose that either the boundary of $X$ is empty or that $X$ is contractible and the upper curvature bound is at most zero. Then, there exists a classifiable, projectionless $\cs$-algebraic quantum metric Bauer simplex $(\js_X,L_\rho)$ observing $(X,\rho)$. The same is true of $(X,\tilde\rho)$ for any metric $\tilde\rho$ on $X$ that is Lipschitz equivalent to $\rho$.
\end{theorem}

\begin{proof}
If $X$ is one dimensional, then no further assumptions are necessary: $X$ is isometrically isomorphic to either an interval or a circle, and these cases are covered by \cite[Example 4.4.4]{Jacelon:2021vc} and \cite[Theorem 4.4]{Jacelon:2022wr}. So we will assume that $n:=\dim X\ge2$.

First suppose that $X$ is boundaryless. By Nikolaev's theorem \cite[Theorem 10.10.13]{Burago:2001aa}, $X$ possesses a $C^3$-smooth atlas in whose charts $\rho$ can be defined by a Riemannian metric tensor whose coefficients $g_{ij}$ are in $C^{1,\alpha}$ for every $\alpha<1$ (that is, each $g_{ij}$ is differentiable and its derivative is H\"{o}lder continuous with exponent $\alpha$, which is not quite sufficient to guarantee local uniqueness of geodesics but that does not matter for us). We can use this structure to verify the conditions of Theorem~\ref{thm:models}, proceeding exactly as in the proof of \cite[Theorem 4.4]{Jacelon:2022wr}. There, $X$ is assumed to be a smooth Riemannian manifold, but the present degree of smoothness is enough.

Fix distinct points $x_0,x_1\in X$. We will show that, for an arbitrary $y\in X\setminus\{x_0,x_1\}$, there is a bi-Lipschitz path from $x_0$ to $x_1$ that passes through $y$. The idea is to concatenate shortest paths $\gamma_1$ from $x_0$ to $y$ and $\gamma_2$ from $x_1$ to $y$. To ensure that this concatenation is bi-Lipschitz, we need to ensure that $\gamma_1$ and $\gamma_2$ do not intersect away from $y$ and that they approach $y$ in different directions (so that the angle at $y$ is nonzero). Let $\varphi\colon U \to \rr^n$ be a $C^3$ (hence bi-Lipschitz) chart with $y \in U$, and let $B=B_r(\varphi(y))$ be an open ball whose closure is contained in $\varphi(U)$. If $n\ge3$, then we can remove any points of intersection in $X\setminus\varphi^{-1}(B)$ by locally perturbing $\gamma_2$ so that it jumps over $\gamma_1$. If $n=2$, we go around $\gamma_1$ instead of over it; in \cite{Jacelon:2022wr}, this is done inside of a tubular neighbourhood containing $\gamma_1$, but we can do it locally by covering $\gamma_1$ by finitely many charts. With this done, we in particular know that $\gamma_1$ and $\gamma_2$ enter the closure of $\varphi^{-1}(B)$ at different points. Our final adjustment, and the completion of the argument, is to linearly connect the images of these points to $\varphi(y)$ in $B$ and then pull back to $X$ via $\varphi$. By Theorem~\ref{thm:models}, we can find the required $(\js_X,L_\rho)$.

Next, let us assume that $X$ is contractible (hence simply connected) and of curvature at most $\kappa\le0$ (and at least $k\le\kappa$). The Cartan--Hadamard theorem \cite[Theorem 9.2.2]{Burago:2001aa} says that in this case, any two points can be connected by a \emph{unique} shortest path. Fix distinct points $x_0,x_1,y\in X$. Let $\eps>0$. As discussed before Definition~\ref{def:angles}, we can find $y'\in X\setminus\{x_0,x_1\}$ that is $\eps$-close to $y$ and is contained in a neighbourhood $U$ that is bi-Lipschitz equivalent to an open subset of $\rr^n$. We once again concatenate shortest paths $\gamma_1$ from $x_0$ to $y'$ and $\gamma_2$ from $x_1$ to $y'$. By geodesic uniqueness, if $\gamma_1$ and $\gamma_2$ intersect away from $y'$, then they would have to agree on a segment. But geodesics in spaces of curvature at least $k$ do not branch (\cite[Exercise 10.1.2]{Burago:2001aa}), so this is not possible since $x_0\ne x_1$. It follows that $y'$ is the only intersection point of $\gamma_1$ and $\gamma_2$. Exactly as above, we can then perturb within $U$ to secure a bi-Lipschitz path and apply Theorem~\ref{thm:models}.

The final claim made in the theorem follows because Lipschitz equivalent metrics share the same set of tracially Lipschitz elements in $\js_X$.
\end{proof}

\begin{remark} \label{rem:hadamard}
Being a compact, contractible Alexandrov space of nonpositive curvature is a somewhat restrictive set of conditions. Examples include compact, convex subsets of Euclidean or hyperbolic space, or more generally of complete Riemannian manifolds of nonpositive sectional curvature (equipped with their inherited intrinsic metrics). In general, a consequence of the Cartan--Hadamard theorem (see \cite[Theorem 9.2.9]{Burago:2001aa}) is that a complete, simply connected length space of curvature at most $\kappa\le0$ is $CAT(\kappa)$, that is, has curvature at most $\kappa$ `in the large', meaning that the open neighbourhoods $U$ appearing in Definition~\ref{def:alexandrov} can be taken to be the whole space. This will become an important observation in the proof of Corollary~\ref{cor:contractible}.
\end{remark}

In Section~\ref{section:arachnid}, we use models $\js_X$ observing metric balls $\ball^k$ or odd metric spheres $\sphere^{2k-1}$ as building blocks within the category of tracially lyriform $\cs$-algebras (defined in Section~\ref{section:limits}). For this task , we need to know the $K$-theory of $\js_{\ball^k}$ and $\js_{\sphere^{2k-1}}$.

\begin{proposition} \label{prop:ktheory}
Let $X_{p,q}$ be a generalised dimension drop algebra associated with a pair of coprime integers $p,q\in\nn$, a connected compact Hausdorff space $X$ and distinct points $x_0,x_1\in X$. Let $\rank$ denote the unique state on the ordered abelian group $(K_0(C(X)),K_0(C(X))_+,[1_{C(X)}])$, which sends the class of a projection over $X$ to its constant fibrewise rank. Then,
\[
K_0(X_{p,q})\cong \rank^{-1}(pq\zz) \quad\text{and}\quad K_1(X_{p,q}) \cong K_1(C(X)) \cong K^1(X).
\]
In particular:
\begin{enumerate}[(1)]
\item \label{it:mv1} if $X$ is contractible, then $(K_0(X_{p,q}),K_1(X_{p,q})) \cong (\zz,0)$;
\item \label{it:mv2} if $n\in\nn$ is odd, then $(K_0((\sphere^n)_{p,q}),K_1((\sphere^n)_{p,q})) \cong (\zz,\zz)$.
\end{enumerate}
In both cases, $K_0$ is generated by the class of the unit $1_{X_{p,q}}$.
\end{proposition}

\begin{proof}
The $\cs$-algebra $X_{p,q}$ can be written as a pullback
\[
\begin{tikzcd}
X_{p,q} \arrow[r,"\pi_1"] \arrow[d,"\pi_2"] & C(X,M_{pq}) \arrow[d,"\ev=\ev_{x_0}\oplus\ev_{x_1}"]\\
M_{p} \oplus M_{q} \arrow[r,"\varphi=\varphi_1\oplus\varphi_2"] & M_{pq} \oplus M_{pq}
\end{tikzcd}
\]
with $\varphi_0(a,b)=(a\otimes1_q,0)$ and $\varphi_1(a,b)=(0,1_p\otimes b)$ for $(a,b)\in M_{p} \oplus M_{q}$. It follows from the Tietze extension theorem (applied to the entries of $M_{pq}\cong M_p\otimes M_q$) that the evaluation map $\ev=\ev_{x_0}\oplus\ev_{x_1}$ is surjective. We can therefore apply the operator $K$-theoretic Mayer--Vietoris theorem (see \cite{Hilgert:1986aa}) to obtain the six-term exact sequence
\[
\begin{tikzcd}
K_0(X_{p,q}) \arrow[r,"(\pi_1)_*\oplus(\pi_2)_*"] & K_0(C(X,M_{pq})) \oplus K_0(M_{p} \oplus M_{q}) \arrow[r,"\ev_*-\varphi_*"] & K_0(M_{pq} \oplus M_{pq})\arrow[d,"\exp"]\\
K_1(M_{pq} \oplus M_{pq}) \arrow[u] & K_1(C(X,M_{pq})) \oplus K_1(M_{p} \oplus M_{q}) \arrow[l,"\ev_*-\varphi_*"] & K_1(X_{p,q}) \arrow[l,"(\pi_1)_*\oplus(\pi_2)_*"]
\end{tikzcd}
\]
which reduces to
\[
\begin{tikzcd}[column sep = large]
K_0(X_{p,q}) \arrow[hookrightarrow,r,"(\pi_1)_*\oplus(\pi_2)_*"] & K^0(X) \oplus \zz^2 \arrow[r,"\ev_*-\varphi_*"] & \zz^2 \arrow[r,"\exp"] & K_1(X_{p,q}) \arrow[twoheadrightarrow,r] & K^1(X).
\end{tikzcd}
\]
The map $\ev_*-\varphi_*$ sends $(x,y,z)\in K^0(X) \oplus \zz^2$ to $(\rank(x)-qy,\rank(x)-pz)$. Since $p$ and $q$ are coprime, it therefore follows from B\'{e}zout's lemma and surjectivity of $\rank$ that $\ev_*-\varphi_*$ is surjective. We then have from exactness that $\exp=0$ and $K_1(X_{p,q}) \cong K^1(X)$. As for $K_0$, we have
\begin{align*}
K_0(X_{p,q}) &\cong \ker (\ev_*-\varphi_*)\\
&\cong \{(x,y,z) \in K^0(X)\oplus\zz^2 \mid (\exists m\in\zz) \rank(x)=mpq,y=mp,z=mq\}\\
&\cong \rank^{-1}(pq\zz).
\end{align*}
If $X$ is contractible, then $(K_0(C(X)),K_1(C(X)))\cong(\zz,0)$, giving \eqref{it:mv1}. For odd $n$, we have $(K_0(C(\sphere^n)),K_1(C(\sphere^n)))\cong(\zz,\zz)$, which gives \eqref{it:mv2}. In both cases, $K_0(X_{p,q})$ is generated by $[1]=\rank^{-1}(\{pq\})$.
\end{proof}

Theorem~\ref{thm:alexandrov} and Proposition~\ref{prop:ktheory} allow us to observe balls and spheres while also keeping track of $K$-theory (see Corollary~\ref{cor:ballsandspheres}). In general, Alexandrov or even Riemannian spheres can have surprising properties like everywhere-negative Ricci curvature (see \cite{Lohkamp:1994aa}). These cases are certainly of interest, for example to the central endeavour of \cite{Jacelon:2024aa}, as they correspond to the isometry group of the manifold being finite (see \cite[Theorem D]{Lohkamp:1994aa} and \cite[Example 5.9]{Jacelon:2024aa}). But in the present article, we would like to imagine our spheres as boundaries of compact convex sets. For the model building in Section~\ref{section:arachnid}, we will in fact use \emph{simplicial} spheres, whose metrics are induced from the $\ell_1$-norm (see Theorem~\ref{thm:arachnid} and Definition~\ref{def:simplicial}). These simplicial spaces are not Alexandrov, but since their metrics are equivalent to Euclidean, they are still covered by Theorem~\ref{thm:alexandrov}.

\begin{corollary} \label{cor:ballsandspheres}
Let $k\in\nn$ and let $(\ball^{k},\omega_k)$ and $(\sphere^{2k-1},\rho_k)$ be metric spaces Lipschitz equivalent to the usual Euclidean ball and odd Euclidean sphere, respectively. Then, there exist classifiable, projectionless $\cs$-algebraic quantum metric Bauer simplices $(\js_{\sphere^{2k-1}},L_{\rho_k})$ and $(\js_{\ball^{k}},L_{\omega_k})$ observing $(\ball^k,\omega_k)$ and $(\sphere^{2k-1},\rho_k)$ such that
\[
(K_0(\js_{\ball^{k}}),K_0(\js_{\ball^{k}})_+,[1_{\js_{\ball^{k}}}],K_1(\js_{\ball^{k}})) \cong (\zz,\nn,1,0)
\]
and
\[
(K_0(\js_{\sphere^{2k-1}}),K_0(\js_{\sphere^{2k-1}})_+,[1_{\js_{\sphere^{2k-1}}}]) \cong (\zz,\nn,1).
\]
\end{corollary}

\begin{remark} \label{rem:dimdrop}
By classification, $\js_{\ball^{k}}$ is isomorphic to a simple inductive limit of prime dimension drop algebras. Viewed another way, we have shown how to equip such an inductive limit $\varinjlim \mathbb{I}_{p_n,q_n}$ with quantum metric Bauer structure whenever the extreme boundary of its trace space is homeomorphic to a ball $\ball^{k}$ (including the Hilbert cube $\ball^{\infty}$). By intertwining with suitable sequences $\js_{\ball^{2k}}$ and $\js_{\sphere^{2k-1}}$ (and with the help of classification), we will see in Section~\ref{section:arachnid} that an arbitrary simple limit $\varinjlim \mathbb{I}_{p_n,q_n}$ can be furnished with the `lyriform' tracial metric structure that we develop in Section~\ref{section:lyriform}.
\end{remark}

\section{Tracially lyriform \texorpdfstring{$\cs$}{C*}-algebras} \label{section:lyriform}

To arrive at the notion of a tracially lyriform $\cs$-algebra, we shift focus from the Lipschitz seminorm on $A$ to the associated metric $\rho$ on $T(A)$, and we correspondingly weaken the definition of a $\cs$-algebraic quantum metric Choquet simplex by:
\begin{items}
\item no longer requiring that $\rho$ should induce the $\ws$-topology;
\item withdrawing the demand that $\rho$ be midpoint-balanced;
\item relaxing convexity of $\rho$ to effective quasiconvexity.
\end{items}
Here, $\rho$ is \emph{convex} if
\begin{equation} \label{eqn:convex}
\rho\left(\sum_{i=1}^n\lambda_i\sigma_i,\sum_{i=1}^n\lambda_i\tau_i\right) \le \sum_{i=1}^n\lambda_i\rho(\sigma_i,\tau_i)
\end{equation}
for every natural number $n$, traces $\sigma_1,\dots,\sigma_n,\tau_1,\dots\tau_n\in T(A)$ and real numbers $\lambda_1,\dots,\lambda_n\in[0,1]$ with $\sum_{i=1}^n\lambda_i=1$. It is \emph{midpoint-balanced} if for every $\sigma,\tau,\sigma',\tau' \in T(A)$,
\begin{equation} \label{eqn:midpt}
\frac{\sigma+\tau'}{2} = \frac{\sigma'+\tau}{2} \implies\rho(\sigma,\tau)=\rho(\sigma',\tau').
\end{equation}
These features hold if $\rho=\rho_L$ is induced from a seminorm $L$ via \eqref{eqn:rho}, and they distinguish $W_1$ from the other Wasserstein metrics $W_\fp$ on probability spaces (see Section~\ref{section:wass}). These higher $\fp$-Wasserstein metrics are at least effectively quasiconvex \eqref{eqn:qc}, a property that is intended to represent a useful degree of compatibility between the metric $\rho$ and the convex set $T(A)$.

\begin{definition} \label{def:lyriform}
Let $A$ be a separable, unital $\cs$-algebra. Suppose that the trace space $T(A)$ is equipped with a metric $\rho$ such that:
\begin{enumerate}[(1)]
\item the diameter of the metric space $(T(A),\rho)$ is finite;
\item $\rho$ is lower semicontinuous with respect to the product topology on $T(A)\times T(A)$ induced by the $\ws$-topology on $T(A)$;
\item $\rho$ is \emph{effectively quasiconvex}, meaning that there exists an increasing function $h\colon [0,\diam(T(A),\rho)] \to [0,\infty)$ such that $\lim_{x\to0}h(x)=0$ and
\begin{equation} \label{eqn:qc}
\rho\left(\sum_{i=1}^n\lambda_i\sigma_i,\sum_{i=1}^n\lambda_i\tau_i\right) \le h(\max\{\rho(\sigma_i,\tau_i) \mid i=1,\dots,n\})
\end{equation}
for every $n\in\nn$, $\sigma_1,\dots,\sigma_n,\tau_1,\dots\tau_n\in T(A)$ and $\lambda_1,\dots,\lambda_n\in[0,1]$ with $\sum_{i=1}^n\lambda_i=1$;
\item tracially $\rho$-Lipschitz elements are dense in $A$, that is, the associated seminorm $L_\rho \colon A \to [0,\infty]$ defined by
\[
L_\rho(a) = \sup\left\{\frac{|\sigma(a)-\tau(a)|}{\rho(\sigma,\tau)} \mid \sigma\ne \tau \in T(A)\right\}
\]
is finite on a norm-dense subset of $A$.
\end{enumerate}
Then we say that $(A,\rho)$ is a \emph{tracially lyriform $\cs$-algebra}. In this case, a \emph{core} for $(A,\rho)$ is a norm-bounded, $\sigma$-compact subset $\mathcal{L}(A)$ of $A_{sa}$ such that
\begin{equation} \label{eqn:core}
\overline{\mathcal{L}(A) + \rr1 + A_0}^{\|\cdot\|} = \{a\in A_{sa} \mid L_\rho(a)\le 1\}
\end{equation}
where $A_0$ is as in \eqref{eqn:a0}. If a core is norm-compact, we call it a \emph{nucleus}.
\end{definition}

The following proposition tells us that cores always exist and indicates these structures' significance.

\begin{proposition} \label{prop:core}
The following hold for a tracially lyriform $\cs$-algebra $(A,\rho)$.
\begin{enumerate}[(1)]
\item \label{it:core0} The topology induced by $\rho$ is stronger than the $\ws$-topology on $T(A)$ (that is, $\rho$-convergent sequences are also $\ws$-convergent).
\item \label{it:core1} $(A,\rho)$ admits a core.
\item \label{it:core2} If $B$ is a unital $\cs$-algebra and $\varphi,\psi \colon A \to B$ are unital $^*$-homomorphisms, then $T(\varphi) = T(\psi) \colon T(B) \to T(A)$ if and only if $\varphi$ and $\psi$ tracially agree on a(ny) core $\mathcal{L}(A)$ for $(A,\rho)$, that is, if and only if $\tau(\varphi(a)) = \tau(\psi(a))$ for every $a\in \mathcal{L}(A)$ and $\tau \in T(B)$.
\end{enumerate}
\end{proposition}

\begin{proof}
Dense finiteness of $L_\rho$ implies \eqref{it:core0} and, in conjunction with \eqref{eqn:core}, also \eqref{it:core2}. More precisely, if $\tau_i\to\tau$ is a $\rho$-convergent sequence in $T(A)$ and $b\in A$ is such that $L_\rho(b)\le K$ for some $K\in(0,\infty)$, then $|\tau_i(b)-\tau(b)| \le K \rho(\tau_i,\tau)$ for every $i$, whence $\tau_i(b)\to\tau(b)$. Since such elements $b$ are dense in $A$, it follows that $\tau_i(a)\to\tau(a)$ for every $a\in A$, which means that $\tau_i\to\tau$ in the $\ws$-topology. For unital $^*$-homomorphisms $\varphi,\psi \colon A \to B$ and $\tau \in T(B)$, if $\tau(\varphi(a)) = \tau(\psi(a))$ for every $a\in \mathcal{L}(A)$, then \eqref{eqn:core} implies that $\tau(\varphi(a)) = \tau(\psi(a))$ for every $a\in A_{sa}$ with $L_\rho(a)\le 1$. Since these latter elements densely span $A$, it follows that $T(\varphi) = T(\psi)$.

Finally, \eqref{it:core1} follows from separability of $A$. Indeed, the image of $\{a\in A_{sa} \mid L_\rho(a)\le 1\}$ in the quotient $A^q=A_{sa}/A_0$ is also separable (with respect to the quotient norm $\|\cdot\|_q$), so admits a countable dense subset $\{x_n\}_{n=1}^\infty$. Write $\|\cdot\|_\infty$ for the uniform norm on the set of continuous functions $(T(A),\rho) \to \rr$. Let $r=r_{(T(A),\rho)}<\infty$. For every $n\in\nn$, the image of $x_n$ as a function $T(A) \to \rr$ is contained in an interval of length at most $2r$; taking $\lambda_n$ to be the midpoint of this interval and setting $y_n=x_n-\lambda_n$, we have $\|y_n\|_\infty\le r$. By \cite[Proposition 2.1]{Carrion:wz}, we can find preimages $(a_n)_{n\in\nn} \subseteq A_{sa}$ of $(y_n)_{n\in\nn}$, such that $\|a_n\|<r+1$. For every $n\in\nn$, let $F_n=\{a_1,\dots,a_n\}$, and then define $\mathcal{L}(A)=\bigcup_{n\in\nn}F_n$. Then, $\mathcal{L}(A)$ is a $\sigma$-finite (hence $\sigma$-compact) subset of the $(r+1)$-ball of $A_{sa}$. Let $\eps>0$ and let $a\in A_{sa}$ with $L_\rho(a)\le 1$. By density of $\{x_n\}_{n=1}^\infty$, there exists $n\in\nn$ such that $\|q(a)-x_n\|_q < \frac{\eps}{2}$, and then by definition of the quotient norm, there exists $a_0\in A_0$ such that
\begin{align*}
\|a-(a_n+\lambda_n\cdot1+a_0)\| &< \|q(a)-q(a_n+\lambda_n\cdot1)\|_q +\frac{\eps}{2}\\
&= \|q(a)-x_n\|_q +\frac{\eps}{2}\\
&< \eps,
\end{align*}
so \eqref{eqn:core} holds.
\end{proof}

\begin{definition} \label{def:compact}
We say that a tracially lyriform $\cs$-algebra $(A,\rho)$ is \emph{compact} or \emph{complete} or \emph{strictly intrinsic} if these properties hold for $(T(A),\rho)$. If $(A,\rho)$ is compact, then we say it is \emph{of Bauer type} if $\partial_e(T(A))$ is a $\rho$-closed subset of $T(A)$.
\end{definition}

We will see examples of compact or complete strictly intrinsic tracially lyriform $\cs$-algebras in Section~\ref{section:wass}. For now, we note that $(A,\rho)$ is compact precisely when $\rho$ induces the $\ws$-topology on the trace space. In particular, a compact tracially lyriform $\cs$-algebra $(A,\rho)$ is of Bauer type if and only if $T(A)$ is a Bauer simplex with respect to the $\ws$-topology. Furthermore, we can characterise those tracially lyriform $\cs$-algebras that are in fact $\cs$-algebraic quantum metric Choquet simplices. 

\begin{theorem} \label{thm:compact}
Consider the following properties of a tracially lyriform $\cs$-algebra $(A,\rho)$ for which the Cuntz--Pedersen quotient $A^q$ is tracially ordered:
\begin{enumerate}[(1)]
\item \label{it:ch2} $(A,L_\rho)$ is a $\cs$-algebraic quantum metric Choquet simplex with $\rho_{L_\rho}=\rho$;
\item \label{it:ch3} $\rho$ induces the $\ws$-topology on $T(A)$;
\item \label{it:ch1} $(T(A),\rho)$ is compact;
\item \label{it:ch4} $(A,\rho)$ admits a nucleus $\mathcal{L}(A)$ such that, for some $r>0$,
\[
q(\mathcal{L}(A)) = \{x\in A^q \mid L_\rho(x)\le 1,\|x\|_q\le r\}.
\]
\end{enumerate}
Then,
\[
\eqref{it:ch2} \implies \eqref{it:ch3} \iff \eqref{it:ch1} \implies \eqref{it:ch4}.
\]
Conversely, $\eqref{it:ch4} \implies \eqref{it:ch3}$ if $\rho_{L_\rho}=\rho$, and $\eqref{it:ch3} \implies \eqref{it:ch2}$ if and only if $\rho$ is convex and midpoint-balanced.
\end{theorem}

\begin{proof}
$\eqref{it:ch2}\iff\eqref{it:ch3}$: The implication $\eqref{it:ch2}\implies\eqref{it:ch3}$ holds by definition. Since $(A,L_\rho)$ is a $\cs$-algebraic quantum metric Choquet simplex if and only if $(A^q,L_\rho)$ is a compact quantum metric space, the reverse implication under the assumption that $\rho$ is convex and midpoint-balanced (necessary conditions for it to be induced by a Lipschitz seminorm) is given by \cite[Theorem 9.11 and Theorem 9.8]{Rieffel:1999aa}. 

$\eqref{it:ch3}\iff\eqref{it:ch1}$: If $\rho$ induces the $\ws$-topology, then $(T(A),\rho)$ is compact because $T(A)$ is $\ws$-compact. Conversely, if $(T(A),\rho)$ is compact, then, by Proposition~\ref{prop:core}\eqref{it:core0}, the identity map is a continuous bijection from $(T(A),\rho)$ to $(T(A),\ws)$, and is therefore a homeomorphism since $(T(A),\rho)$ is compact and $(T(A),\ws)$ is Hausdorff.

$\eqref{it:ch3}\iff\eqref{it:ch4}$: The existence of a suitable nucleus, assuming \eqref{it:ch3}, follows from \cite[Theorem 2.12]{Jacelon:2024aa}, which is based on \cite[Theorem 1.9]{Rieffel:1998ww}. Conversely, if such a nucleus exists and $\rho_{L_\rho}=\rho$, then $\rho$ induces the $\ws$-topology on $T(A)$ as in the proof of \cite[Theorem 1.8]{Rieffel:1998ww}.
\end{proof}

Most of the tracially lyriform $\cs$-algebras that we encounter in this article are commutative in origin. Indeed, the starting point in the construction of the quantum metric Bauer simplices of Section~\ref{section:qmcs} (generalised to the tracial Wasserstein spaces of Section~\ref{section:wass} and used as building blocks in Section~\ref{section:arachnid}) is a compact metric space. For a truly noncommutative example, we turn to \cite{Anshu:2025aa}.

The underlying $\cs$-algebra is the quantum permutation group $C(S_n^+)$ introduced in \cite{Wang:1998aa}. Anshu, Jekel and Landry equip the trace space $T(C(S_n^+))$ with metrics called \emph{free Hamming distances}, based on the Biane--Voiculescu notion of a tracial coupling. Such a coupling between traces $\tau_1$ and $\tau_2$ on a $\cs$-algebra $A$ means a tuple $(M,\tau,\alpha_1,\alpha_2)$, where $(M,\tau)$ is a tracial von Neumann algebra and, for $i=1,2$, $\alpha_i\colon A\to M$ is a $^*$-homomorphism with $\tau\circ\alpha_i=\tau_i$. The two versions of free Hamming distance $W_{H,\wedge}(\tau_1,\tau_2)$  and $W_{H,1}(\tau_1,\tau_2)$ are defined in \cite[Definitions 1.2 and 1.3]{Anshu:2025aa} as infima over all tracial couplings of expressions involving $(M,\tau,\alpha_1,\alpha_2)$ and the canonical generating projections $\{u_{i,j}\}_{1\le i,j\le n}$ of $C(S_n^+)$, namely:
\begin{equation} \label{eqn:hammingwedge}
W_{H,\wedge}(\tau_1,\tau_2) = \inf_{(M,\tau,\alpha_1,\alpha_2)} \left[1-\frac{1}{n}\sum_{i,j=1}^n\tau(\alpha_1(u_{i,j}) \wedge\alpha_2(u_{i,j}))\right]
\end{equation}
and
\begin{equation} \label{eqn:hamming1}
W_{H,1}(\tau_1,\tau_2) = \inf_{(M,\tau,\alpha_1,\alpha_2)} \frac{1}{n}\sum_{i,j=1}^n\frac{1}{2}\tau(|\alpha_1(u_{i,j}) - \alpha_2(u_{i,j})|).
\end{equation}

\begin{theorem}[\cite{Anshu:2025aa}] \label{thm:free}
$(C(S_n^+),W_{H,\wedge})$ and $(C(S_n^+),W_{H,1})$ are tracially lyriform $\cs$-algebras.
\end{theorem}

\begin{proof}
We must verify the conditions of Definition~\ref{def:lyriform}. It follows from their definitions \eqref{eqn:hammingwedge} and \eqref {eqn:hamming1} that $\diam(T(C(S_n^+)),W_{H,\wedge})$ and $\diam(T(C(S_n^+)),W_{H,1})$ are at most $1$. Both expressions are also convex, so they satisfy \eqref{eqn:qc} with $h$ equal to the identity map. To see this, let $m\in\nn$, $\sigma_1,\dots,\sigma_m,\tau_1,\dots\tau_m\in T(C(S_n^+))$ and $\lambda_1,\dots,\lambda_m\in[0,1]$ with $\sum_{l=1}^m\lambda_l=1$. By \cite[Theorem 1.4(3)]{Anshu:2025aa}, for every $l$, the infimum defining $W_{H,\wedge}(\sigma_l,\tau_l)$ is achieved by a tracial coupling $(M_l,\theta_l,\alpha_{l,1},\alpha_{l,2})$ of $\sigma_l$ and $\tau_l$ (and similarly for $W_{H,1}(\sigma_l,\tau_l)$). Let $M=M_1\oplus\dots\oplus M_m$, and define the trace $\tau\in T(M)$ by $\tau(a_1, \dots, a_m)=\lambda_1\theta_1(a_1)+\dots+\lambda_m\theta_m(a_m)$ and the $^*$-homomorphisms $\alpha_1,\alpha_2 \colon C(S_n^+) \to M$ by $\alpha_k = \alpha_{1,k}\oplus\dots\oplus\alpha_{m,k}$ for $k=1,2$. Then, $\tau \circ \alpha_1 = \sum_{l=1}^m \lambda_l\sigma_l$ and $\tau \circ \alpha_2 = \sum_{l=1}^m \lambda_l\tau_l$, so $(M,\tau,\alpha_1,\alpha_2)$ is a tracial coupling of $\sum_{l=1}^m \lambda_l\sigma_l$ and $\sum_{l=1}^m \lambda_l\tau_l$. It follows that
\begin{align*}
W_{H,\wedge}\left(\sum_{l=1}^m \lambda_l\sigma_l,\sum_{l=1}^m \lambda_l\tau_l\right) &\le 1-\frac{1}{n}\sum_{i,j=1}^n\tau(\alpha_1(u_{i,j}) \wedge\alpha_2(u_{i,j}))\\
&= 1-\frac{1}{n}\sum_{i,j=1}^n\sum_{l=1}^m\lambda_l\theta_l(\alpha_{l,1}(u_{i,j}) \wedge\alpha_{l,2}(u_{i,j}))\\
&= \sum_{l=1}^m\lambda_l \left(1-\frac{1}{n}\sum_{i,j=1}^n \theta_l(\alpha_{l,1}(u_{i,j}) \wedge\alpha_{l,2}(u_{i,j})) \right)\\
&= \sum_{l=1}^m\lambda_l W_{H,\wedge}(\sigma_l,\tau_l)\\
&\le \max_{1\le l\le m} W_{H,\wedge}(\sigma_l,\tau_l).
\end{align*}
A similar argument, using the fact that
\[
\tau(|\alpha_1(u_{i,j})-\alpha_2(u_{i,j})|) = \sum_{l=1}^m\lambda_l \theta_l(|\alpha_{l,1}(u_{i,j})-\alpha_{l,2}(u_{i,j})|),
\]
shows convexity of $W_{H,1}$. Finally, lower semicontinuity and density of Lipschitz elements for both distances are provided by \cite[Theorem 1.4]{Anshu:2025aa} and \cite[Proposition 1.7]{Anshu:2025aa}, respectively.
\end{proof}

We end this section with an example from topological dynamics. Let $(X,\rho)$ be a compact metric space and $T\colon X\to X$ a homeomorphism. The article \cite{Babel:2025aa} introduces a metric $\bar\rho$ on the space $\prob_T(X)$ of $T$-invariant measures on $X$ as an extension of Ornstein's $d$-bar metric (which is defined for shifts). Its definition is
\begin{equation} \label{eqn:rhobar}
\bar\rho(\mu,\nu) = \inf_{\pi\in J(\mu,\nu)} \int_{X\times X} \rho(x,y)d\pi(x,y)
\end{equation}
where $\pi\in J(\mu,\nu)$ if and only if it is a \emph{joining} of $\mu$ and $\nu$, that is, a $T\times T$-invariant measure that satisfies $(p_1)_*\pi=\mu$ and $(p_2)_*\pi=\nu$ (where $p_1,p_2\colon X\times X \to X$ are the coordinate projections). As observed in \cite{Babel:2025aa}, $\pi\in J(\mu,\nu)$ is a convex, $\ws$-closed subset of $\prob_{T\times T}(X\times X)$, and the infimum in \eqref{eqn:rhobar} is attained, in fact at an ergodic joining $\pi$.

\begin{theorem} \label{thm:dynamics}
If the topological dynamical system $(X,T)$ is free, then $(C(X)\rtimes_T\zz,\bar\rho)$ is a tracially lyriform $\cs$-algebra.
\end{theorem}

\begin{proof}
Here, we are using the fact that $T(C(X)\rtimes_T\zz)\cong \prob_T(X)$ for free actions (see \cite[Theorem 11.1.22]{Giordano:2018wp}). More precisely, every $\mu\in\prob_T(X)$ uniquely extends to a trace $\tau_\mu$ on $C(X)\rtimes_T\zz$ via the standard conditional expectation $C(X)\rtimes_T\zz \to C(X)$ (so in particular, every trace is of this form). On finite sums $a=\sum_{n\in\zz}a_nu^n$, with $a_n\in C(X)$ and $u$ denoting the unitary implementing the action $T$, we have $\tau_\mu(a) = \int_X a_0\,d\mu$. If $a_0\in\lip(X,\rho)$, then for every $\mu,\nu\in\prob_T(X)$ we have
\[
|\tau_\mu(a) - \tau_\nu(a)| = \left|\int_Xa_0\,d\mu - \int_Xa_0\,d\nu\right| \le W_1(\mu,\nu) \le \bar\rho(\mu,\nu)
\]
where $W_1(\mu,\nu)$ is the $1$-Wasserstein distance \eqref{eqn:w1}, which as discussed before \cite[Remark 7.1]{Babel:2025aa} is dominated by $\bar\rho$. It follows that $\bar\rho$-Lipschitz elements are dense in $C(X)\rtimes_T\zz$. By definition \eqref{eqn:rhobar}, the diameter of $(\prob_T(X),\bar\rho)$ is at most the diameter of $(X,\rho)$. To see that $\bar\rho$ is convex, it suffices to observe that $\sum_{i=1}^n\lambda_i\pi_i\in J(\sum_{i=1}^n\lambda_i\mu_i,\sum_{i=1}^n\lambda_i\nu_i)$ whenever $\mu_i,\nu_i\in\prob_T(X)$, $\pi_i\in J(\mu_i,\nu_i)$ and $\lambda_i\in[0,1]$ with $\sum_{i=1}^n\lambda_i=1$. Choosing the $\pi_i$ to witness the infima defining $\bar\rho(\mu_i,\nu_i)$ then gives
\[
\bar\rho\left(\sum_{i=1}^n\lambda_i\mu_i,\sum_{i=1}^n\lambda_i\nu_i\right) \le \sum_{i=1}^n\lambda_i\int_{X\times X} \rho(x,y)d\pi_i(x,y) = \sum_{i=1}^n\lambda_i\bar\rho(\mu_i,\nu_i).
\]
To see that $\bar\rho$ is $\ws$-lower semicontinuous, suppose that $((\mu_i,\nu_i))$ is a sequence in $\prob_T(X)\times\prob_T(X)$ that $\ws$-converges to $(\mu,\nu)$ and such that, for some $K\ge0$, $\bar\rho(\mu_i,\nu_i)\le K$ for every $i$. For each $i$, let $\pi_i$ be a joining of $\mu_i$ and $\nu_i$ attaining the infimum in the definition of $\bar\rho(\mu_i,\nu_i)$. By passing to a subsequence, we may assume that the sequence $(\pi_i)$ also $\ws$-converges to some $T\times T$-invariant measure $\pi$. By continuity of the coordinate maps, $\pi$ is a joining of $\mu$ and $\nu$. It follows that
\begin{align*}
\bar\rho(\mu,\nu) \le \int_{X\times X} \rho(x,y)d\pi(x,y) = \lim_{i\to\infty} \int_{X\times X} \rho(x,y)d\pi_i(x,y) =  \lim_{i\to\infty} \bar\rho(\mu_i,\nu_i) \le K
\end{align*}
which means that $\bar\rho$ is lower semicontinuous.
\end{proof}

\section{Tracial Wasserstein spaces} \label{section:wass}

The space $\prob(X) \cong T(C(X))$ of Borel probability measures on a compact metric space $(X,\rho)$ can be equipped with a variety of distances (see \cite{Gibbs:2002aa}). The $\fp$-Wasserstein metrics $W_\fp=W_{\fp,\rho}$, $\fp\in[1,\infty]$, are of particular interest to us as they provide $C(X)$ with tracial lyriform structure. Let us recall their definitions and some of their basic properties. 
(For a discussion of the historical accuracy of the terminology, see the bibliographical notes appended to \cite[Chapter 6]{Villani:2009aa}.) For $\fp\in[1,\infty)$, the $\fp$-Wasserstein distance between $\mu,\nu\in\prob(X)$ is
\begin{equation} \label{eqn:wp}
W_\fp(\mu,\nu) = \inf_{\pi\in\Pi(\mu,\nu)} \left(\int_{X\times X} \rho(x,y)^pd\pi(x,y)\right)^\frac{1}{\fp}
\end{equation}
where $\Pi(\mu,\nu)$ denotes the set of Borel probability measures on $X\times X$ with marginals $\mu$ and $\nu$ (that is, with $p_1,p_2\colon X\times X \to X$ the two coordinate projections, $\pi\in\Pi(\mu,\nu)$ if and only if $(p_1)_*\pi=\mu$ and $(p_2)_*\pi=\nu$). The $1$-Wasserstein distance admits the dual Kantorovich--Rubinstein formulation
\begin{equation} \label{eqn:w1}
W_1(\mu,\nu) = \sup\left\{\left|\int_Xf\,d\mu - \int_Xf\,d\nu\right| \mid f\in\lip(X,\rho)\right\}.
\end{equation}
The $\infty$-Wasserstein distance $W_\infty$ between $\mu,\nu\in\prob(X)$ is
\begin{equation} \label{eqn:winf1}
W_\infty(\mu,\nu) = \inf\{r>0 \mid(\forall\, U\subseteq X\text{ Borel}) \:\mu(U) \le \nu(U_r)\}.
\end{equation}
It dominates the \emph{L\'{e}vy--Prokhorov metric}
\begin{equation} \label{eqn:lp1}
\lp(\mu,\nu) = \inf\{r>0 \mid(\forall\, U\subseteq X\text{ Borel})\: \mu(U) \le \nu(U_r) + r\}.
\end{equation}
Note that `Borel' can be replaced by `open' or by `closed' in both \eqref{eqn:winf1} and \eqref{eqn:lp1}. Unlike $\lp$, the $W_\infty$-topology is in general \emph{stronger} than the $\ws$-topology on $\mathcal{M}(X)$ (see Remark~\ref{rem:compact}), but $W_\infty$ is indeed $\ws$-lower semicontinuous (a consequence of \eqref{eqn:winf2} below). We summarise the various interrelationships between these distances in the following proposition.

\begin{proposition} \label{prop:wass}
Let $(X,\rho)$ be a nonempty compact metric space. The Wasserstein metrics on $\prob(X)$ all agree with $\rho$ on $X \cong \partial_e\prob(X)$ and have diameter equal to $\diam(X,\rho)$, and the same is true of $\lp$ if $\diam(X,\rho)\le1$. All of these distances are effectively quasiconvex and all share the same set of Lipschitz elements within $C(X)$, namely, $\lipo(X,\rho)$. All except $W_\infty$ induce the $\ws$-topology on $\prob(X)$. For $1\le \fp \le \fq < \infty$, the Wasserstein metrics  $W_1$, $W_\fp$ and $W_{\fq}$ are related via
\begin{equation} \label{eqn:wincrease}
W_\fp(\mu,\nu) \le W_{\fq}(\mu,\nu) \le \diam(X,\rho)^{1-\frac{1}{\fq}} \cdot W_1(\mu,\nu)^\frac{1}{\fq}.
\end{equation}
The $\infty$-Wasserstein distance $W_\infty$ is related to the other Wasserstein metrics via
\begin{equation} \label{eqn:winf2}
W_\infty(\mu,\nu) = \lim_{\fp\to\infty}W_\fp(\mu,\nu).
\end{equation}
The $1$-Wasserstein distance is related to the L\'{e}vy--Prokhorov metric via
\begin{equation} \label{eqn:lp2}
\lp(\mu,\nu)^2 \le W_1(\mu,\nu) \le (\diam(X,\rho)+1) \cdot \lp(\mu,\nu).
\end{equation}
\end{proposition}

\begin{proof}
To see that $W_\fp$ induces the $\ws$-topology whenever $\fp\in[1,\infty)$, see \cite[Proposition 4]{Givens:1984to}.  For $\fp\in[1,\infty)$ and $x,y\in X$, the equality $W_\fp(\delta_x,\delta_y)=\rho(x,y)$ holds because $\delta_{(x,y)}$ is the unique element of the set $\Pi(\delta_x,\delta_y)$ appearing in \eqref{eqn:wp}. It is easily verified from their definitions \eqref{eqn:winf1} and \eqref{eqn:lp1} that $W_\infty$ also agrees with $\rho$ on $X$, and that
\[
\lp(\delta_x,\delta_y) = \min\{\rho(x,y),1\} \le \rho(x,y).
\]
This implies the claim about the diameter of $\prob(X)$. It also follows that elements of $C(X)$ that are tracially $1$-Lipschitz relative to any of these distances must be elements of $\lip(X)$. Conversely, by \eqref{eqn:w1}, the first half of \eqref{eqn:wincrease} (which can be found in \cite[Proposition 3(2)]{Givens:1984to}) and \eqref{eqn:winf2} (which can be found in \cite[Proposition 3(3)]{Givens:1984to}), elements of $\lip(X)$ are $1$-Lipschitz relative to any of the $W_\fp$. By \eqref{eqn:lp2} (which can be found in \cite[Theorem 2]{Gibbs:2002aa}), elements of $\lip(X)$ are also $(\diam(X,\rho)+1)$-Lipschitz relative to $\lp$. In summary, the set of tracially Lipschitz elements of $C(X)$ is the same in all cases, namely $\lipo(X)$.

For the remaining inequality, the second half of \eqref{eqn:wincrease}, fix $\fp\in[1,\infty)$ and write $\delta=\diam(X,\rho)$. Let $\mu,\nu\in\prob(X)$ and let $\pi$ be a $W_1$-optimal transference plan between $\mu$ and $\nu$, that is, $\pi\in\Pi(\mu,\nu)$ with $W_1(\mu,\nu) = \int_{X\times X} d(x,y)d\pi(x,y)$. Then,
\begin{align*}
W_\fp(\mu,\nu) &\le \left(\int_{X\times X} d(x,y)^pd\pi(x,y)\right)^\frac{1}{\fp}\\
&= \delta\left(\int_{X\times X} \left(\frac{d(x,y)}{\delta}\right)^pd\pi(x,y)\right)^\frac{1}{\fp}\\
&\le \delta\left(\int_{X\times X} \frac{d(x,y)}{\delta}d\pi(x,y)\right)^\frac{1}{\fp}\\
&= \delta^{1-\frac{1}{\fp}}W_1(\mu,\nu)^\frac{1}{\fp}.
\end{align*}
Finally, since $W_1$ is convex, it now follows that $W_\fp$ is effectively quasiconvex with $h(x)=\diam(X,\rho)^{1-\frac{1}{\fp}}x^\frac{1}{\fp}$. Both $\lp$ and $W_\infty$ are quasiconvex (that is, effectively quasiconvex with $h(x)=x$). Let us verify this for $\lp$, and the same argument will work for $W_\infty$. Fix $n\in\nn$, $\mu_1,\dots,\mu_n,\nu_1,\dots\nu_n\in \prob(X)$ and $\lambda_1,\dots,\lambda_n\in[0,1]$ with $\sum_{i=1}^n\lambda_i=1$. Let $U\subseteq X$ be Borel and let $r \ge\max\{\lp(\mu_i,\nu_i) \mid i=1,\dots,n\}$. Then,
\begin{align*}
(\sum_{i=1}^n\lambda_i\mu_i)(U_r)+r = \sum_{i=1}^n\lambda_i(\mu_i(U_r)+r)
\ge \sum_{i=1}^n\lambda_i\nu_i(U)
= (\sum_{i=1}^n\lambda_i\nu_i)(U)
\end{align*}
so $r \ge \lp\left(\sum_{i=1}^n\lambda_i\mu_i,\sum_{i=1}^n\lambda_i\nu_i\right)$. It follows that
\[
\lp\left(\sum_{i=1}^n\lambda_i\mu_i,\sum_{i=1}^n\lambda_i\nu_i\right) \le \max\{\lp(\mu_i,\nu_i) \mid i=1,\dots,n\}. \qedhere
\]
\end{proof}

Proposition~\ref{prop:wass} allows us to consider quantum metric Bauer simplices as special cases of a family of tracially lyriform $\cs$-algebras that we term \emph{tracial Wasserstein spaces}. Note that the diameter requirement for tracial \emph{Prokhorov} spaces can be arranged by scaling the metric on the boundary of the trace space.

\begin{definition} \label{def:qwass}
Let $(A,\omega)$ be a tracially lyriform $\cs$-algebra such that $(\partial_e(T(A)),\omega)$ is compact. We call $(A,\omega)$ a \emph{tracial $\fp$-Wasserstein space}, $\fp\in[1,\infty]$, if $\omega=W_\fp=W_{\fp,\rho}$ is the $\fp$-Wasserstein distance on $T(A)\cong\prob(\partial_e(T(A)))$ associated with a metric $\rho$ on $\partial_e(T(A))$. A \emph{tracial Wasserstein space} is a tracial $\fp$-Wasserstein space for some $\fp\in[1,\infty]$. Similarly, if $\omega=\lp$ is the L\'{e}vy--Prokhorov distance associated with $\rho$ and $\diam(\partial_e(T(A)),\rho)\le1$, then we call $(A,\omega)$ a \emph{tracial Prokhorov space}.
\end{definition}

Just as for quantum metric Bauer simplices, the starting point in the construction of tracial Wasserstein spaces tends to be a compact metric space $(X,\rho)$, with the tracially lyriform $\cs$-algebra $(A,W_\fp)$ then built as a noncommutative space of observables. We again say that $(A,W_\fp)$ \emph{is associated with} or \emph{observes} $(X,\rho)$.

\begin{remark} \label{rem:compact}
It would perhaps be better to use the terminology \emph{compact} tracial Wasserstein (or Prokhorov) spaces to emphasise that the metric spaces being observed are compact. We refrain from doing so because a `compact tracial $\infty$-Wasserstein space' need not be compact in the sense of Definition~\ref{def:compact} (although there is no such conflict for finite $\fp$). To see that $(\prob(X),W_\infty)$ is not compact unless $X$ is a singleton, fix points $x_0,x_1\in X$ with $\rho_X(x_0,x_1)=r>0$ and consider the sequence of measures $\mu_n:=\frac{1}{n}\delta_{x_1}+\frac{n-1}{n}\delta_{x_0}$. Then, $W_\infty(\mu_m,\mu_n)=r$ for every $m\ne n$, so $(\mu_n)_{n\in\nn}$ has no convergent subsequences.
\end{remark}

\begin{proposition} \label{prop:wasscomplete}
Let $\fp\in[1,\infty]$ and let $(A,W_{\fp,\rho})$ be a tracial Wasserstein space associated with a compact length space $(X,\rho)$. Then, $(A,W_{\fp,\rho})$ is complete and strictly intrinsic, and is of Bauer type if $\fp<\infty$.
\end{proposition}

\begin{proof}
By \cite[Corollary 2.7 and Remark 2.8]{Lott:2009aa} (also observed in \cite[Proposition 2.10]{Sturm:2006aa}), $(T(A),W_{\fp,\rho})$ is a compact, strictly intrinsic length space for every $\fp\in[1,\infty)$ . It remains to consider the case $\fp=\infty$. Completeness of $(T(A),W_{\infty,\rho})$ requires only completeness of $(X,\rho)$ (see \cite[Proposition 6]{Givens:1984to}). To see that it is strictly intrinsic, let $\sigma,\tau\in T(A)$ and for each $\fp\in[1,\infty)$ let $\gamma_\fp\colon[0,1]\to T(A)$ be a $W_{\fp,\rho}$-shortest path from $\sigma$ to $\tau$, that is,
\[
W_{\fp,\rho}(\gamma_\fp(s),\gamma_\fp(t)) = |s-t| \cdot W_{\fp,\rho}(\sigma,\tau)
\]
for every $s,t\in[0,1]$. This implies that
\[
W_{1,\rho}(\gamma_\fp(s),\gamma_\fp(t)) \le |s-t| \cdot W_{\infty,\rho}(\sigma,\tau)
\]
for every $s$ and $t$, and hence that $(\gamma_n)_{n\in\nn}$ is a sequence of curves in the compact metric space $(T(A),W_{1,\rho})$ whose lengths are uniformly bounded by $W_{\infty,\rho}(\sigma,\tau)$. By the Arzel\`a--Ascoli \cite[Theorem 2.5.14]{Burago:2001aa}, there is a subsequence $(\gamma_{n_k})_{k\in\nn}$ that uniformly converges to some path $\gamma_\infty\colon[0,1]\to T(A)$ from $\sigma$ to $\tau$. Let $s,t\in[0,1]$. Since
\[
W_{n,\rho}(\gamma_m(s),\gamma_m(t)) \le W_{m,\rho}(\gamma_m(s),\gamma_m(t)) \le |s-t| \cdot W_{\infty,\rho}(\sigma,\tau)
\]
for every $m\ge n$, it follows that
\[
W_{n,\rho}(\gamma_\infty(s),\gamma_\infty(t)) \le |s-t| \cdot W_{\infty,\rho}(\sigma,\tau)
\]
for every $n$, and hence that
\[
W_{\infty,\rho}(\gamma_\infty(s),\gamma_\infty(t)) \le |s-t| \cdot W_{\infty,\rho}(\sigma,\tau).
\]
Therefore, $W_{\infty,\rho}(\gamma_\infty(s),\gamma_\infty(t)) = |s-t| \cdot W_{\infty,\rho}(\sigma,\tau)$ for every $s,t\in[0,1]$, so $\gamma_\infty$ is a shortest path from $\sigma$ to $\tau$ in $(T(A),W_{\infty,\rho})$.
\end{proof}

Suppose that $A$ is a separable, unital $\cs$-algebra such that $T(A)$ is a Bauer simplex with boundary $(X,\rho)$. Then, the obstacle to tracial Wasserstein or Prokhorov structure is density of tracially Lipschitz elements. But if we know that $(A,L_\rho)$ is a quantum metric Bauer simplex, then by Proposition~\ref{prop:wass}, this density is automatic.

\begin{proposition} \label{prop:inherit}
Every tracial $\fp$-Wasserstein space, $\fp\in[1,\infty]$, also inherits the structure of a tracial Prokhorov space and a tracial $\fq$-Wasserstein space for any other $\fq\in[1,\infty]$.
\end{proposition}

The following version of \cite[Example 2.15.1]{Jacelon:2024aa} applies not just to $C(X,M_n)$ but also to unital, continuous-trace subalgebras like generalised dimension drop algebras $X_{p,q}$ (see Definition~\ref{def:dimdrop}) that fall under \cite[Proposition 2.7]{Jacelon:2024aa}. Non-type \rm{I} $\cs$-algebras like $C(X)\otimes\js$ and $C(X)\otimes\mathcal{Q}$, as well as subalgebras like $\js_{p^\infty,q^\infty}$, are allowed, too.

\begin{corollary} \label{cor:wasslyre}
Let $(X,\rho)$ be a compact metric space, let $n\in\nn$ and let $A=C(X,M_n)$. Then, $(A,W_\fp)$ is a tracial $\fp$-Wasserstein space for any $\fp\in[1,\infty]$, and $(A,\lp)$ is a tracial Prokhorov space if $\diam(X,\rho)\le 1$. For any $r\ge r_{(X,\rho)}$, a common nucleus for all of these structures is
\begin{equation} \label{eqn:drmat}
\mathcal{D}_r(A) = \{f \in A_{sa} \mid \|f\|\le r,\: \|f(x)-f(y)\|\le\rho(x,y)\:\text{ for every } x,y\in X\}.
\end{equation}
\end{corollary}

The next corollary is a consequence of Proposition~\ref{prop:inherit} and Theorem~\ref{thm:alexandrov}.

\begin{corollary} \label{cor:qwass}
Let $\fp\in[1,\infty]$ and let $(X,\rho)$ be a compact, finite-dimensional Alexandrov space with lower and upper curvature bounds. Suppose that $X$ is either boundaryless or is contractible and of nonpositive curvature. Then, there exists a classifiable, projectionless tracial $\fp$-Wasserstein space $(\js_X,W_{\rho,\fp})$ observing $(X,\rho)$. If in addition $\diam(X,\rho)\le 1$, then $(\js_X,\lp)$ is a tracial Prokhorov space.
\end{corollary}

Our final goal in this section is to show that the quantum intertwining gap (introduced in \cite[Section 3]{Jacelon:2024aa} and recalled in Definition~\ref{def:quig}) is a complete quasimetric on the set of tracial isomorphism classes (Definition~\ref{def:tracialisom}) of what we call quasi-Wasserstein spaces (Definition~\ref{def:quasiwass}).

\begin{definition} \label{def:tracialisom}
A \emph{tracial isomorphism} between two  tracially lyriform $\cs$-algebras $(A_1,\omega_1)$ and $(A_2,\omega_2)$ is an affine isometric isomorphism between the trace spaces $(T(A_2),\omega_2)$ and $(T(A_1),\omega_1)$. 
\end{definition}

In general, such an isomorphism need not be induced by a $^*$-isomorphism $A_1\to A_2$, but if we are in the $K$-connected, classifiable setting as in \cite[\S3]{Jacelon:2024aa}, then (modulo $K$-theory) we \emph{can} lift tracial isomorphisms to $\cs$-isomorphisms. See Corollary~\ref{cor:quigwass}.

\begin{definition} \label{def:quasiwass}
A \emph{quasi-Wasserstein metric} $\Omega$ is a family of metrics $\Omega_{(X,\rho)}$ on $\prob(X)$, one for each (nonempty) compact metric space $(X,\rho)$ of diameter at most some $D=D_\Omega\in[0,\infty]$, such that:
\begin{enumerate}[(1)]
\item \label{it:quasi1} $\Omega_{(X,\rho)}$ induces the $\ws$-topology on $\prob(X)$ and agrees with $\rho$ on $X\cong\partial_e\prob(X)$;
\item \label{it:quasi2} $\Omega$ is compatible with pushforwards, meaning that if $f\colon(X_1,\rho_1)\to(X_2,\rho_2)$ is $K$-Lipschitz for some $K\ge1$, then the pushforward map $f_*\colon(\prob(X_1),\Omega_{(X_1,\rho_1)})\to(\prob(X_2),\Omega_{(X_2,\rho_2)})$ is also $K$-Lipschitz, and is moreover isometric if $f$ is;
\item \label{it:quasi4} $\Omega$ is uniformly effectively quasiconvex on metric spaces of bounded diameter, that is, for every $M>0$, there is a single function $h=h_\Omega$ such that \eqref{eqn:qc} holds for $\Omega_{(X,\rho)}$ whenever $\diam(X,\rho)\le M$;
\item \label{it:quasi5} $\Omega$ is uniformly bounded below by $W_1$, meaning that there is a constant $K=K_\Omega>0$ such that $K\cdot\Omega_{(X,\rho)}) \ge W_{1,\rho}$ for every $(X,\rho)$.
\end{enumerate}
Given such an $\Omega$, a \emph{tracial $\Omega$-quasi-Wasserstein space} is a tracially lyriform $\cs$-algebra $(A,\Omega_{(X,\rho)})$ (necessarily of Bauer type) observing some nonempty compact metric space $(X,\rho)$ of diameter at most $D_\Omega$. A \emph{tracial quasi-Wasserstein space} is a tracial $\Omega$-quasi-Wasserstein space for some $\Omega$.
\end{definition}

Property \eqref{it:quasi5} implies that the set of $\Omega_{(X,\rho)}$-Lipschitz elements contains $\lipo(X,\rho)$. As in Corollary~\ref{cor:qwass}, this in turn implies that any tracial Wasserstein space (of suitable diameter) also has the structure of a tracial $\Omega$-quasi-Wasserstein space.

\begin{proposition} \label{prop:quasiwass}
Every tracial Prokhorov space is a quasi-Wasserstein space, as is every tracial $\fp$-Wasserstein space for any $\fp\in[1,\infty)$.
\end{proposition}

\begin{proof}
For Prokhorov spaces we take $D=1$ and $K=2$, and for Wasserstein spaces we take $D=\infty$ and $K=1$ (cf.\ \eqref{eqn:lp2} and \eqref{eqn:wincrease}). Properties \eqref{it:quasi1} and \eqref{it:quasi4} also follow from Proposition~\ref{prop:wass}. Compatibility with pushforwards (property \eqref{it:quasi2}) can be checked from the definitions \eqref{eqn:lp1} and \eqref{eqn:wp} of $d_{LP}$ and $W_\fp$. The argument for the Wasserstein metrics is the same as in \cite[Proposition 2.3]{Jacelon:2021vc}, and  the key observation for $d_{LP}$ is that $f^{-1}(U_{Kr})\supseteq (f^{-1}(U))_r$ for every $r>0$, $U\subseteq X_2$ and $K$-Lipschitz map $f\colon(X_1,\rho_1)\to(X_2,\rho_2)$. Suppose that $d_{LP}(\mu,\nu)\le r$ and assume that $K\ge1$. Then, for every Borel set $U\subseteq X_2$, we have
\[
f_*\mu(U_{Kr}) + Kr = \mu(f^{-1}(U_{Kr})) + Kr \ge \mu((f^{-1}(U))_r) + r \ge \nu(f^{-1}(U)) = f_*\nu(U)
\]
so $d_{LP}(f_*\mu,f_*\nu)\le Kr$. It also follows from their definitions that $d_{LP}$ and $W_{\fp}$ are compatible with restriction (that is, if $\emptyset\ne Y\subseteq X$ is closed and $\mu,\nu\in\prob(X)$ are supported on $Y$, then $d_{LP;X}(\mu\,\nu)=d_{LP;Y}(\mu,\nu)$ and $W_{\fp;X}(\mu,\nu)=W_{\fp;Y}(\mu,\nu)$, the extra subscripts indicating the space in which the distance is measured) and are invariant under isometric isomorphism. This shows that isometries extend to isometries and completes the proof. 
\end{proof}

\begin{definition} \label{def:quig}
The \emph{quantum intertwining gap} between compact tracially lyriform $\cs$-algebras $(A_1,\omega_1)$ and $(A_2,\omega_2)$ is
\[
\gamma_q((A_1,\omega_1),(A_2,\omega_2)) = \inf\{\varepsilon>0 \mid \Gamma_\eps(T(A_1),T(A_2)) \times \Gamma_\eps(T(A_2),T(A_1)) \ne \emptyset\}
\]
where, given $\varepsilon>0$, $\Gamma_\eps(T(A_i),T(A_j)) = \Gamma_\eps((T(A_i),\omega_i),(T(A_j),\omega_j))$ denotes the set of continuous affine maps $f\colon T(A_i) \to T(A_j)$ that are
\begin{items}
\item \emph{$\varepsilon$-isometric} ($|\omega_j(f(\tau_1),f(\tau_2)) - \omega_i(\tau_1,\tau_2)| < \varepsilon$ for every $\tau_1,\tau_2\in T(A_i)$) and
\item \emph{$\varepsilon$-invertible} (there exists a continuous affine map $g \colon T(A_j) \to T(A_i)$ such that $\omega_j(f \circ g(\tau),\tau) < \varepsilon$ for every $\tau\in T(A_j)$).
\end{items}
\end{definition}

The following theorem, as well as its proof, is an adaptation of \cite[Theorem 3.7]{Jacelon:2024aa} to the setting of quasi-Wasserstein spaces. It is the motivation for the definition \eqref{eqn:qc} of effective quasiconvexity.

\begin{theorem} \label{thm:quigwass}
For every quasi-Wasserstein metric $\Omega$, the quantum intertwining gap $\gamma_q$ is a separable, complete quasimetric on the set of tracial isomorphism classes of quasi-Wasserstein spaces $(A,\Omega_{(X,\rho)})$ observing nonempty compact metric spaces $(X,\rho)$.
\end{theorem}

\begin{proof}
Recall that for $c>0$, a `$c$-quasimetric' $d$ satisfies the axioms of a metric except that the triangle inequality is weakened to $d(x,z)\le c(d(x,y)+d(y,z))$. It is immediate from its definition that $\gamma_q$ is symmetric, and one can show exactly as in the proof of \cite[Theorem 3.7]{Jacelon:2024aa} that $\gamma_q$ satisfies the weakened triangle inequality with $c=2$. It is also immediate that $\gamma_q$ evaluates to zero on tracially isomorphic pairs of objects. Suppose conversely that $\gamma_q((A_1,\omega_1),(A_2,\omega_2))=0$. For each $n\in\nn$, let $f_n\colon T(A_2) \to T(A_1)$ be an affine map that is  $\frac{1}{n}$-isometric and $\frac{1}{n}$-surjective. Let $Y$ be a countable dense subset of $T(A_2)$. By compactness of $T(A_1)$ and a diagonal argument, there is a subsequence $(f_{n_k})_{k\in\nn}$ such that $f_{n_k}(y)$ converges for every $y\in Y$. Define $f\colon Y \to T(A_1)$ by $f(y)=\lim_{k\to\infty}f_{n_k}(y)$. Then, $f$ is an affine isometry and has dense range, so extends to an affine isometric isomorphism $(T(A_2),\omega_2) \to (T(A_1),\omega_1)$. So, $\gamma_q$ is indeed a $2$-quasimetric on tracial isomorphism classes of compact tracially lyriform $\cs$-algebras.

For separability, we claim that quasi-Wasserstein spaces over finite metric spaces $(F,\sigma)$ taking rational distance values are $\gamma_q$-dense. Let $(X,\rho)$ be a compact metric space of suitable diameter so that $\omega=\Omega_{(X,\rho)}$ is defined, let $M=\diam(\prob(X),\omega)\ge\diam(X,\rho)$ and let $\eps\in(0,1)$. Let $F\subseteq X$ be a finite $\eps$-net, and let $\sigma$ be a metric on $F$ that takes rational values such that, for every $x,y\in F$,
\[
\left(1-\frac{\eps}{M+1}\right)\rho(x,y) \le \sigma(x,y) \le \left(1+\frac{\eps}{M+1}\right)\rho(x,y).
\]
By compatibility of $\Omega$ with pushforwards (Definition~\ref{def:quasiwass}\eqref{it:quasi2}) applied to $\id_F\colon(F,\sigma)\to(F,\rho)$ and $\id_F\colon(F,\rho)\to(F,\sigma)$, this implies that
\[
\left(1-\frac{\eps}{M+1}\right)\omega(\mu,\nu) \le \Omega_{(F,\sigma)}(\mu,\nu) \le \left(1+\frac{\eps}{M+1}\right)\omega(\mu,\nu)
\]
and hence that
\begin{equation} \label{eqn:fisom}
|\omega(\mu,\nu) - \Omega_{(F,\sigma)}(\mu,\nu)| < \eps
\end{equation}
for every $\mu,\nu\in \prob(F)$. (Here we are also implicitly using the fact that $\Omega$ is compatible with restriction, that is, $\Omega_{(X,\rho)}|_{\prob(F)\times\prob(F)}=\Omega_{(F,\rho)}$, which is a special case of compatibility with pushforwards of isometric inclusions.) In other words, the inclusion map $f\colon(\prob(F),\Omega_{(F,\sigma)}) \to (\prob(X),\omega)$ is $\eps$-isometric. Appealing to the Yannelis--Prabhakar selection theorem as in the proof of \cite[Theorem 3.7]{Jacelon:2024aa}, we can find a continuous function $g\colon X \to \prob(F) \subseteq \prob(X)$ such that $\omega(f(g(x)),x)<\eps$ for every $x\in X$. Specifically, for every $x\in X$, the set $S(x):=\{\mu\in\prob(F)\mid \omega(f(\mu),x)<\eps\}$ is a nonempty, convex subset of the topological vector space $\rr^{|F|}\cong C(F)^*\supseteq\prob(F)$ such that for every $\mu\in\prob(F)$, $S^{-1}(\mu):=\{x\in X \mid \mu\in S(x)\} = \{x \in X \mid \omega(f(\mu),x)<\eps\}$ is an open subset of $X$. Yannelis and Prabhakar prove in \cite[Theorem 3.1]{Yannelis:1983aa} that these conditions guarantee the existence of a continuous function $g\colon X \to \prob(F)$ such that $g(x)\in S(x)$, that is, $\omega(f(g(x)),x)<\eps$, for every $x\in X$. By effective quasiconvexity, we then have
\begin{equation} \label{eqn:fsurject}
\omega(f(g_*(\mu)),\mu) \le h(\eps)
\end{equation}
for every $\mu\in\prob(X)$. Indeed, for finite convex combinations $\mu=\sum_{i=1}^n \lambda_ix_i$ with $x_i\in X$, we have
\[
\omega(f(g_*(\mu)),\mu) = \omega\left(\sum_{i=1}^n \lambda_if(g(x_i)),\sum_{i=1}^n \lambda_ix_i\right) \le h\left(\max_{1\le i\le n}\omega(f(g(x_i)),x_i)\right) \le h(\eps)
\]
so the assertion follows from the Krein--Milman theorem. We have thus shown that
\[
f\in\Gamma_{\eps+h(\eps)}((\prob(F),\Omega_{(F,\sigma)}),(\prob(X),\omega)).
\]
As for $g_*$, for $\mu,\nu\in\prob(X)$ we have from \eqref{eqn:fisom} and \eqref{eqn:fsurject} that
\begin{align*}
\Omega_{(F,\sigma)}(g_*(\mu),g_*(\nu)) &< \omega(f(g_*(\mu)),f(g_*(\nu))) + \eps\\
&\le \omega(f(g_*(\mu)),\mu) + \omega(\mu,\nu) + \omega(\nu,f(g_*(\nu))) + \eps\\
&\le \omega(\mu,\nu) + \eps + 2h(\eps)
\end{align*}
and similarly
\[
\Omega_{(F,\sigma)}(g_*(\mu),g_*(\nu)) > \omega(\mu,\nu) - \eps - 2h(\eps)
\]
which together tell us that $g_*\colon (\prob(X),\omega) \to (\prob(F),\Omega_{(F,\sigma)})$ is $(\eps + 2h(\eps))$-isometric. We also deduce from \eqref{eqn:fisom} and \eqref{eqn:fsurject} that
\[
\Omega_{(F,\sigma)}(g_*(f(\mu)),\mu) < \omega(f(g_*(f(\mu))),f(\mu)) + \eps  \le \eps + h(\eps)
\]
for every $\mu\in\prob(F)$. We conclude that 
\[
g_*\in\Gamma_{\eps+2h(\eps)}((\prob(X),\omega),(\prob(F),\Omega_{(F,\sigma)}))
\]
and therefore that $\gamma_g((\prob(X),\omega),(\prob(F),\Omega_{(F,\sigma)})) \le \eps+2h(\eps)$. Since $\eps$ and $h(\eps)$ are arbitrarily small, this completes the proof of separability.

Finally, let us show that $\gamma_q$ is complete. Let $((A_n,\omega_n=\Omega_{(X_n,\rho_n)})_{n\in\nn}$ be a $\gamma_q$-Cauchy sequence of $\Omega$-quasi-Wasserstein spaces. Arguing as in the proof of \cite[Theorem 3.7]{Jacelon:2024aa} (replacing any application of convexity of the Wasserstein metric $W_1$ by uniform effective quasiconvexity of the quasi-Wasserstein metric $\Omega$ and using property \eqref{it:quasi5}), for every $n\in\nn$ we can find $\delta_n>0$ with the following property: for any $(A,\omega=\Omega_{(Y,\lambda)})$, if $f\colon (\prob(X_n),\omega_n) \to (\prob(Y),\omega)$ is $\delta_n$-isometric and $\delta_n$-surjective, then
\[
\widehat\rho(x,y) := \inf_{z\in \prob(X_n)}(\omega_n(x,z) + \omega(f(z),y)) + \frac{\delta_n}{2}
\]
provides an admissible metric on $X_n \sqcup Y$ (that is, $\widehat\rho$ is defined to agree with $\rho_n$ and $\lambda$ on $X_n$ and $Y$, respectively) with the property that the Hausdorff distance in $(\prob(X_n \sqcup Y),W_{1,\widehat\rho})$ between $\prob(X_n)$ and $\prob(Y)$ is less than $2^{-n}$. As in \cite[Proposition 4.8]{Rieffel:2004aa}, this implies that the Gromov--Hausdorff distance $d_{GH}$ between $(X_n,\rho_n)$ and $(Y,\lambda)$ is also less than $2^{-n}$. By passing to a subsequence, we may assume that $\gamma_q((A_n,\omega_n),(A_{m},\omega_{m}))<\delta_n$ for every $m\ge n$. It follows that $(X_n,\rho_n)$ is a $d_{GH}$-Cauchy sequence in the metric space of isometry classes of compact metric spaces. This metric space is complete (a fact whose proof can be found in \cite[Chapter 10]{Petersen:2006aa} or \cite[Section 1]{Rong:2010aa}), so there is a limit space $(X,\rho)$. For a given $\eps>0$, there therefore exists $N\in\nn$ such that, for every $n\ge N$, there exists some admissible metric on $X_n\sqcup X$ with respect to which the Hausdorff distance between $X_n$ and $X$ is less than $\eps$. For such an $n\ge N$, uniform effective quasiconvexity then implies that the Hausdorff distance between $(\prob(X_n),\Omega_{(X_n,\rho_n)})$ and $(\prob(X),\Omega_{(X,\rho)})$ is less than $h(\eps)$. Exactly as in the proof of \cite[Theorem 3.7]{Jacelon:2024aa}, another application of the Yannelis--Prabhakar theorem provides continuous affine functions $f\colon \prob(X_n) \to \prob(X)$ and $g\colon \prob(X) \to \prob(X_n)$ such that
\[
\Omega_{X_n\sqcup X}(x,f(x)) < h(\eps) \quad \text{and} \quad \Omega_{X_n\sqcup X}(y,g(y)) < h(\eps)
\]
for every $x\in X_n$ and $y\in X$. It then follows from uniform effective quasiconvexity and Krein--Milman that
\[
\Omega_{X_n\sqcup X}(\mu,f(\mu)) < h(h(\eps)) \quad \text{and} \quad \Omega_{X_n\sqcup X}(\nu,g(\nu)) < h(h(\eps))
\]
for every $\mu\in\prob(X_n)$ and $\nu\in\prob(X)$. As in \cite{Jacelon:2024aa}, this implies that $f$ and $g$ are $2h(h(\eps))$-isometric mutual $2h(h(\eps))$-inverses in the sense of Definition~\ref{def:quig}. It follows that
\[
\gamma_q((A_n,\Omega_{(X_n,\rho_n)}),(\prob(X),\Omega_{(X,\rho)})) \le 2h(h(\eps)).
\]
Since $h(h(\eps))$ is arbitrarily small, this implies that the tracial isomorphism class of $(\prob(X),\Omega_{(X,\rho)})$ is equal to the $\gamma_q$-limit of the tracial isomorphism classes of the sequence $((A_n,\Omega_{(X_n,\rho_n}))_{n\in\nn}$. We conclude that $\gamma_q$ is indeed complete when restricted to tracial isomorphism classes of $\Omega$-quasi-Wasserstein spaces.
\end{proof}

Finally, we restrict our attention to $K$-connected, classifiable quasi-Wasserstein spaces $(A,\Omega_{(X,\rho)})$ that share the same $K$-theory. In this case, just as in \cite[Corollary 3.10]{Jacelon:2024aa}, classification \cite[Theorem 9.9]{Carrion:wz} allows us to lift tracial isomorphisms to $\cs$-isomorphisms. For notational simplicity in the statement of Corollary~\ref{cor:quigwass}, we identify the index set $\{1,2\}$ with the additive group $\zz/2\zz$.
 
\begin{corollary} \label{cor:quigwass}
Let $\Omega$ be a quasi-Wasserstein metric, let $G_1$ be a countable abelian group and let $(G_0,G_0^+)$ be a countable, simple, weakly unperforated ordered abelian group with distinguished order unit $g_0\in G_0^+$. Then, $\gamma_q$ is a $2$-quasimetric on the set of $\lyre$-isomorphism classes of $\Omega$-quasi-Wasserstein spaces $(A,\omega)$ such that $A$ is unital, $K$-connected and classifiable with
\[
(K_0(A),K_0(A)_+,[1_A],K_1(A)) \cong (G_0,G_0^+,g_0,G_1).
\]
Moreover, $\gamma_q$ admits the description
\begin{align*}
\gamma_q((A_1,\omega_1),(A_2,\omega_2)) = \inf\{\varepsilon>0 \mid &\:\text{there exist}\: \varphi_i \in \emb(A_i,A_{i+1}), \:i=1,2,\: \text{with}\\
&T(\varphi_i) \in \Gamma_\eps((T(A_{i+1}),\omega_{i+1}),(T(A_i),\omega_i))\}.
\end{align*}
\end{corollary}

To illustrate the ideas surrounding the quantum intertwining gap, let us compare classifiable tracial Wasserstein spaces observing normed balls and spheres with their Euclidean counterparts. Let $(E,\|\cdot\|_E)$ be a finite-dimensional real normed space, let $\sphere_E$ and $\ball_E$ denote its unit sphere and unit ball, and let $\rho_{E}$ denote the intrinsic metric on $X_E=\sphere_E$ or $X_E=\ball_E$ induced by $\|\cdot\|_E$, that is,
\[
\rho_{E}(x,y) = \inf\{\ell(\gamma) \mid \gamma \colon [a,b] \to X_E \text{ is a rectifiable path from $x$ to $y$}\}
\]
as in Definition~\ref{def:length}. As described in Example~\ref{ex:normed}, $\rho_E(x,y)=\|x-y\|_E$ for $x,y\in\ball_E$. Fix $\fp\in[1,\infty)$, and let us further assume that $\dim \sphere_E = \dim E-1$ is odd. By Proposition~\ref{prop:inherit} and Corollary~\ref{cor:ballsandspheres}, together with the equivalence of norms on finite-dimensional spaces, there exist classifiable, projectionless tracial $\fp$-Wasserstein spaces $(\js_{\sphere_E},W_{\fp,\rho_E})$ observing $(\sphere_E,\rho_{E})$ and $(\js_{\ball_E},W_{\fp,\rho_E})$ observing $(\ball_E,\rho_{E})$ such that
\[
(K_0(\js_{\sphere_E}),K_0(\js_{\sphere_E})_+,[1_{\js_{\sphere_E}}]) \cong (\zz,\nn,1)
\]
and
\[
(K_0(\js_{\ball_E}),K_0(\js_{\ball_E})_+,[1_{\js_{\ball_E}}],K_1(\js_{\ball_E})) \cong (\zz,\nn,1,0).
\]
Every shortest path in $X_E=\sphere_E$ or $X_E=\ball_E$ provides an isometry from the Euclidean interval $\mathbb{I}=[0,1]\cong\ball^1$ into $X_E$, extending to an isometric embedding $\iota_{\mathbb{I}}$ of the Wasserstein space $(\prob([0,1]),W_\fp)$ into $(T(\js_{X_E}),W_{\fp,\rho_E})$. By classification, there exists a unital embedding $\varphi\colon\js_{X_E}\to\js_{\ball^{1}}$ such that $T(\varphi)=\iota_{\mathbb{I}}$.

In fact, we can embed the spaces $\js_{\sphere_E}$ and $\js_{\ball_E}$ \emph{approximately} tracially isometrically into tracial $\fp$-Wasserstein spaces observing high-dimensional Euclidean spheres $(\sphere^k,\rho_{\Euc})$ or balls $(\ball^k,\rho_\Euc)$, provided that $d:=\dim E$ is sufficiently large. 

\begin{theorem} \label{thm:dvoretzky}
Let $\fp\in[1,\infty)$ and $\eps\in(0,1)$. There is a constant $C(\eps,\fp)$ such that, for every $d,k\in2\nn$ with $k\le C(\eps,\fp)\log d$, if $(E,\|\cdot\|_E)$ is a $d$-dimensional real normed space, then there are embeddings
\[
(\js_{\sphere_E},W_{\fp,\rho_E}) \to (\js_{\sphere^{k-1}},W_{\fp,\rho_\Euc}) \quad \text{and}\quad \text (\js_{\ball_E},W_{\fp,\rho_E}) \to  (\js_{\ball^k},W_{\fp,\rho_\Euc})
\]
whose induced tracial maps are $\eps$-isometric.
\end{theorem}

\begin{proof}
This is a consequence of Dvoretzky's theorem \cite{Dvoretzky:1961aa}  (see also \cite{Milman:1971aa} or \cite[Chapters 3--5]{Milman:1986ux}), which provides a constant $c=c(\eps)$ such that, whenever $k\le c\log d$, there is a $k$-dimensional subspace of $E$ whose intersections $(\sphere^{k-1},\rho_{E})$ with $\sphere_E$ and $(\ball^k,\rho_{E})$ with $\ball_E$ are within Gromov--Hausdorff distance $\eps$ of the Euclidean sphere $(\sphere^{k-1},\rho_{\Euc})$ and the Euclidean ball $(\ball^k,\rho_\Euc)$, respectively. As in the proof of Theorem~\ref{thm:quigwass}, this implies that, for a suitable increasing function $h_\fp$ with $\lim_{x\to0}h_\fp(x)=0$, we have
\[
\gamma_q((\js_{\sphere^{k-1}},W_{\fp,\rho_E}),(\js_{\sphere^{k-1}},W_{\fp,\rho_{\Euc}})) < h_\fp(\eps)
\]
(and similarly for $\ball^k$). If $k-1$ is odd, then by classification \cite[Corollary C]{Carrion:wz} and Corollary~\ref{cor:ballsandspheres}, the inclusions $(\sphere^{k-1},\rho_{E}) \subseteq (\sphere_E,\rho_{E})$ and $(\ball^k,\rho_{E}) \subseteq (\ball_E,\rho_{E})$ can be realised by unital embeddings
\[
(\js_{\sphere_E},W_{\fp,\rho_E}) \to (\js_{\sphere^{k-1}},W_{\fp,\rho_E}) \approx_{h_\fp(\eps)}(\js_{\sphere^{k-1}},W_{\fp,\rho_\Euc})
\]
and
\[
(\js_{\ball_E},W_{\fp,\rho_E}) \to (\js_{\ball^k},W_{\fp,\rho_E}) \approx_{h_\fp(\eps)}  (\js_{\ball^k},W_{\fp,\rho_\Euc})
\]
where $\approx_{h_\fp(\eps)}$ is relative to $\gamma_q$, which admits the description presented in Corollary~\ref{cor:quigwass}. It follows that there are embeddings
\[
(\js_{\sphere_E},W_{\fp,\rho_E}) \to (\js_{\sphere^{k-1}},W_{\fp,\rho_\Euc}) \quad \text{and}\quad \text (\js_{\ball_E},W_{\fp,\rho_E}) \to  (\js_{\ball^k},W_{\fp,\rho_\Euc})
\]
whose induced tracial maps are $h_\fp(\eps)$-isometric.
\end{proof}

The question naturally arises, when is the quasimetric space described in Corollary~\ref{cor:quigwass} complete? Theorem~\ref{thm:quigwass} only gives completeness at the tracial level. For completeness at the $\cs$-level we would need to be able to provide a tracial quasi-Wasserstein space with prescribed $K$-theory $(G_0,G_0^+,g_0,G_1)$, which is not something that we have proved in full generality. This problem disappears if we restrict to a collection of \emph{contractible} spaces (with sufficient geometric regularity to ensure that Gromov--Hausdorff limits are also contractible) observed by classifiable tracial Wasserstein spaces with point-like ordered $K$-theory \eqref{eqn:pointlike}. For an example covered by Corollary~\ref{cor:contractible}, we could, for instance, consider two-dimensional convex polygons converging to the planar disc.

\begin{corollary} \label{cor:contractible}
For fixed $n\in\nn$, $D,v>0$ and $k\le\kappa\le0$, let $\alex_0^n(D,k,\kappa,v)$ denote the collection of $n$-dimensional compact, contractible Alexandrov spaces $(X,\rho)$ of diameter at most $D$, curvature at least $k$ and at most $\kappa$, and such that the $n$-dimensional Hausdorff measure of $X$ is at least $v$. Then, for every $\fp\in[1,\infty)$, $\gamma_q$ is a complete $2$-quasimetric on the set of $\lyre$-isomorphism classes of classifiable tracial $\fp$-Wasserstein spaces $(A_X,W_{\fp,\rho})$ observing objects $(X,\rho)$ in $\alex_0^n(D,k,\kappa,v)$ such that
\begin{equation} \label{eqn:pointlike}
(K_0(A_{X}),K_0(A_{X})_+,[1_{A_{X}}],K_1(A_{X})) \cong (\zz,\nn,1,0).
\end{equation}
\end{corollary}

\begin{proof}
The hypotheses ensure that the collection of (isometry classes of) spaces being observed is Gromov--Hausdorff closed. Indeed, suppose that $(X,\rho)$ is a $d_{GH}$-limit of metric spaces $(X_n,\rho_n)$ in $\alex_0^n(D,k,\kappa,v)$. By Gromov's compactness theorem \cite[Corollary 10.7.2]{Burago:2001aa}, $(X,\rho)$ is an Alexandrov space of curvature at least $k$, diameter at most $D$ and dimension at most $n$. The upper curvature bound $\kappa\le0$ is also preserved in the limit: as mentioned in Remark~\ref{rem:hadamard}, each $X_n$ is $CAT(\kappa)$, hence by \cite[Corollary II.3.10]{Bridson:1999aa} so is the Gromov--Hausdorff limit $X$. The relevance of bounded $n$-dimensional Hausdorff measure is to ensure there is no dimensional collapse: by \cite[Theorem 10.10.10 and Corollary 10.10.11]{Burago:2001aa}, the $n$-dimensional Hausdorff measure of $X$ is at least $v$ and the Hausdorff dimension of $X$ (which for finite-dimensional Alexandrov spaces coincides with the topological dimension) is actually equal to $n$. Perelman's stability theorem (see \cite{Kapovitch:2007aa}) then implies that $X_n$ is eventually homeomorphic to $X$, so in particular, $X$ is contractible. We conclude that $(X,\rho)$ is also an object in $\alex_0^n(D,k,\kappa,v)$.

Now suppose that $(A_{X_n},W_{\fp,\rho_n})$ is a $\gamma_q$-Cauchy sequence of classifiable tracial $\fp$-Wasserstein spaces observing objects $(X_n,\rho_n)$ in $\alex_0^n(D,k,\kappa,v)$, such that each $A_{X_n}$ has ordered $K$-theory given by \eqref{eqn:pointlike}. By Theorem~\ref{thm:quigwass}, there exists $(X,\rho)$ such that $(\prob(X),W_{\fp,\rho})$ is tracially isomorphic to the $\gamma_q$-limit of $(A_{X_n},W_{\fp,\rho_n})$. The proof of the theorem shows that, equivalently, $(X,\rho)$ is the $d_{GH}$-limit of $(X_n,\rho_n)$. As outlined above, this then implies that $(X,\rho)$ is also an object in $\alex_0^n(D,k,\kappa,v)$. By Corollary~\ref{cor:qwass}, $(X,\rho)$ is observed by a projectionless, classifiable tracial $\fp$-Wasserstein space $(\js_X,W_{\fp,\rho})$. As discussed in Section~\ref{section:qmcs}, the $\cs$-algebra $\js_X$ is built as an inductive limit of generalised dimension drop algebras over the space $X$. Since $X$ is contractible, the $K$-theory computation Proposition~\ref{prop:ktheory} implies that
\[
(K_0(\js_X),K_0(\js_X)_+,[1_{\js_X}],K_1(\js_X)) \cong (\zz,\nn,1,0).
\]
Thus, the limit $(\js_X,W_{\fp,\rho})$ is in the required class, which we have now demonstrated is $\gamma_q$-complete.
\end{proof}

\section{Inductive limits and the Lyre category} \label{section:limits}

Aside from the tracial quasi-Wasserstein spaces discussed in Section~\ref{section:wass} and the examples presented at the end of Section~\ref{section:lyriform}, the other structures that motivate the definition of a tracially lyriform $\cs$-algebra are inductive limits. Let us first describe the category in which these limits will be formed.

\begin{definition} \label{def:category}
The \emph{lyre category} $\lyre$ is the category whose objects are tracially lyriform $\cs$-algebras $(A,\rho)$, with morphisms $(A,\rho_A) \to (B,\rho_B)$ consisting of unital $^*$-homomorphisms $\varphi \colon A \to B$ that are injective and \emph{tracially nonexpansive}, that is, such that $T(\varphi) \colon (T(B),\rho_B) \to (T(A),\rho_A)$ is $1$-Lipschitz. In particular, isomorphisms in $\lyre$ are tracially isometric $^*$-isomorphisms.
\end{definition}

The requirement of nonexpansivity also appears in \cite[Definition 2.5]{Aguilar:2021aa}. There, attention must be paid to the (quasi) Leibniz structure of the building blocks because the quantum metric on the limit is obtained via the completeness of Latr\'{e}moli\`{e}re's dual Gromov--Hausdorff propinquity. In that setup, one must also be careful to ensure that the $\cs$-algebra obtained in this way is isomorphic to the $\cs$-algebraic limit of the inductive system. This issue does not arise for us because we will explicitly build the metric $\rho_\infty$ on the trace space of the $\cs$-limit whenever the system is bounded in diameter and degree of effective quasiconvexity.

\begin{proposition} \label{prop:colimits}
Let $((A_n,\rho_n),\varphi_n \colon A_n \to A_{n+1})_{n\in\nn}$ be an inductive system in $\lyre$ and let $A$ be the $\cs$-algebraic direct limit $\varinjlim(A_n,\varphi_n)$. Suppose that
\begin{enumerate}[(1)]
\item \label{it:hyp1} $\sup_{n\in\nn}\diam(T(A_n),\rho_n) < \infty$, and
\item \label{it:hyp2} the metric $\rho_\infty$ on $T(A) \cong \varprojlim (T(A_n),\varphi_n^*)$ defined by
\begin{equation} \label{eqn:limitrho}
\rho_\infty((\sigma_n)_{n=1}^\infty,(\tau_n)_{n=1}^\infty) = \sup_{n\in\nn} \rho_n(\sigma_n,\tau_n)
\end{equation}
is effectively quasiconvex (which holds whenever $(\rho_n)_{n\in\nn}$ is uniformly effectively quasiconvex, meaning that, in the definition \eqref{eqn:qc} of effective quasiconvexity, a single function $h$ can be used for all of the metrics $\rho_n$).
\end{enumerate}
Then, $(A,\rho_\infty)$ is tracially lyriform and is isomorphic to the $\lyre$ direct limit of the system $((A_n,\rho_n),\varphi_n)_{n\in\nn}$. It is also complete if each $(A_n,\rho_n)$ is. This in particular applies to inductive sequences of tracial Prokhorov spaces, and also, for any fixed $\fp\in[1,\infty]$, to sequences of tracial $\fp$-Wasserstein spaces of bounded diameter.
\end{proposition}

\begin{proof}
The listed hypotheses guarantee that $\rho_\infty$ is effectively quasiconvex and $\diam(T(A),\rho_\infty)<\infty$. By Proposition~\ref{prop:wass}, both conditions hold for tracial Prokhorov spaces and for tracial $\fp$-Wasserstein spaces of bounded diameter.  The expression \eqref{eqn:limitrho} is lower semicontinuous, as it defines $\rho_\infty$ as a pointwise supremum of an increasing sequence of lower semicontinuous functions on $T(A)$. It also ensures that the inclusion maps $\varphi_{n,\infty}\colon A_n \to A$ are tracially nonexpansive. In particular, $\rho_n$-Lipschitz elements in any $A_n$ are mapped to $\rho_\infty$-Lipschitz elements in the limit, so $L_{\rho_\infty}$ is indeed densely finite. This verifies that $(A,\rho_\infty)$ is a tracially lyriform $\cs$-algebra. Suppose next that $(B,\rho_B)$ is tracially lyriform and $(\psi_n\colon A_n \to B)_{n\in\nn}$ is a compatible system of $\lyre$-morphisms. Let $\psi\colon A\to B$ be the induced unital $^*$-monomorphism. Then, for $\sigma,\tau\in T(B)$ we have
\[
\rho_\infty(T(\psi)(\sigma),T(\psi)(\tau)) = \sup_{n\in\nn} \rho_n(T(\psi_n)(\sigma),T(\psi_n)(\tau)) \le \rho_B(\sigma,\tau).
\]
In other words, $\psi$ is tracially nonexpansive relative to $\rho_B$ and $\rho_\infty$. This shows that $(A,\rho_\infty)$ has the required universal property to be the direct limit of the system in $\lyre$. Finally, if each $(A_n,\rho_n)$ is complete, then completeness of $(A,\rho_\infty)$ is proved in the same way that one shows completeness of $\ell_\infty$.
\end{proof}

\begin{example} \label{ex:wlim}
Let $(A,W_{1,\rho})$ be a tracial Wasserstein space observing a compact metric space $(X,\rho)$. By Proposition~\ref{prop:inherit}, $(A,W_{\fp,\rho})$ is a tracial $\fp$-Wasserstein space for every $\fp\in[1,\infty]$, and by Proposition~\ref{prop:colimits} and equation \eqref{eqn:winf2}, $(A,W_{\infty,\rho}) \cong \lim_{\fp\to\infty} ((A,W_{\fp,\rho}),\id_A)$ in $\lyre$. So, tracial $\infty$-Wasserstein spaces are special cases of $\lyre$-limits.
\end{example} 

Recall from Theorem~\ref{thm:compact} that we can view a quantum metric Choquet simplex as a special case of a compact tracially lyriform $\cs$-algebra for which the metric on the trace space is induced by the associated seminorm.

\begin{definition} \label{def:qapprox}
A tracially lyriform $\cs$-algebra is said to be an \emph{approximate $\cs$-algebraic quantum metric Choquet} (or \emph{Bauer}) \emph{simplex} if it is isomorphic in $\lyre$ to the inductive limit of a sequence of $\cs$-algebraic quantum metric Choquet simplices (respectively, Bauer simplices). We similarly define \emph{approximate tracial Prokhorov spaces} and \emph{approximate tracial $\fp$-Wasserstein spaces} for $\fp\in[1,\infty]$.
\end{definition}

An \emph{approximate} $\cs$-algebraic quantum metric Choquet simplex is typically not itself a quantum metric Choquet simplex. The obstruction is that the topology induced by the limit metric \eqref{eqn:limitrho} is generally stronger than the $\ws$-topology on $T(A)$, which is given by the product topology on $\varprojlim T(A_n) \subseteq \prod_{n\in\nn} T(A_n)$. If securing this topology is part of the goal, one can achieve this as in \cite{Long:2025aa} (see also \cite[Proposition 1.1]{Rieffel:2002aa} and \cite[Example 2.15.3]{Jacelon:2024aa}) by replacing $\rho_\infty$ by a metric of the form
\begin{equation} \label{eqn:rhosum}
\rho_\Sigma((\sigma_n)_{n=1}^\infty,(\tau_n)_{n=1}^\infty) = \sum_{n=1}^\infty\omega_n\rho_n(\sigma_n,\tau_n)
\end{equation}
for some summable sequence $(\omega_n)_{n=1}^\infty$ of strictly positive real numbers (here assuming purely for the sake of simplicity that $\sup_{n\in\nn}\diam(T(A_n),\rho_n) < \infty$). This is not the approach we take in this article because our priority is not the topology of the trace space of $A = \varinjlim A_n$ but rather the geometry of embedding spaces $\emb(A,B)$. What we need for that is a core \eqref{eqn:core} for $(A,\rho_\infty)$ that is built as a union of cores for the building blocks $(A_n,\rho_n)$ so that local geometry can be made global. That is why we opt to work in the lyre category, that is, with morphisms that are tracially nonexpansive.

The general situation notwithstanding, there is a notable special case (covering many of our foundational examples) in which the limit is indeed a $\cs$-algebraic quantum metric Choquet simplex on the nose.

\begin{proposition} \label{prop:lyrelimit}
Let $((A_n,\rho_n),\varphi_n \colon A_n \to A_{n+1})_{n\in\nn}$ be a tracially nonexpansive system of compact tracially lyriform $\cs$-algebras. Suppose that there is a summable sequence $(\eps_n)_{n=1}^\infty$ of positive real numbers such that for every $n\in\nn$, $T(\varphi_n)$ is $\eps_n$-isometric, meaning that $|\rho_n(T(\varphi_n)(\sigma),T(\varphi_n)(\tau)) - \rho_{n+1}(\sigma,\tau)| < \eps_n$ for every $\sigma,\tau\in T(A_{n+1})$. Then, the lyre limit $(A=\varinjlim(A_n,\varphi_n),\rho_\infty)$ is also compact.
\end{proposition}

\begin{proof}
We must show that $\rho_\Sigma$-convergent sequences in $T(A)$ are $\rho_\infty$-convergent (where $\rho_\Sigma$ is as in \eqref{eqn:rhosum}, say with $\omega_n=2^{-n}$). Let $(\sigma_n)_{n=1}^\infty \in T(A)$, let $\eps>0$ and let $N\in\nn$ be such that $\sum_{k=N}^\infty \eps_k < \frac{\eps}{2}$. If $(\tau_n)_{n=1}^\infty\in T(A)$ is sufficiently close to $(\sigma_n)_{n=1}^\infty$ (relative to $\rho_\Sigma$), then $\rho_n(\sigma_n,\tau_n) < \frac{\eps}{2}$ for every $n\in\{1,\dots,N\}$. But then for $n>N$ we have
\begin{align*}
\rho_N(\sigma_N,\tau_N) &= \rho_N(T(\varphi_{N}) \circ \dots \circ T(\varphi_{n-1})(\sigma_n),T(\varphi_{N}) \circ \dots \circ T(\varphi_{n-1})(\tau_n))\\
&> \rho_n(\sigma_n,\tau_n) - \sum_{k=N}^{n-1} \eps_k\\
&> \rho_n(\sigma_n,\tau_n) - \frac{\eps}{2}.
\end{align*}
This implies that $\rho_n(\sigma_n,\tau_n) < \rho_N(\sigma_N,\tau_N)  + \frac{\eps}{2} < \eps$ and hence that 
\[
\rho_\infty((\sigma_n)_{n=1}^\infty,(\tau_n)_{n=1}^\infty) = \sup_{n\in\nn}\rho_n(\sigma_n,\tau_n) < \eps.
\]
It follows that $\rho_\Sigma$-convergent sequences are indeed $\rho_\infty$-convergent, which completes the proof.
\end{proof}

\begin{example} \label{ex:models}
Let us use the models $\js_X=\varinjlim(A_n,\varphi_n)$ constructed in Theorem~\ref{thm:models} to illustrate the ideas developed above, particularly Proposition~\ref{prop:lyrelimit}. Here, each $A_n$ is a generalised dimension drop algebra $X_{p_n,q_n}$ over a compact metric space $(X,\rho)$. As described in Corollary~\ref{cor:wasslyre}, for any $\fp\in[1,\infty]$, $(A_n,W_{\fp,\rho})$ is a tracial $\fp$-Wasserstein space whose canonical nucleus of radius $r\ge r_{(X,\rho)}$ is
\[
\mathcal{D}_r(A_n) = \{f \in (A_n)_{sa} \mid \|f\|\le r,\: \|f(x)-f(y)\|\le\rho(x,y)\:\text{ for every } x,y\in X\}.
\]
Although the connecting maps $\varphi_n\colon A_n \to A_{n+1}$ are in general not tracially $W_{1,\rho}$-contractive, they are approximately so and the discrepancy $\eps_n$ is summably small. In fact, $T(\varphi_n)$ is within $\eps_n$ of the identity map relative to both the uniform metric and Lipschitz seminorm associated with $W_{1,\rho}$. To realise $(\js_X,\rho)$ as a limit in $\lyre$, we can correspondingly adjust the metric $\rho$ at each stage of the inductive sequence. More precisely, we replace the seminorm $L_\rho$ by $L_\rho^{(n)}:=\sup_{m\ge n}L_\rho\circ\varphi_{n,m}$. Equivalently, we replace $\rho$ by $\rho^{(n)}$, where
\begin{equation} \label{eqn:remetric}
\rho^{(n)}(\sigma,\tau) = \sup\{\sigma(a)-\tau(a) \mid a\in(A_n)_{sa},\: L_\rho^{(n)}(a)\le1\}.
\end{equation}
Since $\sum_{n=1}^\infty\eps_n<\infty$, the infinite product $\prod_{n=1}^\infty(1+\eps_n)$ also converges. Writing $\delta_n=\prod_{m=n}^\infty(1+\eps_m)$, we have $L_\rho \le L_\rho^{(n)} \le \delta_nL_\rho$ and
\begin{equation} \label{eqn:rhon}
\delta_n^{-1}\rho \le \rho^{(n)} \le \rho.
\end{equation}
So, $(\rho^{(n)})_{n\in\nn}$ is a sequence of uniformly Lipschitz equivalent metrics on $X$ that converges uniformly to $\rho$. By construction, the maps $T(\varphi_n)\colon (T(A_{n+1}),W_{1,\rho^{(n+1)}}) \to (T(A_n),W_{1,\rho^{(n)}})$ are all nonexpansive, and $(\partial_e(\varprojlim(T(A_n),T(\varphi_n))),\rho^{(\infty)})$ is isometrically isomorphic to $(X,\rho)$. We conclude that $(\js_X,W_{1,\rho})$ is isomorphic in $\lyre$ to the direct limit of the sequence $(A_n,W_{1,\rho^{(n)}})$. In the language of Definition~\ref{def:qapprox}, we have exhibited $(\js_X,W_{1,\rho})$ as an approximate $\cs$-algebraic quantum metric Bauer simplex.

For use in Section~\ref{section:geostats}, we introduce the modified nuclei
\begin{equation} \label{eqn:modify}
\mathcal{D}^{(n)}_r(X_{p_n,q_n}) = \{a \in \mathcal{D}_r(X_{p_n,q_n}) \mid L_\rho^{(n)}(a) \le 1\}
\end{equation}
which by \eqref{eqn:rhon} satisfy
\begin{equation} \label{eqn:inclusions}
\delta_n^{-1}\mathcal{D}_r(X_{p_n,q_n}) \subseteq \mathcal{D}^{(n)}_r(X_{p_n,q_n}) \subseteq \mathcal{D}_r(X_{p_n,q_n}).
\end{equation}
The \emph{canonical core} (of radius $r\ge r_{(X,\rho)})$ associated with this modification is
\begin{equation} \label{eqn:unioncore}
\mathcal{L}_r(\js_X) =  \bigcup_{n\in\nn} \varphi_{n,\infty}\left(\mathcal{D}^{(n)}_r(X_{p_n,q_n})\right)
\end{equation}
and the \emph{pro-$W_\infty$ metric} on $\varprojlim T(A_n)$ is
\begin{equation} \label{eqn:prowinfty1}
\overrightarrow{W}_{\infty,\rho}((\sigma_n)_{n=1}^\infty,(\tau_n)_{n=1}^\infty) = \limsup_{n\to\infty} W_{\infty,\rho^{(n)}}(\sigma_n,\tau_n) = \limsup_{n\to\infty} W_{\infty,\rho}(\sigma_n,\tau_n).
\end{equation}
Unlike the situation with $W_1$, it is not clear whether $(\varprojlim T(A_n),\overrightarrow{W}_{\infty,\rho})$ is isometrically isomorphic to $(T(\js_X),W_{\infty,\rho})$. But the pro-$W_\infty$ metric gives the right lyriform structure for us to be able to patch together the local geometry provided by embedding spaces $\emb(X_{p_n,q_n},B)/\sim_{au}$ into a global geometric picture of $\emb(\js_X,B)/\sim_{au}$. See Corollary~\ref{cor:dimdroporbits} and Corollary~\ref{cor:spiderorbits}.
\end{example}

\section{Tracial exoskeletons} \label{section:arachnid}

In this section, use classifiable tracial Wasserstein spaces observing spheres and balls to fit simple inductive limits of prime dimension drop algebras into the framework presented in Section~\ref{section:limits}.

\begin{lemma} \label{lemma:nsimplex}
Let $(X,\rho)$ be a finite metric space consisting of $n\ge1$ equidistant points $x_1,\dots,x_n$. Then,
\[
W_1(\mu,\nu) = \|\bar\mu-\bar\nu\|_1 \cdot r_{(X,\rho)}
\]
for every $\mu=\sum_{i=1}^n\mu_i\delta_{x_i}, \nu=\sum_{i=1}^n\nu_i\delta_{x_i}\in\prob(X)$, where $\bar\mu=(\mu_1,\dots,\mu_n)$ and $\bar\nu=(\nu_1,\dots,\nu_n)$.
\end{lemma}

\begin{proof}
Suppose that $f\colon X\to \rr$ is $1$-Lipschitz. Write $r=r_{(X,\rho)}$. Then, $|f(x_i)-f(x_j)|\le2r$ for every $i,j\in\{1,\dots,n\}$, so there exists $c\in\rr$ such that the range of the $1$-Lipschitz function $\tilde{f}:=f-c$ is contained in the interval $[-r,r]$. It follows that
\begin{align*}
\left|\int_Xf\,d\mu - \int_Xf\,d\nu\right| &= \left|\int_X\tilde{f}\,d\mu - \int_X\tilde{f}\,d\nu\right|\\
&= \left|\sum_{i=1}^n\mu_i\tilde{f}(x_i) - \sum_{i=1}^n\nu_i\tilde{f}(x_i)\right|\\
&\le \sum_{i=1}^n|\mu_i-\nu_i| \cdot |\tilde{f}(x_i)|\\
&\le \|\bar \mu-\bar \nu\|_1 \cdot r
\end{align*}
and hence by \eqref{eqn:w1} that $W_1(\mu,\nu)\le\|\bar \mu-\bar \nu\|_1\cdot r$. Note that equality is achieved by the $1$-Lipschitz function $f\colon X\to [-r,r]$ defined by $f(x_i) = \sgn(\mu_i-\nu_i)\cdot r$, thus proving that $W_1(\mu,\nu) = \|\bar \mu-\bar \nu\|_1\cdot r$ as claimed.
\end{proof}

\begin{lemma} \label{lemma:lipschitzextension}
Let $(A_1,W_{1,\rho_1})$ and $(A_2,W_{1,\rho_2})$ be tracial Wasserstein spaces observing compact metric spaces $(X_1,\rho_1)$ and $(X_2,\rho_2)$, respectively. Then, every $1$-Lipschitz map $f\colon (X_1,\rho_1) \to (T(A_2),W_{1,\rho_2})$ admits a unique extension to a $1$-Lipschitz affine map $(T(A_1),W_{1,\rho_1}) \to (T(A_2),W_{1,\rho_2})$. 
\end{lemma}

\begin{proof}
Suppose that $\mu$ and $\nu$ are finitely supported measures on $X_1$, say $\mu=\sum_{i=1}^ns_ix_i, \nu=\sum_{i=1}^nt_ix_i\in\prob(X_1)\cong T(A_1)$. We define $f(\mu)=\sum_{i=1}^ns_if(x_i)$ and $f(\nu)=\sum_{i=1}^nt_if(x_i)$. For each $i$, let $\lambda_i\in\prob(X_2)\cong T(A_2)$ be the representing measure of $f(x_i)$, so that $\sum_{i=1}^ns_i\lambda_i$ and $\sum_{i=1}^nt_i\lambda_i$ are representing measures of $f(\mu)$ and $f(\nu)$, respectively. Let $h\colon X_2\to\rr$ be $1$-Lipschitz. By \eqref{eqn:w1}, its unique continuous affine extension to $T(A_2)$, defined by $\tilde h\colon\lambda\mapsto\int_{X_2}h\,d\lambda$, is also $1$-Lipschitz. Then,
\begin{align*}
W_{1,\rho_2}(f(\mu),f(\nu)) &= \sup_{h\in\lip(X_2,\rho_2)}\left|\int_{X_2}h\,df(\mu) - \int_{X_2}h\,df(\nu)\right|\\
&= \sup_{h\in\lip(X_2,\rho_2)}\left|\sum_{i=1}^ns_i\int_{X_2}h\,d\lambda_i - \sum_{i=1}^nt_i\int_{X_2}h\,d\lambda_i\right|\\
&= \sup_{h\in\lip(X_2,\rho_2)}\left|\sum_{i=1}^ns_i\tilde h(f(x_i)) - \sum_{i=1}^nt_i\tilde h(f(x_i))\right|\\
&= \sup_{h\in\lip(X_2,\rho_2)}\left|\int_{X_1}\tilde h\circ f\,d\mu - \int_{X_1}\tilde h\circ f\,d\nu\right|\\
&\le \sup_{g\in\lip(X_1,\rho_1)}\left|\int_{X_1}g\,d\mu - \int_{X_1}g\,d\nu\right|\\
&= W_{1,\rho_1}(\mu,\nu).
\end{align*}
Since a general $\mu\in\prob(X_1)$ is a limit of measures that are finitely supported, it follows that there is a well-defined and unique extension of $f$ to a continuous affine map $(T(A_1),W_{1,\rho_1}) \to (T(A_2),W_{1,\rho_2})$, and that this extension is $1$-Lipschitz.
\end{proof}

Recall  from Proposition~\ref{prop:colimits} that the category of tracially lyriform $\cs$-algebras admits direct limits of inductive sequences of tracial $\fp$-Wasserstein spaces of bounded diameter. Theorem~\ref{thm:arachnid} says that any simple inductive limit of prime dimension drop algebras can be constructed in this way. In its statement, we write $W_{1,\rho}$ and $W_{1,\omega}$ for the limit metrics \eqref{eqn:limitrho} associated with $(W_{1,\rho_{k}})_{k\in\nn}$ and $(W_{1,\omega_{k}})_{k\in\nn}$.

\begin{theorem} \label{thm:arachnid} 
Let $A$ be a simple $\cs$-algebra that is isomorphic to an inductive limit of prime dimension drop algebras. Then, there are inductive systems of classifiable tracial Wasserstein spaces $((\js_{\sphere^{2k-1}},W_{1,\rho_{k}}),\varphi_k)_{k\in\nn}$ and $((\js_{\ball^{2k}},W_{1,\omega_{k}}),\psi_k)_{k\in\nn}$ in $\lyre$ such that $(\varinjlim (\js_{\sphere^{2k-1}},\varphi_k),W_{1,\rho})$ and $(\varinjlim (\js_{\ball^{2k}},\psi_k),W_{1,\omega})$ provide isomorphic tracially lyriform structures on $A$, that is, there is a $^*$-isomorphism $\varinjlim (\js_{\sphere^{2k-1}},\varphi_k) \to \varinjlim (\js_{\ball^{2k}},\psi_k) \cong A$ that induces an isometric isomorphism of trace spaces with respect to the limit metrics $W_{1,\omega}$ and $W_{1,\rho}$ associated with $(W_{1,\omega_{k}})_{k\in\nn}$ and $(W_{1,\rho_{k}})_{k\in\nn}$. If $\partial_eT(A)$ is finite, then we can take stationary inductive limits and the common tracially lyriform $\cs$-algebra is a tracial $1$-Wasserstein space.
\end{theorem}

\begin{proof}
To prove the theorem, we first construct suitable intertwinings to show that the metrisable Choquet simplex $T(A)$ can be written as an inverse limit of probability spaces over spheres $\sphere^{2k-1}$ or balls $\ball^{2k}$. Then we use classification to lift to sequences $(\js_{\sphere^{2k-1}},\varphi_k)_{k\in\nn}$ and $(\js_{\ball^{2k}},\psi_k)_{k\in\nn}$.

We begin by writing $T(A)$ as an inverse limit $\varprojlim (\Delta_k,\eta_k)$ of finite-dimensional simplices $\Delta_k$. If $\dim T(A) = N < \infty$, then we can take a stationary sequence $(\Delta_N,\eta_k)_{k\in\nn}$, where $\Delta_N=\conv\{e_1,\dots,e_{N+1}\}\subseteq\rr^{N+1}$ is the $N$-simplex whose extreme points are the elements of the standard basis of $\rr^{N+1}$ and $\eta_k=\id_{\Delta_k}$ for every $k\in\nn$. If $T(A)$ is infinite dimensional, then we set $\Delta_k=\conv\{e_1,\dots,e_{k+1}\}\subseteq\rr^{k+1}$ and use the Lazar--Lindenstrauss theorem \cite[Theorem 5.2 and its corollary]{Lazar:1971kx} to find surjective affine maps $\eta_k\colon\Delta_{k}\to\Delta_{k-1}$ such that $\varprojlim (\Delta_k,\eta_k) \cong T(A)$.

Next, for each $k\in\nn$ we let:
\begin{items}
\item $\sphere^{k-1}\subseteq\rr^{k+1}$ be the circumscribed sphere of $\Delta_k$ centred at the simplex's barycentre $(\frac{1}{k+1},\dots,\frac{1}{k+1})$;
\item $\ball^k\subseteq\rr^{k+1}$ be the convex hull of $\sphere^{k-1}$;
\item $\widehat\pi_k\colon\ball^k\to\Delta_k$ be a homeomorphism that restricts to a homeomorphism $\pi_k$ of $\sphere^{k-1}$ onto the $(k-1)$-skeleton of the $k$-simplex $\Delta_k$, and that fixes $\partial_e\Delta_k$.
\end{items}
We equip each $\Delta_k$ with the $\ell_1$-metric $\delta_k(x,y)=\|x-y\|_1$. As observed in \cite{Blackadar:1980zr}, each $\eta_k\colon\Delta_{k}\to\Delta_{k-1}$ is then nonexpansive. Then we equip:
\begin{items}
\item $\ball^k$ with the pullback metric $\omega_k(x,y)=\|\widehat\pi_k(x)-\widehat\pi_k(y))\|_1$;
\item $\sphere^{k-1}$ with the intrinsic metric $\rho_k(x,y)$ induced by the pullback metric $(x,y)\mapsto\|\pi_k(x)-\pi_k(y))\|_1$
\end{items}
both of which are strictly intrinsic (see Example~\ref{ex:normed} and Example~\ref{ex:intrinsic}).

By Lemma~\ref{lemma:lipschitzextension}, we can extend $\pi_k$ and $\widehat\pi_k$ to $1$-Lipschitz maps $\pi_k\colon\prob(\sphere^{k-1})\to\Delta_k$ and $\widehat\pi_k\colon\prob(\ball^{k})\to\Delta_k$. Note that the inclusions $\partial_e\Delta_k\hookrightarrow\sphere^{k-1}$ and $\partial_e\Delta_k\hookrightarrow\ball^{k}$ are both isometries, the latter by definition of $\omega_k$ and the former because the edges of $\Delta_k$ are shortest paths, so $\rho_k(e_i,e_j)=\|e_i-e_j\|_1$ for every $i,j$. Affinely extending these inclusions from $\partial_e\Delta_k$ to $\Delta_k$ by defining $\iota_k(\mu_1,\dots,\mu_{k+1}) = \sum_{i=1}^{k+1}\mu_i\delta_{e_i}$ and similarly for $\widehat\iota_k$, we get right inverses $\iota_k\colon \Delta_k \to \prob(\sphere^{k-1})$ and $\widehat\iota_k\colon \Delta_k \to \prob(\ball^{k})$ for $\pi_k$ and $\widehat\pi_k$. By Lemma~\ref{lemma:nsimplex}, $\iota_k$ and $\widehat\iota_k$ are isometric isomorphisms onto their images.

By Corollary~\ref{cor:ballsandspheres}, we can find projectionless, classifiable quantum metric Bauer simplices $(\js_{\sphere^{2k-1}},W_{1,\rho_{2k}})$ and $(\js_{\ball^{2k}},W_{1,\omega_{2k}})$ observing $(\sphere^{2k-1},\rho_{2k})$ and $(\ball^{2k},\omega_{2k})$ respectively, such that
\[
(K_0(\js_{\ball^{2k}}),K_0(\js_{\ball^{2k}})_+,[1_{\js_{\ball^{2k}}}],K_1(\js_{\ball^{2k}})) \cong (\zz,\nn,1,0)
\]
and
\[
(K_0(\js_{\sphere^{2k-1}}),K_0(\js_{\sphere^{2k-1}})_+,[1_{\js_{\sphere^{2k-1}}}]) \cong (\zz,\nn,1).
\]
Let us first assume that $T(A)$ is infinite dimensional. In this case, for every $k\in\nn$ we set $A_k=\js_{\sphere^{2k-1}}$ and $B_k=\js_{\ball^{2k}}$ (together with their respective $1$-Wasserstein metrics) and define nonexpansive affine maps
\begin{align}
\kappa_k &= \iota_{k-1}\circ\eta_k\colon\Delta_k\to\prob(\sphere^{k-2}) \label{eqn:intertwine}\\
\zeta_k &= \kappa_k\circ\pi_k\colon\prob(\sphere^{k-1})\to\prob(\sphere^{k-2})\notag\\
\lambda_k &= \eta_k\circ\widehat\pi_k\colon\prob(\ball^k)\to\Delta_{k-1}\notag\\
\theta_k &= \widehat\iota_{k-1}\circ\lambda_k\colon\prob(\ball^k)\to\prob(\ball^{k-1}).\notag
\end{align}
By classification \cite[Corollary C]{Carrion:wz}, we can lift
\[
\widehat\varphi_k := \zeta_{2k+1}\circ\zeta_{2k+2}\colon T(A_{k+1})\cong\prob(\sphere^{2k+1})\to\prob(\sphere^{2k-1})\cong T(A_k)
\]
and
\[
\widehat\psi_k := \theta_{2k+1}\circ\theta_{2k+2}\colon T(B_{k+1})\cong\prob(\ball^{2k+2})\to\prob(\ball^{2k})\cong T(B_k)
\]
to unital, $K_1$-trivial $^*$-homomorphisms $\varphi_k\colon A_k\to A_{k+1}$ and $\psi_k\colon B_k\to B_{k+1}$, respectively. We therefore obtain the following commutative diagram, in which all arrows are $1$-Lipschitz. Note that the pictures only indicate the \emph{extreme points} of the images of the simplices $\Delta_k$ under the inclusions $\iota_k$ and $\widehat\iota_k$, while interior points are contained in the infinite-dimensional probability spaces $\prob(\sphere^{k-1})$ and $\prob(\ball^{k})$.
\[
\begin{tikzcd}
\eqmathbox[N]{\dots}
&
\eqmathbox[M]{
\begin{tikzpicture}
  \draw[thick, SeaGreen] (0,0) circle(1cm);
  \foreach \angle in {90,210,330}
    \coordinate (T\angle) at (\angle:1cm);
  \draw[thick, blue] (T90) -- (T210) -- (T330) -- cycle;
\end{tikzpicture}
}
\arrow[d,"\pi_{k-1}"] \arrow[l,"\zeta_{k-1}"]
&
\eqmathbox[M]{
\begin{tikzpicture}
  \draw[thick, SeaGreen] (0,0) circle(1cm);
  \coordinate (A) at (90:1cm);
  \coordinate (B) at (210:1cm);
  \coordinate (C) at (330:1cm);
  \coordinate (D) at (0,0);
  \draw[dashed, thick, blue] (D) -- (A);
  \draw[dashed, thick, blue] (D) -- (B);
  \draw[dashed, thick, blue] (D) -- (C);
  \draw[thick, blue] (A) -- (B) -- (C) -- cycle;
\end{tikzpicture}
}
\arrow[l,"\zeta_{k}"] \arrow[d,"\pi_{k}"]
&
\eqmathbox[N]{\dots} \arrow[l,"\zeta_{k+1}"]
&
\eqmathbox[L]{~T(\varinjlim A_k))} \arrow[l] \arrow[d, dashed,"~\cong"]\\
\eqmathbox[N]{\dots}
&
\eqmathbox[M]{
\begin{tikzpicture}
  \foreach \angle in {90,210,330}
    \coordinate (T\angle) at (\angle:1cm);
  \draw[thick, blue] (T90) -- (T210) -- (T330) -- cycle;
\end{tikzpicture}
}
\arrow[d,"\widehat\iota_{k-1}"] \arrow[l,"\eta_{k-1}"]
&
\eqmathbox[M]{
\begin{tikzpicture}
  \coordinate (A) at (90:1cm);
  \coordinate (B) at (210:1cm);
  \coordinate (C) at (330:1cm);
  \coordinate (D) at (0,0);
  \draw[dashed, thick, blue] (D) -- (A);
  \draw[dashed, thick, blue] (D) -- (B);
  \draw[dashed, thick, blue] (D) -- (C);
  \draw[thick, blue] (A) -- (B) -- (C) -- cycle;
\end{tikzpicture}
}
\arrow[l,"\eta_{k}"] \arrow[d,"\widehat\iota_{k}"] \arrow[ul,"\kappa_{k}"]
&
\eqmathbox[N]{\dots} \arrow[l,"\eta_{k+1}"]
&
\eqmathbox[L]{~T(A)} \arrow[l] \arrow[d, dashed,"~\cong"] \arrow[u,dashed,xshift=0.15cm]\\
\eqmathbox[N]{\dots}
&
\eqmathbox[M]{
\begin{tikzpicture}
  \draw[thick, SeaGreen,fill=SeaGreen!80!black, opacity=0.7] (0,0) circle(1cm);
  \foreach \angle in {90,210,330}
    \coordinate (T\angle) at (\angle:1cm);
  \draw[thick, blue] (T90) -- (T210) -- (T330) -- cycle;
\end{tikzpicture}
}
\arrow[l,"\theta_{k-1}"]
&
\eqmathbox[M]{
\begin{tikzpicture}
  \draw[thick, SeaGreen] (0,0) circle(1cm);
  \shade[ball color=SeaGreen!90!black, opacity=0.7] (0,0) circle(1cm);
  \coordinate (A) at (90:1cm);
  \coordinate (B) at (210:1cm);
  \coordinate (C) at (330:1cm);
  \coordinate (D) at (0,0);
  \draw[dashed, thick, blue] (D) -- (A);
  \draw[dashed, thick, blue] (D) -- (B);
  \draw[dashed, thick, blue] (D) -- (C);
  \draw[thick, blue] (A) -- (B) -- (C) -- cycle;
\end{tikzpicture}
}
\arrow[l,"\theta_{k}"] \arrow[ul,"\lambda_{k}"]
&
\eqmathbox[N]{\dots} \arrow[l,"\theta_{k+1}"]
&
\eqmathbox[L]{T(\varinjlim B_k)} \arrow[l] \arrow[u,dashed,xshift=0.15cm]
\end{tikzcd}
\]
We get induced isometric isomorphisms
\[
(\varprojlim T(A_k),W_{1,\rho}) \cong (T(A),\delta_\infty) \cong (\varprojlim T(B_k),W_{1,\omega})
\]
which can be lifted via classification to $^*$-isomorphisms $\varinjlim A_k \cong \varinjlim B_k \cong A$. 

Finally, we consider the case where $N=\dim T(A)$ is finite. In this situation, we can take stationary inductive limits of suitably trace-collapsing endomorphisms. Moreover, the metrics $W_{1,\rho}$ and $W_{1,\omega}$ correspond to the $\ell_1$-metric on $T(A)\cong\Delta_N$, and the construction produces a tracial $1$-Wasserstein space observing a finite metric space consisting of $N+1$ equidistant points, thus verifying the last statement of the theorem. More precisely, let $M$ be the minimal element of $(2\nn+2)\cap[N,\infty)$. Then for every $k\in\nn$, we set $A_k=\js_{\sphere^{M-1}}$ and $B_k=\js_{\ball^{M}}$ (again with their respective $1$-Wasserstein metrics). We define the intertwining maps just as in \eqref{eqn:intertwine}, with $\eta_k=\id_{\Delta_N}$, with $\iota_k$ and $\widehat\iota_k$  adjusted by pre-composing them with the inclusion $\Delta_N\hookrightarrow\Delta_{M}$, and with $\pi_k$ and $\widehat\pi_k$ similarly adjusted by post-composing them with a collapsing surjection $\Delta_{M}\to\Delta_{N}$ (equal to $\id_{\Delta_N}$ if $M=N$). As before, we complete the proof by using classification to lift to a diagram in the category of tracially lyriform $\cs$-algebras.
\end{proof}

I like to think of the above construction as an `arachnid model' for the $\cs$-algebra $A$, with the sequence of balls and spheres imagined as the segments and exoskeleton of a (potentially infinite-dimensional) spider. 

\begin{definition} \label{def:simplicial}
We refer to the spheres and balls being observed in Theorem~\ref{thm:arachnid} (with their respective $\ell_1$-derived metrics) as \emph{simplicial}. We write $\mathfrak{ball}(A)$ and $\mathfrak{sphere}(A)$ for the union of the images in $A$ of the canonical cores \eqref{eqn:unioncore} (for some fixed radius $r\ge r_{(\ball^{2k},\omega_k)}=1$ or $r\ge r_{(\sphere^{2k-1},\rho_k)}=2$) of the building blocks $\js_{\ball^{2k}}$ and $\js_{\sphere^{2k-1}}$, respectively.
\end{definition}

\begin{remark}
All of the the spheres and balls appearing in Theorem~\ref{thm:arachnid} can be assumed to have dimension at least three (explicitly arranged in the finite-dimensional case, and without loss of generality in the infinite-dimensional case). The significance of this is that we then have access to the continuous optimal transport theory developed in \cite{Jacelon:2021wa,Jacelon:2021vc} (see Theorem~\ref{thm:transport} and Theorem~\ref{thm:orbits}). But if $N=\dim T(A)\in\{0,1\}$, we can also use the one-dimensional spaces $\sphere^1$ and $\ball^1$. (Since Theorem~\ref{thm:alexandrov} does not apply to $S^{-1}$ or $S^0$, the minimal dimension for which the proof of Theorem~\ref{thm:arachnid} works is one.) Here, for every $k\in\nn$ we set $A_k=\js_{\sphere^{1}}$ and $B_k=\js_{\ball^{1}}$. If $N=0$, we take $\widehat\varphi_k\colon T(A_{k+1}) \to T(A_k)$ and $\widehat\psi_k\colon T(B_{k+1}) \to T(B_k)$ to be constant. If $N=1$, instead of the homeomorphism $\pi_2$ that maps $\sphere^1$ onto the $1$-skeleton of $\Delta_2$, we take the double cover $\pi'_2\colon\sphere^1\to\Delta_2\to\Delta_1$ given by composing $\pi_2$ with a collapsing map $\Delta_2\to\Delta_1$. In other words, the line segment $\Delta_1$ joining $e_1$ and $e_2$ divides $\sphere^1$ into two arcs, each mapped homeomorphically by $\pi'_2$ onto $\Delta_1$. We then replace the metric $\rho_2$ by the circular metric $\rho'_2$ on $\sphere^1$ defined by $\rho'_2(x,y) = \delta_1(\pi'_2(x),\pi'_2(y)) = \|\pi'_2(x)-\pi'_2(y)\|_1$ if $x$ and $y$ are in the same arc and 
\[
\rho'_2(x,y) = \min\{\delta_1(\pi'_2(x),e_1)+\delta_1(e_1,\pi'_2(y)),\delta_1(\pi'_2(x),e_2)+\delta_1(e_2,\pi'_2(y))\}
\]
if $x$ and $y$ are in different arcs. We also adjust $\iota_k$ by pre-composing with the inclusion $\Delta_1\hookrightarrow\Delta_{2}$, once again define the intertwining maps as in \eqref{eqn:intertwine} with $\eta_k=\id_{\Delta_1}$ for every $k\in\nn$, and then lift via classification.

Note finally that $\dim T(A)=0$ precisely when $A$ is isomorphic to the Jiang--Su algebra $\js$, now reconstructed as a limit of generalised dimension drop algebras over simplicial spheres or balls.
\end{remark}

\section{The geometry and statistics of embedding spaces} \label{section:geostats}

In this section, we look at unital embeddings of projectionless models $\js_X$ over balls $X=\ball^k$ or odd spheres $X=\sphere^{k-1}$ into suitable stably finite classifiable codomains. As it turns out, it is the (pro-)$\infty$-Wasserstein distance that dictates the geometry of unitary orbits of these embeddings. Let us assume for the moment that the codomain $B$ is monotracial, and also real rank zero with $K_1(B)=0$ if we are talking about spheres $X=\sphere^{k-1}$ (otherwise, we would need to worry about Hausdorffised algebraic $K_1$). By classification \cite[Theorem B]{Carrion:wz} and the $K$-theory computation Corollary~\ref{cor:ballsandspheres}, the approximate unitary equivalence classes of these embeddings are parameterised by traces on $\js_X$, that is, by elements of the probability space $\prob(X)$. (Remember, $\js_X$ is a simple $\cs$-algebra, so all traces are faithful, not just the ones represented by faithful measures on $X$.) So, the embedding space $\emb(\js_X,B)/\sim_{au}$ has the structure of a Bauer simplex. Well, not quite, because the right topology on $\emb(\js_X,B)/\sim_{au}$ is the topology of pointwise convergence modulo unitary conjugation. This is \emph{not} the $\ws$-topology but rather the pro-$W_\infty$-topology. Indeed, by Corollary~\ref{cor:dimdroporbits}, the unitary distance $d_U(\varphi,\psi)|_{\mathcal{L}_r(\js_X)}$ between $\sim_{au}$-classes of embeddings $\varphi,\psi\colon\js_X\to B$ (see Definition~\ref{def:unitarydistance}) is \emph{equal} to $\overrightarrow{W}_\infty(\tau_B\circ\varphi,\tau_B\circ\psi)$. So, $\overrightarrow{W}_\infty$ will provide not just the topology but the geometry as well.

\subsection{Noncommutative continuous transport} \label{subsection:transport}

Recall from Proposition~\ref{prop:inherit} that every tracial Wasserstein space observing a compact metric space $(X,\rho)$ can be viewed as a tracial $\fp$-Wasserstein space for any $\fp\in[1,\infty]$. For the definition of a core and a nucleus as well as relevant examples, see \eqref{eqn:core}, \eqref{eqn:drmat} and \eqref{eqn:unioncore}.

\begin{definition} \label{def:transport}
The \emph{continuous transport constant} of a tracial Wasserstein space $(A,W_{\fp,\rho})$ relative to a nucleus $\mathcal{L}(A)$ is the minimal element $k=k_{(A,\rho)}$ of $[1,\infty]$  with the following property: For every pair of faithful traces $\sigma,\tau\in T(A)$ and every $\eps>0$, there is an automorphism $\alpha$ of $A$ that is homotopic to the identity and satisfies
\begin{equation} \label{eqn:transport1}
W_{\infty,\rho}(\sigma\circ\alpha,\tau) \le \eps
\end{equation}
and
\begin{equation} \label{eqn:transport2}
W_{\infty,\rho}(\sigma,\tau) - \eps \le \sup_{f\in\mathcal{L}(A)}\|\alpha(f)-f\| \le k \cdot W_{\infty,\rho}(\sigma,\tau) + \eps.
\end{equation}
\end{definition}

\begin{remark} \label{rem:inequality}
Given \eqref{eqn:transport1}, the first inequality of \eqref{eqn:transport2} is automatic. Indeed, let $\alpha$ be an automorphism of $A$. Since automorphisms preserve the tracial boundary, $\alpha$ induces a homeomorphism $h_\alpha$ of $X\cong \partial_eT(A)$. For every $x\in X$, the function $\rho_x\colon y\mapsto \rho(x,y)$ is in $\lip(X,\rho)=\lip(T(A),W_{\fp,\rho})$ (see Proposition~\ref{prop:inherit} and its surrounding discussion), so by definition \eqref{eqn:L}, any lift $a_x\in A_{sa}$ of $\rho_x$ satisfies $L_\rho(a_x)\le1$. By \eqref{eqn:core} and the assumed compactness of $\mathcal{L}(A)$, there exist $l_x\in\mathcal{L}(A)$ and $\lambda_x\in\rr$ such that $\tau(l_x+\lambda_x)=\tau(a_x)$ for every $\tau\in T(A)$. In particular, writing $\tau_y$ for the extremal trace at $y\in X$,
\begin{align}
\sup_{f\in\mathcal{L}(A)}\|\alpha(f)-f\| &\ge \sup_{x\in X} \|\alpha(l_x)-l_x\| \notag \\
&\ge \sup_{x\in X} |\tau_x(\alpha(l_x)-l_x)| \notag \\
&= \sup_{x\in X} |\tau_x(\alpha(l_x+\lambda_x)-(l_x+\lambda_x))| \notag \\
&= \sup_{x\in X} |\tau_{h_\alpha(x)}(a_x)-\tau_x(a_x)| \notag \\
&= \sup_{x\in X} |\rho(x,h_\alpha(x))-\rho(x,x)| \notag \\
&= \sup_{x\in X} \rho(x,h_\alpha(x)). \label{eqn:transportlb}
\end{align}
This implies that $h_\alpha^{-1}(U)\subseteq U_\delta$ for any $U\subseteq X$, where $\delta=\sup_{f\in\mathcal{L}(A)}\|\alpha(f)-f\|$, and hence by definition \eqref{eqn:winf1} that
\[
W_{\infty,\rho}(\sigma\circ\alpha,\sigma) = W_{\infty,\rho}(\mu_\sigma\circ h_\alpha^{-1},\mu_\sigma) \le \delta.
\]
It then follows from \eqref{eqn:transport1} that
\[
W_{\infty,\rho}(\sigma,\tau) \le W_{\infty,\rho}(\sigma,\sigma\circ\alpha) + W_{\infty,\rho}(\sigma\circ\alpha,\tau) \le \sup_{f\in\mathcal{L}(A)}\|\alpha(f)-f\| + \eps,
\]
which is the required inequality.
\end{remark}

\begin{remark} \label{rem:commutative}
If $A=C(X)$ is commutative, then the inequality \eqref{eqn:transportlb} is in fact an equality. Indeed, for any nucleus $\mathcal{L}(C(X))$ we have
\begin{align*}
\sup_{f\in\mathcal{L}(C(X))}\|\alpha(f)-f\| &= \sup_{f\in\lip(X,\rho)}\|\alpha(f)-f\|\\
&= \sup_{f\in\lip(X,\rho)}\sup_{x\in X}|f(h_\alpha(x))-f(x)|\\
&\le \sup_{x\in X} \rho(x,h_\alpha(x)).
\end{align*}
This in particular implies that the `noncommutative' transport constant $k_{C(X)}$ of $(C(X),\rho)$ (in the sense of Definition~\ref{def:transport}, relative to any nucleus $\mathcal{L}(C(X))$) agrees with the `commutative' transport constant $k_X$ introduced in \cite[Definition 2.6]{Jacelon:2021vc}. Actually, that is not quite true: Definition~\ref{def:transport} is stated for faithful measures while \cite[Definition 2.6]{Jacelon:2021vc} only considers measures that are \emph{good}, that is, faithful as well as diffuse. So, formally, we have $k_{C(X)}\ge k_X$. The difference really becomes substantive when one is interested in an exact transport map rather than an approximate one, and for our present needs is insignificant. Indeed, as observed in \cite[Remark 2.7]{Jacelon:2021vc}, our typical method of transport does not depend on diffuseness (see the proof of Theorem~\ref{thm:transport}), so there is no practical difference between the two notions. Moreover, as good measures are generic in $\prob(X)$, faithful measures can be perturbed to good ones (see the proof of \cite[Theorem 3.1]{Jacelon:2021vc}). A final point in favour of working with faithful traces, as discussed at the beginning of this section, is that these are what actually parameterise embeddings.
\end{remark}

Recall from Definition~\ref{def:dimdrop} that if $p$ and $q$ are natural numbers and $x_0$ and $x_1$ are distinct elements of $X$, then the associated generalised dimension drop algebra is
\[
X_{p,q} = \{f\in C(X,M_p\otimes M_q) \mid f(x_0)\in M_p\otimes1_q,f(x_1)\in 1_p\otimes M_q\}.
\]
Corollary~\ref{cor:wasslyre} records the fact that, for any $\fp\in[1,\infty]$ and $r\ge r_{(X,\rho)}$, $X_{p,q}$ has the structure of a tracial $\fp$-Wasserstein space with canonical nucleus
\begin{equation} \label{eqn:drx}
\mathcal{D}_r(X_{p,q}) = \{f \in (X_{p,q})_{sa} \mid \|f\|\le r,\: \|f(x)-f(y)\|\le\rho(x,y)\:\text{ for every } x,y\in X\}.
\end{equation}
Examples of spaces $X$ that admit optimal continuous transport (that is, satisfy $k_X=1$) are compact, convex subsets of Euclidean space (see \cite[Theorem 2.13]{Jacelon:2021wa}) and compact, connected Riemannian manifolds with the intrinsic metric induced from the smooth structure (see \cite[Theorem 2.8]{Jacelon:2021vc}). These examples possess the property that finite sets of distinct points can be joined by shortest paths that can be locally perturbed to avoid undesirable intersection points and then be encased in small disjoint neighbourhoods homeomorphic to thin Euclidean tubes. A simplicial sphere or ball, or more generally, a Riemannian manifold equipped with an intrinsic metric equivalent to the natural smooth one, also has this property. With these base spaces, optimal continuous transport in the noncommutative sense (Definition~\ref{def:transport}) carries over to the associated generalised dimension drop algebras.

\begin{theorem} \label{thm:transport}
Let $X$ be a compact, connected Riemannian manifold of dimension at least three and let $\rho$ be an intrinsic metric on $X$ that is Lipschitz equivalent to the intrinsic metric induced by the Riemannian structure. Let $X_{p,q}$ be a generalised dimension drop algebra over $X$ for some $p,q\in\nn$ and $x_0\ne x_1\in X$. Then, $k_{(X_{p,q},\rho)}=1$ relative to $\mathcal{D}_r(X_{p,q})$. 
\end{theorem}

\begin{proof}
Let $\sigma,\tau\in T(X_{p,q})$ be faithful traces, corresponding to fully supported measures $\mu_\sigma,\mu_\tau\in\prob(X)$, and let $\eps>0$. As in the proof of \cite[Lemma 2.3]{Jacelon:2021wa}, we can find finitely supported measures $\mu=\sum_{i=1}^n\delta_{x_i}$ and $\nu=\sum_{i=1}^n\delta_{y_i}$ such that
\begin{items}
\item $\max\{W_\infty(\mu,\mu_\sigma),W_\infty(\nu,\mu_\tau)\} < \eps$;
\item $\{x_1,\dots,x_n\}$ and $\{y_1,\dots,y_n\}$ are disjoint subsets of $X\setminus\{x_0,x_1\}$, and each is $\eps$-dense in $X$ with respect to both $\rho$ and the intrinsic Riemannian metric $\rho_{\text{Riem}}$;
\item $\{x_1,\dots,x_n\}$ and $\{y_1,\dots,y_n\}$ are ordered so that
\[
\max_{1\le i\le n}\rho(x_i,y_i) = \min_{\sigma\in S_n} \max_{1\le i\le n}\rho(x_i,y_\sigma(i)) = W_\infty(\mu,\nu).
\]
\end{items}
For each $i\in\{1,\dots,n\}$, let $\gamma_i\colon[0,1]\to X$ be a shortest path from $x_i=\gamma_i(0)$ to $y_i=\gamma_i(1)$. Since $N:=\dim X\ge3$, by locally perturbing at any intersection points we may assume that the images of these paths are pairwise disjoint and do not contain $x_0$ or $x_1$. Let $\eta\in (0,\eps)$ such that
\begin{equation} \label{eqn:eta}
\eta \le \min_{i\ne j}\inf_{s,t\in[0,1]}\min\{\rho(\gamma_i(s),\gamma_j(t)),\rho_{\text{Riem}}(\gamma_i(s),\gamma_j(t))\}.
\end{equation}
For $\delta\in(0,\frac{\eta}{2}]$ and $i\in\{1,\dots,n\}$, let $V_i(\delta)$ be the set of elements of $X$ whose $\rho_{\text{Riem}}$-distance from $\gamma_i([0,1])$ is at most $\delta$. If $\eta$ and $\delta$ are sufficiently small, then these $V_i(\delta)$ do not intersect each other or $\{x_0,x_1\}$, each is contained in the $\rho$-neighbourhood of radius $\frac{\eta}{8}$ around $\gamma_i([0,1])$, and each is homeomorphic to an $N$-dimensional ball. It follows that there are homeomorphisms $h_i\colon V_i(\delta) \to V_i(\delta)$ such that
\begin{items}
\item $h_i(x_i)=y_i$;
\item $\sup_{x\in V_i(\delta)}\rho(h(x),x) \le \rho(x_i,y_i) + \eps$;
\item $h_i$ is homotopic to  $\id_{V_i(\delta)}$;
\item $h_i$ restricts to the identity on the topological boundary of $V_i(\delta)$.
\end{items}
The union of these functions (suitably finessed as discussed after \eqref{eqn:pushforward}) is the desired transport map $h$. Since by construction we have $h(x_0)=x_0$ and $h(x_1)=x_1$, $h$ lifts to an automorphism $\alpha_h$ of $X_{p,q}$ (namely, $f\mapsto f\circ h$). Note that $\alpha_h$ is homotopic to the identity (since $h$ is, via maps fixing $\{x_0,x_1\}$) and satisfies
\begin{align*}
\sup_{f\in\mathcal{D}_r(X_{p,q})} \|\alpha_h(f)-f\| &= \sup_{f\in\mathcal{D}_r(X_{p,q})} \sup_{x\in X} \|f(h(x))-f(x)\|\\
&\le \sup_{x\in X} \rho(h(x),x)\\
&\le \max_{1\le i\le n}\rho(x_i,y_i) + \eps\\
&= W_\infty(\mu,\nu) + \eps.
\end{align*}
To see that it maps $\sigma$ approximately to $\tau$, note that
\begin{align} \label{eqn:pushforward}
W_\infty(\sigma\circ\alpha_h,\tau) &= W_\infty(h_*\mu_\sigma,\mu_\tau) \notag\\
&\le W_\infty(h_*\mu_\sigma,h_*\mu) + W_\infty(h_*\mu, \nu) + W_\infty(\nu,\mu_\tau) \notag\\
&< W_\infty(h_*\mu_\sigma,h_*\mu) + \eps.
\end{align}
As observed in \cite[Remark 2.7]{Jacelon:2021vc}, if we were able to provide a bound $L$ for the Lipschitz seminorm of $h$ (uniformly over $n$ and $\eps$), then we would also have a bound for $W_\infty(\mu_\sigma\circ h^{-1},\mu\circ h^{-1})$. Indeed, for any Borel $U\subseteq X$, if $x\in (h^{-1}(U))_\eps$, then there exists $y\in X$ such that $h(y)\in U$ and $\rho(x,y)<\eps$, hence $x\in h^{-1}(U_{L\eps})$. In other words, $(h^{-1}(U))_\eps \subseteq h^{-1}(U_{L\eps})$. So we would have
\[
h_*\mu_\sigma(U_{L\eps}) = \mu_\sigma(h^{-1}(U_{L\eps})) \ge \mu_\sigma((h^{-1}(U))_\eps) \ge \mu(h^{-1}(U)) = h_*\mu(U)
\]
the second inequality holding because $h^{-1}(U)$ is Borel and $W_\infty(\mu,\mu_\sigma)<\eps$. We would thus deduce that $W_\infty(h_*\mu_\sigma,h_*\mu)\le L\eps$.

This argument does not quite complete the proof (or for that matter, the proof of \cite[Theorem 2.13]{Jacelon:2021wa} or \cite[Theorem 2.8]{Jacelon:2021vc}), because we do not actually have such an $L$. In fact, since $h_i$ fixes the boundary of $V_i(\delta)$ but sends $x_i$ to $y_i$, its Lipschitz seminorm grows to infinity as $\eps$ goes to zero. Luckily, we can hide this problem underneath regions of small measure.

More precisely, since $\mu_\sigma$ is finite, for every $i\in\{1,\dots,n\}$ there can be at most countably many $\delta$ such that the measure of the boundary of $V_i(\delta)$ is nonzero. So, we can assume from the start that $\delta\in(0,\frac{\eta}{16})$ is chosen so that $\mu_\sigma(\partial V) = 0$, where $V:=\bigcup_{i=1}^nV_i(\delta)$. Finiteness of $\mu_\sigma$ also ensures its continuity with respect to decreasing sequences of Borel sets, so there is then $\delta'\in(0,\delta)$ such that the $\mu_\sigma$-measure of $W:=\bigcup_{i=1}^nV_i(\delta)\setminus V_i(\delta')$ is small. This is the small carpet under which we will sweep the bad behaviour of $h$.

Let us specify what we mean by `small'. Set $\gamma=\inf_{x \in X}\mu_\sigma(B_{\frac{\eta}{16}}(x))$. Note that $\gamma>0$. (If not, there would be a convergent sequence $x_k\to x$ such that $\mu_\sigma(B_{\frac{\eta}{16}}(x_k))<\frac{1}{k}$ for every $k\in\nn$. But then $B_{\frac{\eta}{32}}(x)$ would be contained in $B_{\frac{\eta}{16}}(x_k)$ for all sufficiently large $k$, so its measure must be zero, which would contradict faithfulness of $\mu_\sigma$.) We choose $\delta'$ so that $\mu_\sigma(W) < \gamma$.

Note that outside of this region, $h$ is either the identity map on $X \setminus V$, or corresponds under the homeomorphism between $V_i(\delta')$ and an $N$-ball to rotation by $\pi$ about the midpoint of $x_i$ and $y_i$. In the small region $W$, $h$ is defined via a homotopy between these rotation maps and the identity on the boundary of $\bigcup_{i=1}^n V_i(\delta)$. Therefore, we can assume for $x,y\in\bigcup_{i=1}^nV_i(\delta')$ that
\[
\rho(h(x),h(y)) \le \rho(x,y) + \eps
\]
and hence that $(h^{-1}(U))_{\eps} \subseteq h^{-1}(U_{2\eps})$ for any $U \subseteq \bigcup_{i=1}^nV_i(\delta')$.

Now let $U\subseteq X$ be a nonempty open set. The summary of the above discussion is that
\begin{equation} \label{eqn:carpetsweep}
\mu_\sigma(h^{-1}(U_{2\eps})) \ge \mu_\sigma((h^{-1}(U))_{\eps} \setminus W) \ge \mu_\sigma((h^{-1}(U))_{\eps}) - \mu_\sigma(W) > \mu(h^{-1}(U)) - \gamma.
\end{equation}
Suppose that $U$ is large, meaning that $X= U_{3\eps}$. Then,
\begin{equation} \label{eqn:hook}
h_*\mu_\sigma(U_{3\eps}) = \mu_\sigma(h^{-1}(U_{3\eps})) = \mu_\sigma(X) = 1 \ge h_*\mu(U).
\end{equation}
If $U$ is not large, then we claim that we can find an open ball of $\rho$-radius $\frac{\eta}{16}$ disjoint from $V$ and contained in $U_{3\eps} \setminus  U_{2\eps}$. To see this, note that the function $x\mapsto\rho(x,U)$ is continuous and contains $\{0,3\eps\}$ in its range. Since $X$ is connected, there therefore exists $x\in X$ such that $\rho(x,U)=\frac{5\eps}{2}$. Since $\eta<\eps$, the ball $B_\frac{\eta}{16}(x)$ is contained in $U_{3\eps} \setminus U_{2\eps}$. If $B_\frac{\eta}{16}(x)$ does not intersect $V$, then it is the required ball. Otherwise, it intersects $V_i(\delta)$ for some $i$ (unique, by choice \eqref{eqn:eta} of $\eta$). In this case, there exists $y\in\gamma_i([0,1])\cap B_\frac{\eta}{8}(x)$ (because the tube $V_i(\delta)$ is $\frac{\eta}{16}$-thin). Within the neighbourhood $V_i(\frac{\eta}{2})$, we proceed outwards from $y$ along a curve $\gamma$ that is normal to $\gamma_i$ at $y$. For some $t$, $B_\frac{\eta}{16}(\gamma(t))$ is disjoint from $V_i(\delta)$ and is contained in $V_i(\frac{\eta}{8})$, hence is also disjoint from any other $V_j(\delta)$. Noting that $\frac{\eta}{4}+\frac{\eta}{16}=\frac{5\eta}{16}<\frac{\eps}{2}$, and that
\[
\rho(\gamma(t),U) \le \rho(\gamma(t),y) + \rho(y,x) + \rho(x,U) < \frac{\eta}{8} + \frac{\eta}{8} + \frac{5\eps}{2} = \frac{5\eps}{2} + \frac{\eta}{4}
\]
and
\[
\rho(\gamma(t),U) \ge \rho(x,U) -  \rho(y,x) - \rho(\gamma(t),y)  > \frac{5\eps}{2} - \frac{\eta}{8} - \frac{\eta}{8} = \frac{5\eps}{2} - \frac{\eta}{4}
\]
it follows that $B_\frac{\eta}{16}(\gamma(t))$ is contained in $U_{3\eps} \setminus U_{2\eps}$. This proves the claim.

Taking such an $\frac{\eta}{16}$-ball $B$, we have  $\mu_\sigma(B) \ge \gamma$ and $h(B)=B$, so $B\subseteq h^{-1}(U_{3\eps}) \setminus  h^{-1}(U_{2\eps})$. We then have from \eqref{eqn:carpetsweep} that
\begin{equation} \label{eqn:crook}
h_*\mu_\sigma(U_{3\eps}) = \mu_\sigma(h^{-1}(U_{3\eps})) \ge \mu_\sigma(h^{-1}(U_{2\eps})) + \mu_\sigma(B) > h_*\mu(U).
\end{equation}
By hook \eqref{eqn:hook} or by crook \eqref{eqn:crook}, we therefore have that $W_\infty(h_*\mu_\sigma,h_*\mu)\le3\eps$, and hence from \eqref{eqn:pushforward} that $W_\infty(\sigma\circ\alpha,\tau)<4\eps$. We conclude that the transport automorphism $\alpha_h$ does indeed have the required properties to demonstrate that $k_{(X_{p,q},\rho)}=1$.
\end{proof}

\subsection{Distances between unitary orbits} \label{subsection:orbits}

\begin{definition} \label{def:unitarydistance}
Let $B$ be a unital $\cs$-algebra, and let $\mathcal{L}(A)$ be a core for a tracially lyriform $\cs$-algebra $(A,\rho)$. The \emph{unitary distance} in $\emb(A,B)$ relative to $\mathcal{L}(A)$ is defined to be
\begin{equation} \label{eqn:unitary}
d_U(\varphi,\psi)|_{\mathcal{L}(A)} := \adjustlimits \sup_{F\ssubset\mathcal{L}(A)} \inf_{u\in\mathcal{U}(B)} \sup_{a\in F} \|\varphi(a)-u\psi(a)u^*\|.
\end{equation}
The \emph{tracial distance} between $\varphi,\psi\in\emb(A,B)$ is
\begin{equation} \label{eqn:tracialdistance}
\rho(\varphi,\psi) = \sup_{\tau\in T(B)}\rho(\tau\circ\varphi,\tau\circ\psi).
\end{equation}
\end{definition}

Since the core $\mathcal{L}(A)$ in principle only captures tracial information, distance zero is formally weaker than approximate unitary equivalence $\sim_{au}$. So, $d_U(\cdot,\cdot)|_{\mathcal{L}(A)}$ is only a \emph{pseudo}metric on the set $\emb(A,B)/\sim_{au}$ of approximate unitary equivalence classes of embeddings of $A$ into $B$. The same is true of the tracial distance $\rho$. But $d_U(\varphi,\psi)_{\mathcal{L}(A)}=0$ (or $\rho(\varphi,\psi)=0$) does imply that $T(\varphi)=T(\psi)$, so if $B$ is real rank zero and classifiable, and (for simplicity of discussion) $K_*(A)$ is finitely generated and torsion free, then for any fixed $\varphi \in \emb(A,B)$, $d_U(\cdot,\cdot)|_{\mathcal{L}(A)}$ is genuinely a metric on the set of $\sim_{au}$-equivalence classes in
\[
\emb(A,B)_\varphi = \{\psi \in \emb(A,B) \mid K_*(\psi) = K_*(\varphi)\}.
\]
That said, if $A$ is a generalised dimension drop algebra $X_{p,q}$, then its canonical nucleus $\mathcal{D}_r(A)$ has the property that its linear span is dense in $A_{sa}$. So in this case, $d_U(\cdot,\cdot)|_{\mathcal{D}_r(A)}$  is a metric on any $\emb(A,B)/\sim_{au}$ without any assumptions on $B$.

Theorem~\ref{thm:orbits} and its corollaries apply in particular to $B=\js$ or $\mathcal{Q}$, as appropriate. As mentioned in the introduction of this section, if $K_1(A)$ is nontrivial, then because the noncommutative continuous transport constant $k_{(A,\rho)}$ does not keep track of the Hausdorffised algebraic $K_1$ component of the classifying invariant, it is necessary for us to assume that the real rank of $B$ is zero. 

\begin{theorem} \label{thm:orbits}
Let $A=X_{p,q}$ be a generalised dimension drop algebra over a compact, connected metric space $(X,\rho)$ such that $K_*(A)$ is finitely generated and torsion free and $k_{(A,\rho)}<\infty$ via the canonical nucleus $\mathcal{D}_r(A)$. Let $B$ be a simple, separable, unital, nuclear, $\js$-stable $\cs$-algebra. Suppose that either
\begin{enumerate}[(1)]
\item $B$ has a unique trace and $K_1(A)=0$, or
\item $B$ has real rank zero and $\partial_e(T(B))\ne\emptyset$ is compact and of finite Lebesgue covering dimension.
\end{enumerate}
Then,
\begin{equation} \label{eqn:unitaryinequality}
d_U(\varphi,\psi)|_{\mathcal{D}_r(A)} \le k_{(A,\rho)} \cdot W_{\infty,\rho}(\varphi,\psi)
\end{equation}
for every $r\ge r_{(X,\rho)}$ and every pair of unital $^*$-monomorphisms $\varphi,\psi\colon A\to B$ with $K_*(\varphi)=K_*(\psi)$. If we further assume that $X$ is a Riemannian manifold of dimension at least three, and $\rho$ is an intrinsic metric on $X$ that is equivalent to the one induced by the Riemannian structure, then
\begin{equation} \label{eqn:unitaryequality}
d_U(\varphi,\psi)|_{\mathcal{D}_r(A)} = W_{\infty,\rho}(\varphi,\psi)
\end{equation}
for every such $\varphi,\psi\colon A \to B$.
\end{theorem}

\begin{proof}
The inequality \eqref{eqn:unitaryinequality} is proved in exactly the same way as \cite[Theorem 4.11]{Jacelon:2021wa} and its generalisation \cite[Theorem 3.1]{Jacelon:2021vc}, which cover commutative domains $A=C(X)$. (Briefly, classification \cite[Theorem B]{Carrion:wz} lets us convert transport maps $\alpha\colon A\to A$ into conjugating unitaries $u\in B$.) For \eqref{eqn:unitaryequality}, since we know from Theorem~\ref{thm:transport} that $k_{(A,\rho)}=1$ relative to the canonical nucleus $\mathcal{D}_r(A)$ defined in \eqref{eqn:drx}, it then suffices to prove that $d_U(\varphi,\psi)|_{\mathcal{D}_r(A)} \ge W_\infty(\varphi,\psi)$. This can also be deduced from \cite{Jacelon:2021wa,Jacelon:2021vc}, since the embedding $\iota$ of $C(X)=X_{1,1}$ into the centre of $A=X_{p,q}$ restricts to an inclusion of canonical nuclei. It follows that
\[
d_U(\varphi,\psi)|_{\mathcal{D}_r(X_{p,q})} \ge d_U(\varphi\circ\iota,\psi\circ\iota)|_{\mathcal{D}_r(X_{1,1})} \ge W_\infty(\varphi\circ\iota,\psi\circ\iota) = W_\infty(\varphi,\psi). \qedhere
\]
\end{proof}

Recall from \eqref{eqn:unioncore} that our canonical inductive limit model $\js_X$ observing $(X,\rho)$ admits a canonical core $\mathcal{L}_r(\js_X)$ built as the union of modified nuclei \eqref{eqn:modify} of generalised dimension drop algebras. We measure tracial distances \eqref{eqn:tracialdistance} relative to the pro-$W_\infty$ metric \eqref{eqn:prowinfty1}, that is, relative to the metric defined for $\sigma=(\sigma_n)_{n=1}^\infty, \tau=(\tau_n)_{n=1}^\infty \in \varprojlim T(X_{p_n,q_n}) \cong T(\js_X)$ by
\begin{equation} \label{eqn:prowinfty2}
\overrightarrow{W}_{\infty,\rho}((\sigma_n)_{n=1}^\infty,(\tau_n)_{n=1}^\infty)  = \limsup_{n\to\infty} W_{\infty,\rho}(\sigma_n,\tau_n).
\end{equation}
Note that Corollary~\ref{cor:dimdroporbits} applies more generally to domains $\js_X$ observing compact, connected metric spaces $(X,\rho)$ of dimension at least three, assuming that $\js_X$ has the same $K$-theory as $\js_{\ball^k}$ or $\js_{\sphere^{k-1}}$ (as appropriate), and that the metric $\rho$ on $X$ is an intrinsic metric that is Lipschitz equivalent to Riemannian as in Theorem~\ref{thm:orbits}.

\begin{corollary} \label{cor:dimdroporbits}
Let $(\ball^k,\omega_k)$ be a simplicial ball and let $(\sphere^{k-1},\rho_k)$ be an odd simplicial sphere, for some natural number $k\ge3$. Let $B$ be a simple, separable, unital, nuclear, $\js$-stable $\cs$-algebra with a unique trace. Then,
\begin{equation} \label{eqn:orbits1}
d_U(\varphi,\psi)|_{\mathcal{L}_r(\js_{\ball^k})} = \overrightarrow{W}_{\infty,\omega_k}(\varphi,\psi)
\end{equation}
for every $r\ge r_{(\ball^k,\rho_k)}=1$ and every pair of unital $^*$-homomorphisms $\varphi,\psi\colon \js_{\ball^k}\to B$. If $B$ has real rank zero, then \eqref{eqn:orbits1} holds and so does
\begin{equation} \label{eqn:orbits2}
d_U(\varphi,\psi)|_{\mathcal{L}_s(\js_{\sphere^{k-1}})} = \overrightarrow{W}_{\infty,\rho_k}(\varphi,\psi)
\end{equation}
assuming in the latter case that $s\ge r_{(\sphere^{k-1},\rho_k)}=2$ and that $\varphi,\psi\colon \js_{\sphere^{k-1}}\to B$ are unital $^*$-homomorphisms with $K_1(\varphi)=K_1(\psi)$.
\end{corollary}

\begin{proof}
Let $\tau_B$ denote the unique trace on $B$ and let $(X,\rho)=(\ball^k,\omega_k)$ or $(\sphere^{k-1},\rho_k)$, as appropriate. Then by \eqref{eqn:inclusions} and Theorem~\ref{thm:orbits},
\begin{align*}
d_U(\varphi,\psi)|_{\mathcal{L}_r(\js_{X})} &= \sup_{n\in\nn}d_U(\varphi|_{X_{p_n,q_n}},\psi|_{X_{p_n,q_n}})|_{\mathcal{D}^{(n)}_r(X_{p_n,q_n})}\\
&= \limsup_{n\to\infty} d_U(\varphi|_{X_{p_n,q_n}},\psi|_{X_{p_n,q_n}})|_{\mathcal{D}_r(X_{p_n,q_n})}\\
&= \limsup_{n\to\infty} W_{\infty,\rho}((\tau_B\circ\varphi)_n,(\tau_B\circ\psi)_n)\\
&=: \overrightarrow{W}_{\infty,\rho}(\tau_B\circ\varphi,\tau_B\circ\psi)\\
&=: \overrightarrow{W}_{\infty,\rho}(\varphi,\psi).
\end{align*}
Note that, for $X=\ball^k$ or $\sphere^{k-1}$, we automatically have $K_0(\varphi)=K_0(\psi)$ for unital $^*$-homomorphisms $\varphi,\psi\colon\js_X\to B$, because $K_0(\js_{\ball^k})\cong K_0(\js_{\sphere^{k-1}})\cong \zz$ are both generated by the class of the unit (see Corollary~\ref{cor:ballsandspheres}). 
\end{proof}

Finally, we turn to the models $\varinjlim (\js_{\ball^{2k}},W_{1,\omega_{k}})$ and $\varinjlim (\js_{\sphere^{2k-1}},W_{1,\rho_{k}})$ of Theorem~\ref{thm:arachnid}. We measure unitary distances \eqref{def:unitarydistance} relative to the cores $\mathfrak{ball}(A)$ and $\mathfrak{sphere}(A)$ described in Definition~\ref{def:simplicial} and tracial distances \eqref{eqn:tracialdistance} relative to the (pro-)pro-$W_\infty$ metrics defined on $\varprojlim T(\js_{\ball^{2k}})$ and $\varprojlim T(\js_{\sphere^{2k-1}})$ by
\begin{equation} \label{eqn:omegapro}
\overrightarrow{W}_{\infty,\omega}((\sigma_k)_{k=1}^\infty,(\tau_k)_{k=1}^\infty) = \limsup_{k\to\infty} \overrightarrow{W}_{\infty,\omega_k}(\sigma_k,\tau_k) = \sup_{k\in\nn} \overrightarrow{W}_{\infty,\omega_k}(\sigma_k,\tau_k)
\end{equation}
and
\begin{equation} \label{eqn:rhopro}
\overrightarrow{W}_{\infty,\rho}((\sigma_k)_{k=1}^\infty,(\tau_k)_{k=1}^\infty) = \limsup_{k\to\infty} \overrightarrow{W}_{\infty,\rho_k}(\sigma_k,\tau_k) = \sup_{k\in\nn} \overrightarrow{W}_{\infty,\rho_k}(\sigma_k,\tau_k).
\end{equation}

\begin{corollary} \label{cor:spiderorbits}
Suppose that $A$ is a simple $\cs$-algebra that is isomorphic to an inductive limit of prime dimension drop algebras, constructed as in Theorem~\ref{thm:arachnid}. Let $B$ be a simple, separable, unital, nuclear, $\js$-stable $\cs$-algebra with a unique trace, and let $\varphi,\psi\colon A\to B$ be unital $^*$-homomorphisms. Then,
\[
d_U(\varphi,\psi)|_{\mathfrak{ball}(A)} = \overrightarrow{W}_{\infty,\omega}(\varphi,\psi).
\]
If in addition $B$ has real rank zero, then
\[
d_U(\varphi,\psi)|_{\mathfrak{sphere}(A)} = \overrightarrow{W}_{\infty,\rho}(\varphi,\psi).
\]
\end{corollary}

\begin{proof}
As in the proof of Corollary~\ref{cor:dimdroporbits}, we have
\[
d_U(\varphi,\psi)|_{\mathfrak{ball}(A)} = \sup_{k\in\nn} d_U(\varphi,\psi)|_{\mathcal{L}_r(\js_{\ball^{2k}})} = \sup_{k\in\nn} \overrightarrow{W}_{\infty,\omega_k}(\varphi,\psi) = \overrightarrow{W}_{\infty,\omega}(\varphi,\psi)
\]
and similarly for $\mathfrak{sphere}(A)$, noting that $K_1(\varphi)=K_1(\psi)=0$ since $K_1(A)=0$.
\end{proof}

By design, the pro-$W_\infty$ topologies appearing in Corollary~\ref{cor:dimdroporbits} and Corollary~\ref{cor:spiderorbits} are very strong. A minute perturbation of the trace profile of an embedding can result in a relatively large unitary-orbit displacement. Unpacking the definitions, the heart of the matter is that a \emph{local} geometric picture is provided by a generalised dimension drop algebra viewed as a tracial $\infty$-Wasserstein space over a ball or a sphere. Topologically, this local picture is of course already provided by the $\cs$-algebraic inductive limit structure. But our focus is on metric geometry, not topology. 

\subsection{Random embeddings} \label{subsection:random}

Let us return to the situation described at the beginning of this section: $B$ is a classifiable $\cs$-algebra that has a unique trace $\tau_B$, real rank zero and trivial $K_1$, and $(X,\rho)$ is a suitable compact length space for which the metric space $(\emb(\js_{X},B)/\sim_{au},d_U|_{\mathcal{L}_r(\js_{X}})$ is isometric to $(\prob(X),\overrightarrow{W}_{\infty,\rho})$ as in Corollary~\ref{cor:dimdroporbits}.

\begin{question} \label{q:random}
Suppose that we want to choose an embedding $\varphi$ (or rather, an equivalence class $[\varphi]\in\emb(\js_{X},B)/\sim_{au}$) \emph{at random}.
\begin{enumerate}[(1)]
\item How might we do this?
\item What can we say about the distribution of random embeddings?
\end{enumerate}
\end{question}
We will consider two ways of addressing the first question, both based on randomly choosing elements of $X$ via a fixed probability measure $\mu\in\prob(X)$ in order to produce random measures $\nu$ and associated embeddings $\varphi_{\nu}$. In both cases, faithfulness of $\mu$ will maximise the collection of (non-null) possible outcomes. Natural candidates for $\mu$ might be Haar measure on a sphere, normalised Lebesgue measure on a ball, or more generally, normalised volume measure on a Riemannian manifold. In answer to the second question, summarised in Theorem~\ref{thm:e}, we will in both cases observe tight clustering around the mean in the form of estimates of large deviation. Core elements $f\in\mathcal{L}(\js_X):=\mathcal{L}_{r_{(X,\rho)}}(\js_X)$ will play the role of observables. As usual, $\widehat f$ denotes the evaluation function $\tau\mapsto\widehat f(\tau) = \tau(f)$.

As described at the end of Section~\ref{subsection:orbits}, the $\overrightarrow{W}_{\infty,\rho}$ metric is very sensitive. That means in the current context that estimates of $\overrightarrow{W}_{\infty,\rho}$-large deviation will be hard to come by. On the other hand, cruder estimates involving $W_{1,\rho}$ are available from the asymptotic geometry of our spaces and sampling processes. That is what we present below.

\subsubsection{Random extremal embeddings} \label{subsubsection:extremal}

We are interested in the structure of the \emph{extremal embedding space}
\[
\emb_e(\js_X,B) = \{\varphi\in\emb(\js_X,B) \mid \tau_B\circ\varphi\in\partial_eT(\js_X)\}
\]
or rather the structure of $\emb_e(\js_X,B)/\sim_{au}$. We use the probability measure $\mu$ to randomly select an element $x \in X$. In this way, we produce (the $\sim_{au}$-class of) a random extremal embedding $\varphi_{x}\colon \js_X \to B$. Geometric properties of the metric measure space $(X,\rho,\mu)$ allow us to describe the central tendency of $\varphi_{x}(f)$ for $f\in\mathcal{L}(\js_X)$. The overview is that, for highly curved spaces $X$ like high-dimensional spheres, $\{\tau_B\circ\varphi_{x}(f) \mid x\in X\}$ is highly concentrated around the mean trace $\tau_B\circ\varphi_\mu(f) = \int_X\widehat f\,d\mu=:\ee_\mu(\widehat f)$.

For odd unit spheres $X=\sphere^{k-1}$ (with the round metric $\rho$ and Haar measure $\mu$), this can be seen as a consequence of the concentration of measure phenomenon. It follows from the spherical isoperimetric inequality (see, for example, \cite[Chapter 2]{Milman:1986ux}) that there are constants $c_1,c_2>0$ such that, for every $1$-Lipschitz function $h\colon\sphere^{k-1}\to\rr$ and every $\eps>0$,
\begin{equation} \label{eqn:median}
\mu(\{x\in \sphere^{k-1} \mid |h(x)-M_{h}| > \eps\}) \le c_1\exp(-c_2k\eps^2).
\end{equation}
Here, $M_h$ denotes the median of $h$, which by definition means that it satisfies $\mu(\{x \mid h(x)\ge M_h\}) \ge \frac{1}{2}$ and $\mu(\{x \mid h(x)\le M_h\}) \ge \frac{1}{2}$. Since the range of $h$ is contained in an interval of length at most $\diam(\sphere^{k-1},\rho)=\pi$, we can bound the uniform norm of $|h-M_h|$ by $\pi$, and then deduce from \eqref{eqn:median} that
\[
\left| \int_X h\,d\mu - M_h \right| \le \int_X |h - M_h|\,d\mu \le \eps + c_1\pi\exp(-c_2k\eps^2).
\]
It follows that we can find constants $c_1,c_2$ such that, for every $\eps>0$, \eqref{eqn:median} holds and so does
\[
\mu(\{x\in \sphere^{k-1} \mid |h(x)-\ee_\mu h| > \eps\}) \le c_1\exp(-c_2k\eps^2)
\]
for all sufficiently large $k$. In other words, for any given $\eps>0$ we have
\begin{equation} \label{eqn:mean}
\sup_{f\in\mathcal{L}(\js_{\sphere^{k-1}})}\mu(\{x\in \sphere^{k-1} \mid |\tau_B\circ\varphi_{x}(f)  - \tau_B\circ\varphi_\mu(f)| > \eps\}) \le c_1\exp(-c_2k\eps^2)
\end{equation}
for every sufficiently large $k$, where $\mathcal{L}(\js_X)$ is simplified notation for the core $\mathcal{L}_{r_{(X,\rho)}}(\js_{X})$.

For spaces other than spheres, we can bound the variance of $\tau_B\circ\varphi_{x}(f)=\widehat f(x)$ when the measured length space $(X,\rho,\mu)$ has bounded Ricci curvature in the sense of \cite{Lott:2009aa} or \cite{Sturm:2006aa,Sturm:2006ab}. More precisely, the Poincar\'{e} inequality \cite[Corollary 0.19]{Lott:2009aa} says that if the $\infty$-Ricci curvature of $(X,\rho,\mu)$ is bounded below by $K>0$, then for every $h\in\lip(X)$,
\[
\int_X (h-\ee_\mu h)^2\,d\mu \le \frac{1}{K}\int_X|\nabla h|^2\,d\mu
\]
where $|\nabla h|(x):=\limsup_{y\to x}\frac{|h(x)-h(y)|}{\rho(x,y)}$. Consequently, the tracial variance of the core is bounded by
\begin{equation} \label{eqn:variance}
\sup_{f\in\mathcal{L}(\js_{X})}\Var\tau_B\circ\varphi_x(f) \le \frac{1}{K}.
\end{equation}
Note that, for Riemannian manifolds, the Lott--Villani--Sturm curvature bound can be rephrased in classical terms via the Bakry--\'Emery tensor, a metric-measure adjustment of the Ricci tensor that accounts for a smooth volume measure $\mu=e^{-\Psi}d\mathrm{vol}_X$ (see \cite{Bakry:1985aa,Lott:2003aa}). If $\Psi$ is constant, the bound is equivalent to $\mathrm{Ric}\ge Kg$, where $g$ is the Riemannian metric tensor. In particular, if $X=\sphere^{k-1}$ is an odd unit sphere (again with the round metric $\rho$ and Haar measure $\mu$), then $\mathrm{Ric}\ge (k-2)g$, whence
\begin{equation} \label{eqn:spherevariance}
\sup_{f\in\mathcal{L}(\js_{\sphere^{k-1}})}\Var\tau_B\circ\varphi_x(f) \le \frac{1}{k-2}.
\end{equation}

\subsubsection{Empirical embeddings} \label{subsubsection:empirical}
This time, instead of choosing a single random element $x \in X$, we take a large sample.  Let us write $\pp_\mu$ for the law on the probability space associated with an infinite sequence $(x_n)_{n\in\nn}$ of independent random variables taking values in $X$ and identically distributed according to $\mu$. For each $n\in\nn$, we take the empirical measure $\frac{1}{n}\sum_{i=1}^n\delta_{x_i}$ as our random measure $\mu_n$, and $[\varphi_{\mu_n}]$ (the class of embeddings $\varphi$ with $\mu_{\tau_B\circ\varphi}=\mu_n$) as our random embedding.

For the asymptotic behaviour of these embeddings, we appeal to Sanov's theorem for Polish spaces (see \cite[Theorem 6.2.10 and Exercise 6.2.18]{Dembo:2010aa}). This remarkable result in particular implies that, for every $\ws$-closed subset $V\subseteq\prob(X)$,
\begin{equation} \label{eqn:ldp}
\limsup_{n\to\infty} \frac{1}{n}\log\pp_\mu(\mu_n\in V) \le -\inf_{\nu\in V} H(\nu|\mu).
\end{equation}
Here, $H(\nu|\mu)$ is the \emph{relative entropy of $\nu$ relative to $\mu$}, also known as the \emph{Kullback--Leibler divergence}, defined by
\[
H(\nu|\mu) =
\begin{cases}
\int_X f \log f d\mu & \text{ if $f:=\frac{d\nu}{d\mu}$ exists}\\
\infty & \text{otherwise.}
\end{cases}
\]
Relative entropy is not symmetric, and not even its symmetrisation satisfies the triangle inequality, so it is not a metric. But what matters for us is that it is related to the Wasserstein distance $W_1=W_{1,\rho}$ by
\begin{equation} \label{eqn:wassentropyinequality}
H(\nu|\mu) \ge 2\left(\frac{W_1(\nu,\mu)}{\diam(X,\rho)}\right)^2
\end{equation}
(see \cite[Figure 1]{Gibbs:2002aa}). Given $\eps>0$, let $V=V_\eps$ be the $\ws$-closed set $\{\nu\in\prob(X) \mid W_1(\nu,\mu)\ge\eps\}$. From \eqref{eqn:ldp} and \eqref{eqn:wassentropyinequality} we get
\[
\limsup_{n\to\infty} \frac{1}{n}\log\pp_\mu(W_1(\mu_n,\mu)\ge\eps) \le - \inf_{\nu\in V_\eps} 2\left(\frac{W_1(\nu,\mu)}{\diam(X,\rho)}\right)^2 \le -\frac{2\eps^2}{\diam(X,\rho)^2}.
\]
We conclude that, if $n$ is sufficiently large, then
\begin{equation} \label{eqn:dusanov}
\pp_\mu\left(\sup_{f\in\mathcal{L}(\js_X)}|\tau_B\circ\varphi_{\mu_n}(f)  - \tau_B\circ\varphi_\mu(f)| \le \eps\right) \ge 1 - \exp\left(-\frac{n\eps^2}{\diam(X,\rho)^2}\right).
\end{equation}

\subsection*{Funding} The author was supported by the Czech Science Foundation (GA\v{C}R) project 25-15403K and the Institute of Mathematics of the Czech Academy of Sciences (RVO: 67985840). 

\subsection*{Competing interests} The author has none to declare.

\end{document}